\documentclass[12pt,oneside,reqno]{amsart}
\usepackage{inputenc}
\usepackage{amsmath,amssymb,latexsym,natbib,amsthm,amsfonts}
\usepackage{fullpage}
\usepackage{color}
\usepackage{graphicx}
\usepackage{microtype}
\usepackage[colorlinks=true,allcolors=blue]{hyperref}
\usepackage{enumitem}

\usepackage{tikz}
\usetikzlibrary{backgrounds}
\usetikzlibrary{shapes.multipart}
\usetikzlibrary{
  graphs,
  graphs.standard
}
\usetikzlibrary{fit}
\usepackage{subfigure}
\usepackage[T1]{fontenc}

\newcommand{\semantics}[1]{[\![\mbox{\em $ #1 $\/}]\!]}

\usepackage{nicefrac}
\usepackage{verbatim}
\usepackage{hhline}

\usepackage{thmtools}
\usepackage{thm-restate}

\newtheorem{fact}{Fact}
\newtheorem{conjecture}{Conjecture}
\newtheorem{proposition}{Proposition}
\newtheorem{corollary}{Corollary}
\newtheorem{lemma}{Lemma}
\newtheorem{theorem}{Theorem}
\theoremstyle{definition}
\newtheorem{definition}{Definition}
\newtheorem{example}{Example}
\newtheorem{remark}{Remark}

\usepackage{setspace}

\title{Interleaving Logic and Counting}

\author{Johan van Benthem and Thomas Icard}

\date{\today. This is a corrected version of \href{https://www.cambridge.org/core/journals/bulletin-of-symbolic-logic/article/interleaving-logic-and-counting/E8C36FFD28218F157BE9D150D7B55AA6}{`Interleaving Logic and Counting'}, an article that appeared in the \emph{Bulletin of Symbolic Logic}, 29:4, pp. 503--587, 2023.}

\begin{document}

\maketitle 

\begin{abstract} Reasoning with quantifier expressions in natural language combines logical and arithmetical features, transcending strict divides between qualitative and quantitative. Our topic is this cooperation of styles as it occurs in common linguistic usage and its extension into the broader practice of natural language plus ``grassroots mathematics''. 

We begin with a brief review of $\mathsf{FO}(\#)$, first-order logic with counting operators and cardinality comparisons. This system is known to be of very high complexity, and drowns out finer aspects of the combination of logic and counting. We therefore move to a small fragment that can represent numerical syllogisms and basic reasoning about comparative size: monadic first-order logic with counting, $\mathsf{MFO}(\#)$. We provide normal forms that allow for axiomatization, determine which arithmetical notions can be defined on finite and on infinite models, and conversely, we discuss which logical notions can be defined out of purely arithmetical ones, and what sort of (non-)classical logics can be induced.

Next, we investigate a series of strengthenings of $\mathsf{MFO}(\#)$, again using  normal form methods. The monadic second-order version is close, in a precise sense, to additive Presburger Arithmetic, while versions with the natural device of tuple counting take us to Diophantine equations, making the logic undecidable. We also define a system $\mathsf{ML}(\#)$ that combines basic modal logic over binary accessibility relations with counting, needed to formulate ubiquitous reasoning patterns such as the Pigeonhole Principle. We prove decidability of $\mathsf{ML}(\#)$, and provide a new kind of bisimulation matching the expressive power of the  language. 

As a complement to the fragment approach pursued here, we also discuss two other ways of lowering the complexity of $\mathsf{FO}(\#)$ by changing the semantics of counting in natural ways. A first approach  replaces cardinalities by abstract but well-motivated values of ``mass'' or other mereological aggregating notions. A second approach keeps the cardinalities but generalizes the meaning of counting to work in models that allow dependencies between variables.
  
Finally, we return to our starting point in natural language, confronting the architecture of our formal systems with linguistic quantifier vocabulary and syntax, as well as with natural reasoning modules such as the monotonicity calculus. In addition to these encounters with  formal semantics, we discuss the role of counting in semantic evaluation procedures for quantifier expressions and determine, for instance, which binary quantifiers are computable by finite ``semantic automata''. We conclude with some general thoughts on yet further entanglements of logic and counting in formal systems, on rethinking the qualitative/quantitative divide, and on connecting our analysis to empirical findings in  cognitive science.
\end{abstract} 

\newpage

\tableofcontents

\newpage

\section{Introduction: Inference and Computing}

Here is the archetypal logical inference with a basic quantifier: 

\begin{quote} From `All $A$ are $B$' and `All $B$ are $C$', conclude that `All $A$ are $C$'. \end{quote}
Next, here are two slightly modified premises in natural language. 

\begin{quote} `All $A$ \emph{except one} are $B$ and all $B$ \emph{except two} are $C$'. \end{quote}
This time, one may need to think just a little bit more to conclude that \begin{quote} `All $A$ \emph{except at most three} are $C$'. \end{quote}
That extra bit of thought involves considering possible exceptions, or more generally: \emph{counting}. In fact, the very term \emph{quantifier} suggests quantities, and the semantics of quantifier expressions in logic and linguistics  involves numbers by its emphasis on  permutation invariance, which abstracts away from every feature of predicates except their size. This mix of logic and counting is not just about absolute numbers, it also extends to size comparisons. From \begin{quote}`Most $A$ are $B$' and `All $B$ are $C$',\end{quote} we may safely draw the conclusion that  \begin{quote} `Most $A$ are $C$' \end{quote} and similar simple inference patterns govern explicitly comparative expressions such as `More $A$ than $B$ are $C$'. But valid reasoning patterns with comparatives can also be more challenging, as in the following  inference, which may require drawing a Venn diagram:
\vspace{0.5ex}
\begin{quote}
`More $A$ than $B$ are $C$', `More $B$ than $C$ are $A$',  \vspace{0.5ex}  \\
 Therefore: `More $A$ than $C$ are $B$'.
\vspace{0.5ex}
    \end{quote}

\noindent This has echoes of the mathematical \emph{Triangle Inequality} underlying metric geometry.

Numerical comparisons in natural language can even occur between \emph{proportions}, as happens in the relative sense of `Many $A$ are $B$', comparing the numbers of $B$s among the $A$s with the number of $B$s overall, defined more precisely in \S\ref{section:basicmonadic}, and a running example later on.

%
%
%
\vspace{1ex}

Qualitative logical analyses are sometimes  seen as replacing quantitative theories by ``more basic'' qualitative ones, for instance, in   the foundations of probability or in measurement theory. This can be illuminating, and  success  can be measured by  representation theorems.
And yet, historically, logic and quantitative reasoning, for instance with probability,  went together  in the pioneering work of  Bolzano and Boole. It is hard to say whether Boole's propositional logic is a qualitative basic form of binary arithmetic, or a way of  making logical inference a form of counting. In a sense, it is both. A  divide arose only in the time of Frege, when logicism insisted that logical notions come first, and arithmetical ones are constructed  out of these. To be sure, this  reductionist program has yielded many fundamental notions and results, and we owe a lot of modern logic to its arrival. But in this paper, we will follow the linguistic practice that we started with, and treat logic and \emph{counting}, taken as the realm of numerical comparisons and basic arithmetic, on a par. 

In what follows we will take this linguistic practice in a broad sense, including ubiquitous forms of reasoning that might be called ``grassroots mathematics'' rather than pure natural language inference. A typical example underlies the  following pattern:
\vspace{0.5ex}
\begin{itemize}
\item[] `Twenty farmers own at most 15 cows each'. Therefore: 
\item[] `At least two farmers own the same number of cows'.
\end{itemize}
\vspace{0.5ex}
The reader may find it difficult to see how this would follow as a straightforward matter of overt logical or linguistic form. Instead, what is needed is the following
\vspace{-1ex}
\begin{itemize}
    \item[]\emph{Pigeonhole Principle} \quad If one puts $n$ objects into $k$ boxes, with $n > k$, then  at least one box must contain at least two objects.
\end{itemize}
\vspace{-1ex}
Here $k$ is the number of cows owned, which runs from 0 to 15,  $n$ the number of farmers. The Pigeonhole Principle occurs in elementary mathematics where it can have non-trivial consequences when applied imaginatively, but it is also of interest in cognitive science as a benchmark in reasoning ability including finding the right representation of problems \citep{Sperber}. In this paper, the  principle 
will occur at various places  as we determine its position in combined systems of logic and counting. 

\vspace{1ex}

Where should we start with our investigation of logic and counting?  It is well-known that combining a standard system like first-order logic with counting syntax and cardinality comparisons leads to a system $\mathsf{FO}(\#)$ of very high complexity. Therefore, for our purpose, this ``view from above'' is not that illuminating, and after just a quick look at $\mathsf{FO}(\#)$ and its properties, we will start work ``from below'', exploring very simple combinations of logic and counting, and only then move to more complex systems.

\vspace{1ex}

Our presentation follows mainstream  practice in offering a sequence of formal systems of increasing expressive strength. We will prove many results about these systems that demonstrate their precise mixture of logic and counting. Toward the end of the paper, we return to the naturally occurring practice of mixed qualitative and quantitative reasoning that we started with here, linking up with Generalized Quantifier Theory for natural language, and touching on empirical issues in cognitive science. Finally, in a sequence of appendices, we broaden the context, and point out yet further entanglements of logic and counting that show the ubiquity of the phenomenon we are after. True understanding of how logical systems work involves numbers and counting from manipulating syntax to proofs by formula induction, but also semantically, e.g., in the use of numerical invariants in Ehrenfeucht games. 

\vspace{1ex}

There are several ways of looking at the topics and results presented in this paper. Simple combinations of logic and counting can often be seen as fragments of richer \emph{logics of generalized quantifiers} \citep{BarwiseFeferman,PetersWesterstahl}. In this sense, we are looking at  fine-structure of fragments of well-known systems from mathematical logic. Moreover, the interplay of logic and counting  has long been studied in  \emph{computational logic} \citep{Otto,Schweikardt}. Accordingly, themes and results from the literature in theoretical computer science will appear at many places in this paper. We have added an appendix with references to a wide, and hopefully representative, swath of the preceding literature, though a full overview is beyond our capacity.

\vspace{1ex}

Against this background, the technical main novelty of this paper is the series of simple combined systems that we define and study. However, a further contribution may be the more empirical perspective we are adding of \emph{connections with natural language and cognition}. In addition to our technical results about logic and counting, we see this stance in between logic, computation and cognition, as fruitful and worth pursuing. 

\vspace{1ex}

In the next section, we first present a higher-end combination of logic and counting, as a first pass through our main themes. After that, we give more detail on the lower-end systems that will be the focus of our analysis in the core of the paper.


\section{First-Order Logic with Counting} \label{section:first-order}
Perhaps the obvious starting point is to consider a counting operator $\#$ on top of standard first-order logic, allowing us to count the number of objects satisfying a given formula. Where $x$ is a first-order variable and $\varphi$ a first-order formula, in a first-order model $\mathcal{M}$ with variable assignment $s$, the term $\#_x\varphi$ denotes the cardinality of the set of $x$'s satisfying $\varphi$: \begin{eqnarray*} \semantics{\#_x\varphi}^{\mathcal{M},s} & = &  |\{d \in D: \mathcal{M},s^x_d \models \varphi\}| \end{eqnarray*} 
Count terms thus denote cardinal numbers. What kinds of assertions would we want to make about cardinal numbers to formalize interesting reasoning about counting? Here we start with a basic and fundamental capacity, namely \emph{comparison}. We inductively define count comparison formulas $\#_x\varphi \succsim \#_y\psi$, with the obvious interpretation according to which: \begin{eqnarray*} \mathcal{M},s \vDash \#_x\varphi \succsim \#_y\psi & \mbox{ iff }  & \semantics{\#_x\varphi}^{\mathcal{M},s} \geq \semantics{\#_y\psi}^{\mathcal{M},s}. \end{eqnarray*}
Call this language $\mathcal{L}_\#$ and call the logical system $\mathsf{FO}(\#)$.

This system has been studied thoroughly. It is natural to construe $\mathsf{FO}(\#)$ as first-order logic with a generalized quantifier, sometimes known in the literature as the \emph{Rescher quantifier} \citep{hartig,Otto} after a related extension considered in \cite{rescher}; in the philosophical literature it has sometimes been called the \emph{Frege quantifier} \citep{antonelli2010}. Other well known quantifiers, such as the so called \emph{H\"{a}rtig quantifier} and the \emph{Chang quantifier}, are easily definable in $\mathsf{FO}(\#)$ (see, e.g., \citealt{PetersWesterstahl}). It will be convenient to abbreviate  the H\"{a}rtig quantifier $(\#_x\varphi \succsim \#_y\psi) \wedge (\#_y\psi \succsim \#_x\varphi)$ by $\#_x\varphi \approx \#_y \psi$; likewise, we abbreviate $(\#_x\varphi \succsim \#_y\psi) \wedge \neg (\#_y\psi \succsim \#_x\varphi)$ by $\#_x\varphi \succ \#_y \psi$.

Typical of extensions of first-order logic,  we have the following:
\begin{proposition} \label{prop:compactness}
$\mathsf{FO}(\#)$ fails to be compact and it lacks the L\"{o}wenheim-Sk\o lem property down to any cardinality below $\aleph_{\omega}$.
\end{proposition}
\begin{proof} First, note that the infinity quantifier is easily definable in $\mathsf{FO}(\#)$: \begin{eqnarray}
  \exists^\infty y.\varphi & \equiv & \exists y. \big(\varphi \wedge \#_x(\varphi[\nicefrac{x}{y}] \wedge x\neq y) \succsim \#_y(\varphi) \big),  \label{infinity-quantifier}
\end{eqnarray} 
where the substitution $\varphi[\nicefrac{x}{y}]$ is defined as usual. 
Then we can force the domain to have size at least $\aleph_k$ simply by stating, for instance, $\exists^\infty x. P_0(x) \wedge \bigwedge_{i \leq k}\big(\#_x P_{i+1}(x) \succ \#_x P_i(x)\big)$, for $k+1$ predicate symbols $P_0,\dots,P_k$.

Compactness also fails easily: abbreviating $\bigwedge_{i,j \leq n} x_i\neq x_j$ by $\mathsf{diff}(\mathbf{x})$, and using $\exists^{\geq n}x.P(x)$ to abbreviate $\exists x_1\dots x_{n}.\big(\mathsf{diff}(\mathbf{x}) \wedge \bigwedge_{i \leq n} P(x_i) \big)$, 
the set \begin{equation} \{\neg \exists^\infty x. P(x)\} \cup \{ \exists^{\geq n}x.P(x): n < \omega \}
    \label{incompact}
\end{equation} is unsatisfiable, but finitely satisfiable. \end{proof}

To see just how much stronger $\mathsf{FO}(\#)$ is than ordinary $\mathsf{FO}$, note the following:
\begin{fact} We can enforce in $\mathsf{FO}(\#)$ that a binary relation $R$ is a well-order of order type $\omega$. \label{ordertype}
\end{fact}
\begin{proof} Let $\sigma$ be the statement that $R$ is a serial, strict total order (i.e., serial, irreflexive, transitive, total), and conjoin $\sigma$ with the statement $\forall x. \neg \exists^{\infty} y. R(y,x)$, saying that each element has only finitely many $R$-predecessors. 
\end{proof}

It follows that the validity problem for $\mathsf{FO}(\#)$ is not arithmetical; in fact it is $\Pi^1_1$-hard. If we do not allow embedding $\#$ comparisons, then we can also show that the satisfiability problem is in $\Sigma^1_1$: every comparison amounts to the existence of an injective function.
\begin{fact} The set of validities of $\mathsf{FO}(\#)$ without embedded $\#$ terms is $\Pi^1_1$-complete.
\end{fact}

However, for the general case the situation is much worse. \cite{hartig} showed the following result for first-order logic with the H\"{a}rtig quantifier:\footnote{In fact, the hardness result in Theorem \ref{pi12} holds even if we take the \emph{negationless} fragment of $\mathsf{FO}(\#)$. \cite{Mayer} has given a computable reduction from the negationless fragment to the full fragment with negation.} 
\begin{theorem}[\citealt{hartig}] The set of validities of $\mathsf{FO}(\#)$ is neither in $\Pi^1_2$ nor in $\Sigma^1_2$. \label{pi12}
\end{theorem}

$\mathsf{FO}(\#)$ clearly brings a potent combination of logical expressive power and explicit count comparison.  To what degree can we tease apart the separate contributions of  logic and counting in this rich setting? Specifically, how much do $\#$ comparisons add to the counting repertoire native to first-order logic; and vice versa, how much logic could we already extract from counting alone? We begin with the second question. 

\subsection{From Counting to Logic} \label{section:countinglogic}
Let us restrict attention to a very small fragment of the language $\mathcal{L}_\#$ described above. Given some variables $\mathsf{Var}$ and predicate symbols $\mathsf{Pred}$, we only allow two types of atomic formulas and one operation for building complex formulas. Let $\mathcal{L}_\#^-$ be generated by the grammar: \begin{eqnarray*}
\varphi \quad::=\quad\quad   P(x_1,\dots,x_n) \quad\mid\quad x\neq y \quad\mid\quad \#_x\varphi \succsim \#_y\varphi
\end{eqnarray*} Aside from predication and variable inequality, we can only compare cardinalities.

A first observation is that Boolean implication can already be defined in $\mathcal{L}_\#^-$. Where $x$ occurs free in neither $\varphi$ nor $\psi$, we can take: 
\begin{eqnarray}
 \psi \rightarrow \varphi & \equiv & \#_x\varphi \succsim \#_x\psi. \label{implication}
\end{eqnarray}
Boolean negation can also be defined. Where $\mathsf{0}$ is an abbreviation for $\#_x (x \neq x)$ (cf. Frege), and again $x$ is a variable that does not occur free in $\varphi$, we can define: \begin{eqnarray}
\neg \varphi & \equiv & \mathsf{0} \succsim \#_x\varphi. \label{negation}
\end{eqnarray} With these we recover any other Boolean connective, as well as variable equality. In some respect, count comparison already incorporates Boolean structure, and familiar Boolean laws emerge as principles of count comparisons. For instance, the pattern $\varphi \rightarrow (\psi \rightarrow \varphi)$ is encoded simply as $\#_x(\#_x\varphi \succsim \#_x\psi) \succsim \#_x\varphi$. 

Going further, first-order quantification is expressible in $\mathcal{L}_\#^-$:
\begin{eqnarray}
\exists x. \varphi & \equiv & \#_x\varphi \succ \mathsf{0}. \label{existential}
\end{eqnarray} This thus brings us back to full $\mathsf{FO}(\#)$, in which we can again define the infinity quantifier $\exists^\infty$ in (\ref{infinity-quantifier}), its dual $\forall^\infty$, and so on. From rather austere (atomic) primitives, count comparisons already encode a significant amount of logic, provided of course that we allow iteration of comparisons within comparisons.

\begin{remark}[Extended logical vocabulary] Counting can also define non-first-order quantifiers that are often considered logical in an extended sense. An example is the binary quantifier `Most $\varphi$ are $\psi$', which is definable as $\#_x(\varphi \land \psi) \succ \#_x(\varphi \land \neg\psi)$. But even closer to first-order logic, counting suggests different kinds of universal quantifier, depending on how we extend the standard meaning on finite sets to infinite ones. One option is $\neg \exists x. \neg\varphi$, the dual of the existential quantifier defined in (\ref{existential}), which expresses exceptionless  universal quantification. But there are also interesting weaker variants, such as  $\#_x\varphi \approx \#_x\top \wedge \#_x\varphi\succ \#_x\neg\varphi$. This says that the set of objects satisfying $\varphi$ has the size of the universe, while the possible exceptions have a smaller size. This is a  version of the quantifier `almost all' which has elegant mathematical properties and interesting measure-theoretic applications \citep{STEINHORN1985161}.
\end{remark}

\begin{remark}[Non-classical Logics] In addition to options qua expressive power, counting also offers options for deductive power. The definitions (\ref{implication}), (\ref{negation}), and (\ref{existential}) above show that we can reconstruct classical logic from $\#$ comparisons. Is $\mathsf{FO}(\#)$ in some way inherently classical, or could we instead naturally extract \emph{non-classical} connectives?

One route would be to keep the same implication in (\ref{implication}), but to redefine negation in terms of an arbitrary predicate, say, $G(x)$. If we then let $\neg \varphi$ stand for the sentence $\#_xG(x) \succsim \#_x\varphi$, where again $x$ is not free in $\varphi$, we lose one direction of the law of double negation, namely, $\neg \neg \varphi \rightarrow \varphi$, while the other direction remains valid. We retain the converse of contraposition, $(\varphi \rightarrow \psi) \rightarrow (\neg \psi \rightarrow \neg \varphi)$, while losing contraposition, $(\neg \psi \rightarrow \neg \varphi) \rightarrow (\varphi \rightarrow \psi)$. The resulting logic has some intuitionistic flavor, which would be worth determining exactly. 


A more dramatic route to non-classical logics would be to change the semantics of $\#$ terms altogether.  We explore this route further in \S\ref{section:beyond}.
\end{remark}


\subsection{From Logic to Counting}
While pure $\mathsf{FO}$ is also capable of encoding facts about counting and arithmetic, it is far less extensive. As already mentioned, first-order logic can define the simple counting quantifiers like $\exists ^{\geq n}x$; however, first-order logic does so by means of \emph{counting in the syntax}. That is, the formula expressing that there are at least $n$ objects satisfying a given condition achieves this by concatenating $n$ existential quantifiers and adding $n$ conjuncts. Basic arithmetic principles like $\exists^{\geq m}x.\varphi \rightarrow \exists^{\geq n}x.\varphi$ for $m\geq n$, thus follow from elementary logical patterns like distribution of existential quantification over conjunction, applied the requisite number of times (e.g., $m-n$ times). This style of counting in the syntax also produces a case-by-case formulation of the Pigeonhole Principle:\footnote{Even more simply, the Pigeonhole Principle has a natural encoding in propositional logic where complexity theorists have been interested in lower bounds on the lengths of proofs for instances of the principle across different proof systems \citep{cook_reckhow_1979,Krajicek}.}
\begin{example} \label{example:php} Suppose we have $k$ monadic predicate symbols $P_1,\dots,P_k$ and let $n>k$. Then, 
 \begin{equation} \label{eq:fophp}\big( \exists^{= n}x. \bigvee_{i \leq k} P_i(x) \wedge \forall x. \bigwedge_{i \neq j} \neg (P_i(x) \wedge P_j(x)) \big)\rightarrow \bigvee_{i \leq k} \exists^{\geq 2}x. P_i(x) \end{equation}
says that if these $k$ predicates together include $n$ objects, then at least one must include at least two objects. This schema is of course valid for every choice of $k$ and $n>k$.
\end{example}
We will see more examples of counting in the syntax with subsequent sections  (see especially  Remark \ref{remark:mlsr} and Appendix \ref{appendix:syntax}). 

\begin{remark} \label{remark:finvar} The fact that $\mathsf{FO}$ can only count in the syntax reverberates in interesting ways when we consider \emph{finite variable} fragments of $\mathsf{FO}$.  While the two-variable fragment is known to have the (bounded) finite-model property \citep{Mortimer}, which in turn establishes its  decidability, this fragment with counting quantifiers $\exists ^{\geq n}$ can easily enforce infinite models: 
$$ \forall x. \exists^{=1}y. R(x,y) \wedge \forall y .\exists^{\leq 1} x. R(x,y) \wedge \exists y . \forall x. \neg R(x,y) . $$
Such a language is in fact decidable \citep{GradelOtto}: like the two-variable fragment without counting, its satisfiability problem is \textsc{NExpTime}-complete \citep{Pratt-Hartmann}. However, the complexity analysis of this system and its extensions \citep{Kieronski}  reveals arithmetical content that does not appear in analyses of the plain two-variable fragment, witness connections to integer programming (\S\ref{integer}) and to semi-linear sets (\S\ref{section:finiteSO}). \end{remark}

\subsection{Finite Models} It is natural to consider a related system in the same language, but with interpretations restricted to \emph{finite} models. Call such a system $\mathsf{FO}^\phi(\#)$. As $\mathcal{L}_\#$ extends the language of first-order logic, Trakhtenbrot's Theorem tells us that the validities of $\mathsf{FO}^\phi(\#)$ are still not computably enumerable. Nonetheless, $\mathsf{FO}^\phi(\#)$ and variations on it have also been intensely studied in the literature on finite model theory.  See, e.g., \cite{Otto} or \cite{Schweikardt} for summaries of relevant work.

As an example of distinctive issues that come up in the finitary setting, one might ask about the \emph{asymptotic probabilities} of formulas in $\mathsf{FO}^\phi(\#)$ over finite structures. It was shown in \cite{Grumbach}
that $\mathsf{FO}$ with the H\"{a}rtig quantifier in fact possesses a zero-one law, just as pure $\mathsf{FO}$ does. As possession of a zero-one law is commonly interpreted as evidence that a logic cannot formalize any non-trivial counting, this can be taken as justification for our choice of comparison rather than  equality as a primitive. Indeed, $\mathsf{FO}^\phi(\#)$ lacks a zero-one law; e.g., $\#_x P(x) \succsim \#_x \neg P(x)$ has asymptotic probability $1/2$. It is conjectured in \cite{Grumbach} that (an extension of) $\mathsf{FO}^\phi(\#)$ nonetheless possesses a limit law, and that the limits are all rational numbers between $0$ and $1$. 

For many purposes in finite model theory (e.g., descriptive complexity) authors have been motivated to consider proper \emph{extensions} of our language $\mathcal{L}_\#$, a notable example being \emph{fixed point logic} with counting \citep{Immerman1992}. Our purpose here is different: we aim to isolate weaker fragments of this language that might further reveal the subtle interplay between logic and counting in elementary reasoning, also pinpointing differences and commonalities between finitary and infinitary patterns in counting. 

\subsection{Fragments of $\mathcal{L}_\#$} While full first-order logic with counting may be a natural starting point for exploring our subject, the above observations invite the search for natural fragments and weaker variants of $\mathsf{FO}(\#)$. It may be desirable, for example, to identify \emph{decidable} fragments of $\mathcal{L}_\#$. From this perspective it is noteworthy that some familiar ways of taming complexity are less effective here. For example, finite-variable fragments do not result in decidability: as shown by \cite{Gradel1999}, the two-variable fragment of $\mathsf{FO}(\#)$ is still undecidable ($\Pi^1_1$-complete, so we do observe a reduction in complexity, compared to Theorem \ref{pi12}). The two-variable fragment of $\mathsf{FO}^\phi(\#)$ is also undecidable. Evidently, a significant source of the complexity is the potent combination of counting and arbitrary quantificational-relational reasoning, witness Lemma \ref{ordertype}. The undecidability proof in \citealt{Gradel1999} for the two-variable fragment crucially involves counting successors along binary relations.

A more dramatic route would be to move to a much tamer syllogistic or propositional fragment \citep{Moss2016,Ding2020}. For instance, if we let $\mathcal{L}_\#^0$ be the language of propositional logic with count comparisons, the resulting system $\mathsf{PL}(\#)$ is easily shown to be decidable (e.g., it will follow immediately from our results below). This route at once eliminates relational reasoning and first-order quantification. 

\vspace{1ex}

An alternative route is to put relational reasoning to the side, but still retain first-order quantification. The \emph{monadic} fragment of $\mathcal{L}_\#$, which we will call $\mathcal{L}_\#^1$, does not allow counting along relations, but it otherwise preserves the counting content of $\mathsf{FO}(\#)$. Observe, for example, that our definition of the infinity quantifier in (\ref{infinity-quantifier}) and our reconstruction of logical connectives from count comparisons (\S\ref{section:countinglogic}) depend in no way on the arity of available predicates. We will thus use $\mathsf{MFO}(\#)$, monadic first-order logic with counting, as a base system to explore richer combinations (\S\ref{section:basicmonadic}). In this context we will consider adding second order quantification (system $\mathsf{MSO}(\#)$ in \S\ref{section:second-order}), as well as the ability to count not just individuals but \emph{sequences} of individuals (systems $\mathsf{MFO}(\sharp)$ and $\mathsf{MSO}(\sharp)$ in \S\ref{section:sequences}). 

Of course, counting along relations is also common and natural. We therefore explore a tractable \emph{modal} fragment of $\mathcal{L}_\#$, which we call $\mathcal{L}_\#^{\mathsf{ml}}$, as a way of taming the interaction among counting, quantification, and relational reasoning.  A summary appears in Table \ref{table:overview}.

\begin{table}
    \centering
    \begin{tabular}{c | l | l}
       \textbf{Language} \;\quad & \quad \textbf{Logical System} \;\quad & \quad \textbf{Typical Expression} \\ \hhline{===}
       $\mathcal{L}_\#$ \quad & \quad $\mathsf{FO}(\#)$ \quad & \quad $\forall x. \#_y R(x,y) \succ \#_y \big(R(y,x) \wedge P(y)\big)$ \\ \hhline{---} 
       $\mathcal{L}^2_\sharp$ \quad & \quad $\mathsf{MSO}(\sharp)$ \quad&\quad $\exists Y. \big( \sharp_x P(x) \approx \sharp_{x,u,v}(Y(x) \wedge Y(u) \wedge Y(v)) \big)$ \\ \hhline{---} 
       $\mathcal{L}^1_\sharp$ \quad & \quad $\mathsf{MFO}(\sharp)$ \quad&\quad  $\sharp_{x,y} \big(P(x) \wedge P(y)\big) \succsim \sharp_{x,y,z}\big(Q(x) \wedge Q(y) \wedge Q(z)\big)$ \\ \hhline{---} 
       $\mathcal{L}_\#^2$  \quad & \quad $\mathsf{MSO}(\#)$ \quad & \quad $\exists Y.\big(\#_x Y(x) \approx \#_x Q(x) \wedge \#_x P(x) \succ \#_xY(x)\big) $ \\  \hhline{---} 
       $\mathcal{L}_\#^1$  \quad & \quad $\mathsf{MFO}(\#)$ \quad & \quad $\exists y.\big( P(y) \wedge \#_x(P(x) \wedge x\neq y) \succ \#_x Q(x)\big) $ \\ \hhline{---}  
        $\mathcal{L}_\#^{\textsf{ml}}$ \quad & \quad  $\mathsf{ML}(\#)$ \quad & \quad $\#(\#\neg p \succ \# p) \succ \#(\#p \succsim \#\neg p)$ \\ \hhline{---}
       $\mathcal{L}_\#^0$ \quad & \quad $\mathsf{PL}(\#)$ \quad & \quad  $\#\neg p \succsim \#(p \vee q)$ \\ \hhline{---}
    \end{tabular} \vspace{.1in} 
    \vspace{1ex}
    \caption{A hierarchy of counting languages and logics, covered in  \S\ref{section:first-order}-\S\ref{section:modal}. For each logical system $\mathsf{L}(\#)$ we also have a version $\mathsf{L}^\phi(\#)$, where we restrict to finite models. In these systems terms can only denote natural numbers.}
    \label{table:overview}
\end{table}

Following this work we consider a different route altogether, namely changing the semantics of $\mathcal{L}_\#$. Relaxing either the logical interpretation (relativizing to sets of ``admissible'' variable assignments; cf. \citealt{Nemeti}) or the numerical content of the $\#$ terms again results in systems that retain much of the character of $\mathsf{FO}(\#)$, while gaining in tractability.

\section{Monadic First-Order Counting Logic} \label{section:basicmonadic} 
The system $\mathsf{MFO}(\#)$ of monadic first-order logic with identity and cardinality comparisons, though restricted in its expressive power, still captures a good deal of the natural reasoning mentioned in our Introduction. It is easy to see that \emph{numerical syllogisms} can be represented, and so can simple comparative reasoning with quantifiers like `most'. But $\mathsf{MFO}(\#)$ can also represent the earlier more complex inference
$$\begin{array}{l l l}
 \mbox{from }    & \mbox{`More }A\mbox{ than }B\mbox{ are }C\mbox{'} \quad\quad & (\#_x\big(A(x) \wedge C(x)\big) \succ \#_x \big(B(x) \wedge C(x)\big))\\
 \mbox{and}    & \mbox{`More }B\mbox{ than }C\mbox{ are }A\mbox{'} \quad\quad & (\#_x\big(B(x) \wedge A(x)\big) \succ \#_x \big(C(x) \wedge A(x)\big))\\
  \mbox{to }    & \mbox{`More }A\mbox{ than }C\mbox{ are }B\mbox{'} \quad\quad & (\#_x\big(A(x) \wedge B(x)\big) \succ \#_x \big(C(x) \wedge B(x)\big)).
\end{array}$$
The underlying Venn diagram-style reasoning will be analyzed more generally below.

Beyond the basic linguistic inference repertoire, $\mathsf{MFO}(\#)$ can also represent some of what we called ``grassroots mathematics''. Note, for instance, that Example \ref{example:php} encoding the Pigeonhole Principle only involved monadic predicates (and in fact did not even need $\#$-terms). In $\mathsf{MFO}(\#)$ we can also express a natural infinitary generalization:
 \begin{equation} \label{eq:fophpinf}\big( \exists^{\infty}x. \bigvee_{i \leq k} P_i(x) \wedge \forall x. \bigwedge_{i \neq j} \neg (P_i(x) \wedge P_j(x)) \big)\rightarrow \bigvee_{i \leq k} \exists^{\infty}x. P_i(x), \end{equation} stating that infinitely many objects in finitely many disjoint boxes (``pigeonholes'') must result in at least box  having infinitely many objects.

We will now look more systematically at what this monadic counting logic can express. 
Suppose $\mathsf{Pred} = \{P_1,\dots,P_n\}$ is  finite, and list the $2^n$ possible state-descriptions over $\mathsf{Pred}$ as $S_1,\dots,S_{2^n}$, so that each $S_i(x)$ is of the form $\bigwedge_{j \in J} P_j(x) \wedge \bigwedge_{j \notin J} \neg P_j(x)$. Call the extension of a state-description $S_i$ in a model a \emph{region}.  In $\mathcal{L}_\#^1$ we can easily state count comparisons between regions. A count comparison, such as a statement $\#_x S_i(x) \succsim \#_xS_j(x)$, can be succinctly written with numerical  variables replacing cardinalities: $\mathsf{s}_i \geq \mathsf{s}_j$. As the $S_i$ are pairwise disjoint we can more generally encode constraints involving sums of (cardinalities of) regions by disjunctions of state-descriptions. For instance, a sentence like $\#_x \bigvee_{i} S_i(x) \succsim \#_x \bigvee_j S_j(x)$ encodes a typical linear inequality between sums of variables $\mathsf{s}_1,\dots,\mathsf{s}_{2^n}$: \begin{eqnarray}
    \sum_{i} \mathsf{s}_i & \geq & \sum_j \mathsf{s}_j. \label{simple-inequality}
\end{eqnarray} By closing under Booleans we can of course express equality and strict inequality versions of (\ref{simple-inequality}). When restricting to finite models call the resulting logical system $\mathsf{MFO}^\phi(\#)$. In this case ``solutions'' to such (in)equations will always be natural numbers. However, if we allow models of arbitrary cardinality then solutions may involve infinite cardinal numbers. This is the system that we call  $\mathsf{MFO}(\#)$.

How much more can we express in $\mathsf{MFO}^\phi(\#)$ or $\mathsf{MFO}(\#)$ than the simple linear inequalities in (\ref{simple-inequality})? We have already seen an instructive example in the formula (\ref{infinity-quantifier}) defining the infinity quantifier. The encoding of $\exists^\infty x. S(x)$ for a state description $S$ is essentially an inequality statement $\mathsf{s} \geq \mathsf{s}+1$. The use of individual variables here is an instance of a more general pattern, also relevant in the finite case. Indeed, everything we say in the present section will apply equally to $\mathsf{MFO}^\phi(\#)$ and $\mathsf{MFO}(\#)$.

As above, consider two non-overlapping sets $T_1=\{S_i\}_i$, $T_2=\{S_j\}_j$ of state-descriptions, whose respective cardinalities we will label $\{\mathsf{s}_i\}_{i}$ and $\{\mathsf{s}_j\}_j$. Then we can encode not only inequalities like those in (\ref{simple-inequality}), but also those such as \begin{eqnarray}
    \sum_{i} \mathsf{s}_i & = & \sum_j \mathsf{s}_j + k \label{inequality-1} \\
    \sum_{i} \mathsf{s}_i & > & \sum_j \mathsf{s}_j + k. \label{inequality-2}
\end{eqnarray} For instance, to express (\ref{inequality-1}) we can assert the existence of $k$ distinct variables $\mathbf{y}$ all of which satisfy one of $T_1$, such that ``removing'' these elements from the regions spanned by $T_1$ results in the same cardinality as the regions spanned by $T_2$: \begin{eqnarray*}
\exists \mathbf{y}.\Big( \mathsf{diff}(\mathbf{y}) \wedge \bigwedge_{y \in \mathbf{y}} T_1(y) \wedge \#_x\big(\bigwedge_{y \in \mathbf{y}} x \neq y \wedge T_1(x)\big) \approx \#_xT_2(x)\Big).
\end{eqnarray*} Here $T_1(y)$ is shorthand for $\bigvee_i S_i(y)$, and similarly for $T_2(x)$. 

Meanwhile (\ref{inequality-2}) is expressed by replacing the equality with a strict inequality. In fact, with $k$ variables $\mathbf{y}$ (in addition to the variable $x$ used in the count comparisons) we can already encode (\ref{inequality-1}) and (\ref{inequality-2}) with a constant $2k$, simply by taking these variables $\mathbf{y}$ and ``adding'' them to the regions spanned by the $T_2$ (see Figure \ref{fig:variables} for visualization):
\begin{eqnarray}
\exists \mathbf{y}.\Big( \mathsf{diff}(\mathbf{y}) \wedge \bigwedge_{y \in \mathbf{y}} T_1(y) \wedge \#_x\big(\bigwedge_{y \in \mathbf{y}} x \neq y \wedge T_1(x)\big) \approx \#_x\big( \bigvee_{y \in \mathbf{y}} x = y \vee T_2(x)\big)\Big). \label{main-inequality}
\end{eqnarray} We are effectively stating that $|T_1| \geq k$, and that $|T_1|-k = |T_2|+k$; in other words, $|T_1| = |T_2|+2k$. Again, the same argument extends to inequality statements.  
\begin{figure} \centering 
   \begin{tikzpicture}
\draw [black,fill=blue!10] (0,0) rectangle (3,1.5);
\node at (.5,.5) {$P$};
\node at (3.75,.75) {\Large{}$=$};
\draw [black,fill=green!10] (7.5,0) rectangle (10,1.5);
\node at (8,.5) {$\neg P$};
\draw [black,fill=white,label=m1] (1.75,1) ellipse (.8cm and .25cm);
\draw [black,fill=blue!10,label=m2] (5.25,.75) ellipse (.8cm and .25cm);
\node at (6.75,.75) {\Large{}$+$};
\node at (5.25,-.75) {$\exists \mathbf{y}.\big(\mathsf{diff}(\mathbf{y}) \wedge \bigwedge_{i \leq k} P(y_i) \wedge \#_x(P(x) \wedge \bigwedge_{i \leq k}x \neq y_i) \approx \#_x (\bigvee_{i \leq k} x=y_i \vee \neg P(x))  \big)$};
 \path (2,1) edge[->,thick,dotted,bend left] (5.2,.8);
\end{tikzpicture} 
    \caption{A visualization of the formula expressing that the number of $P$ points (blue) is exactly $2k$ greater than the numbers non-$P$ points (green), where $k$ is the size of the ``extracted'' set of $P$ points (i.e., the size of $\mathbf{y}$). }\label{fig:variables}
\end{figure}


\subsection{Some Core Principles} \label{section:coreprinciples} Both systems, $\mathsf{MFO}^\phi(\#)$ and $\mathsf{MFO}(\#)$, are evidently invariant under automorphisms. In the monadic setting automorphisms are precisely the maps that permute elements within a region: all the points that satisfy a given state-description are indistinguishable. This means that if a property holds for one point in a region, it holds for every point in that region. This theme of \emph{permutation invariance} is characteristic of counting, and it will return when we discuss generalized quantifiers in \S\ref{section:gq}.

As demonstrated above, use of individual variables essentially allows manipulating regions---removing or adding points. We can correspondingly state a more general \emph{invariance principle}. Fix some variables $\mathbf{y}$ and a fixed (finite) set $\mathbf{P}$ of predicate letters, and let $\alpha^{\mathbf{y}}(x)$ specify a state-description for $x$ as well as which of the variables $\mathbf{y}$ are (un)equal to $x$. Then, for any formula $\varphi$ (in predicates $\mathbf{P}$), if there is at least one $x$ satisfying $\alpha^{\mathbf{y}}$ and $\varphi$, then \emph{every} $x$ satisfying $\alpha^{\mathbf{y}}$ also satisfies $\varphi$. Codified in a general invariance principle:
 \begin{equation} \tag{$\mathsf{INV}$} \exists x.\big(\alpha^{\mathbf{y}}(x) \wedge \varphi(x)\big) \rightarrow \#_x\big(\alpha^{\mathbf{y}}(x) \wedge \varphi(x)\big) \approx \#_x\big(\alpha^{\mathbf{y}}(x)\big). \label{invariance-main} \end{equation} Since either none of the $\alpha$'s satisfy $\varphi$ or all of them do, once we have specified $\alpha$ in a count formula, reference to $\varphi$ becomes redundant. In fact, (\ref{invariance-main}) follows from an even stronger statement (that is, stronger provided we admit infinite models): 
 \begin{equation} \tag{$\mathsf{INV2}$} \#_x\big(\alpha^{\mathbf{y}}(x) \wedge \varphi(x)\big) \succ \mathsf{0} \rightarrow \#_x\big(\alpha^{\mathbf{y}}(x) \wedge \neg \varphi(x)  \big) \approx \mathsf{0}. \label{invariance-infinite} \end{equation} 

 \vspace{0.5ex}
 
 A further useful observation about terms $\#_x\varphi$ is the following Extraposition Principle (\ref{uniform-replace}) for `uncaptured subformulas'. Subformulas of $\varphi$ that do not involve $x$ do not contribute any fine-grained information to the term's denotation. If the free variables of $\psi$ are not among the bound variables of $\#_x\varphi$, then the following is valid: \begin{equation} \tag{$\mathsf{SUB}$} (\psi \wedge \#_x\varphi \approx \#_x\varphi[\nicefrac{\top}{\psi}]) \vee (\neg\psi \wedge \#_x\varphi \approx \#_x\varphi[\nicefrac{\bot}{\psi}]).
    \label{uniform-replace}
 \end{equation} Here $\alpha[\nicefrac{\beta}{\gamma}]$ is the result of substituting $\beta$ for every occurrence of $\gamma$ in $\alpha$.  
 

\subsection{Normal Forms} \label{section:normalform}
The principles recorded in (\ref{invariance-main}) and (\ref{uniform-replace}), together with basic propositional reasoning and a few other elementary principles (see \S\ref{section:axioms} for the others), allow derivation of a normal form result, which works uniformly for $\mathsf{MFO}^\phi(\#)$ and $\mathsf{MFO}(\#)$. As a first step, we can show that any formula is equivalent to one with no embedded $\#$-terms or quantifiers within $\#$-terms, as these terms can always be replaced by stacks of unembedded existential quantifiers. This is already a dramatic departure from full relational $\mathsf{FO}(\#)$, where embedding is non-trivial. (Recall that $\mathsf{FO}(\#)$ with no embedded count comparisons was $\Pi^1_1$-complete, in stark contrast to Theorem \ref{pi12}.)

Define \emph{depth} $\mathsf{d}(\varphi)$ by recursion, with $\mathsf{d}(\alpha) = 0$ for $\alpha$ atomic, $\mathsf{d}(\varphi \wedge \psi) = \mbox{max}(\mathsf{d}(\varphi),\mathsf{d}(\psi))$, $\mathsf{d}(\neg \varphi) = \mathsf{d}(\varphi)$, while $\mathsf{d}(\#_x\varphi \succsim \#_y\psi) = \mbox{max}(\mathsf{d}(\varphi),\mathsf{d}(\psi))+1$ and $\mathsf{d}(\exists x. \varphi) =\mathsf{d}(\varphi)+1$.

\vspace{1ex}

Generically, a monadic formula with free variables $\mathbf{y},x$ can be written in disjunctive normal form $\bigvee_i \big(\alpha_i(\mathbf{y}) \wedge \alpha^{\mathbf{y}}_i(x) \wedge \varphi_i(\mathbf{y},x)\big)$, where $\alpha_i(\mathbf{y})$ specifies state-descriptions for $\mathbf{y}$ and which of these variables are (un)equal, $\alpha^{\mathbf{y}}_i(x)$ is as in the previous subsection, and $\varphi_i(\mathbf{y},x)$ is some other formula that may in general have positive depth. We want to show that any formula \begin{equation*}
    \#_x\bigvee_{i\in I} \big(\alpha_i(\mathbf{y}) \wedge \alpha^{\mathbf{y}}_i(x) \wedge \varphi_i(\mathbf{y},x)\big)\succsim \#_x\bigvee_{j\in J} \big(\alpha_j(\mathbf{y}) \wedge \alpha^{\mathbf{y}}_j(x) \wedge \varphi_j(\mathbf{y},x)\big)
\end{equation*} is equivalent to one with no embedded count comparisons inside the $\#_x$ terms. 
First, by (\ref{uniform-replace}) we can take the subformulas  $\alpha_i(\mathbf{y}),\alpha_j(\mathbf{y})$ outside the count comparisons, which leaves
\begin{equation*}
    \#_x\bigvee_{i\in I} \big(\alpha^{\mathbf{y}}_i(x) \wedge \varphi_i(\mathbf{y},x)\big) \succsim \#_x\bigvee_{j\in J} \big( \alpha_j^{\mathbf{y}}(x) \wedge \varphi_j(\mathbf{y},x)\big) \label{toward-normal}
\end{equation*} to analyze. Let $\kappa_k$ range over formulas $\exists x.\big(\alpha_k^{\mathbf{y}}(x) \wedge \varphi_k(\mathbf{y},x)\big)$ for $k \in I \cup J$. Then by appeal to (\ref{invariance-main}), we have the equivalent formula:
\begin{equation*}
    \bigvee_{K \subseteq I \cup J}\big(\bigwedge_{k \in K}\kappa_k \wedge \bigwedge_{k \notin K} \neg \kappa_k \wedge \#_x\bigvee_{i\in I\cap K} \alpha^{\mathbf{y}}_i(x) \succsim \#_x\bigvee_{j\in J \cap K}  \alpha_j^{\mathbf{y}}(x) \big)
\end{equation*}
Note that we have traded one level of $\#$ embedding for one existential quantifier (and since the $\varphi_k(\mathbf{y},x)$ subformulas have  a lower  nesting of count operators, they are taken care of by the inductive hypothesis). Since $\alpha_i^{\mathbf{y}},\alpha_j^{\mathbf{y}}$ are of depth $0$, this concludes the argument for:
\begin{lemma} Every $\mathcal{L}_\#^1$ formula is equivalent 
to one in which every count comparison subformula has depth exactly $1$. \label{lemma:depth1}
\end{lemma}
Using Lemma \ref{lemma:depth1}, the main result of this section is: 
\begin{theorem} Every depth $k+1$ sentence is equivalent, in $\mathsf{MFO}^\phi(\#)$ as well as in $\mathsf{MFO}(\#)$, to a disjunction of conjunctions of sentences specifying $T_1 = T_2 +m$ or $T_1 > T_2+m$, for $T_1,T_2$ sums of (cardinalities of) state-descriptions, and $m\leq 2k$. \label{normal}
\end{theorem} 
\begin{proof} We show more generally that a formula of depth $k+1$ over predicates $\mathbf{P}$ with free variables $\mathbf{y} = y_1,\dots,y_n$ is equivalent to a disjunction \begin{equation}\bigvee \big( \alpha(\mathbf{y}) \wedge  (\sigma)_{\alpha(\mathbf{y})} \big), \label{main-normalform}\end{equation} where $\alpha(\mathbf{y})$ ranges over possible descriptions of $\mathbf{y}$, and $\sigma$ is a complete description of the regions over $\mathbf{P}$, i.e., specifying $T_1 = T_2 +m$ or $T_1 > T_2+m$ for all $m\leq 2(n+k)$. The notation $(\sigma)_{\alpha(\mathbf{y})}$ denotes a formula that specifies the description $\sigma$ on the assumption of $\alpha(\mathbf{y})$. 

Now, we claim that for each disjunct of (\ref{main-normalform}), for all variable assignments $s$: \begin{eqnarray} \mathcal{M},s \vDash \alpha(\mathbf{y}) \wedge  (\sigma)_{\alpha(\mathbf{y})} & \Rightarrow & \mathcal{M}\mbox{ satisfies the }2(n+k)\mbox{ description }\sigma. \label{eq:normalform-equiv}
\end{eqnarray}
The statement in the theorem will be the special case of (\ref{eq:normalform-equiv}) with no free variables ($n=0$). 

\begin{example} For an example of such a disjunct over one predicate letter $P$, see the formula inside the existential quantifier in Fig. \ref{fig:variables}. This formula has $k$ free variables and $\#$-depth $1$. Here $\alpha(\mathbf{y})$ is the formula $\mathsf{diff}(\mathbf{y}) \wedge \bigwedge_{i \leq k} P(y_i)$, while $(\sigma)_{\alpha(\mathbf{y})}$ is the count comparison. Note that $(\sigma)_{\alpha(\mathbf{y})}$ has free variables and it ``means'' that $|P| = |\neg P| + 2k$ provided $\alpha(\mathbf{y})$ holds.  \end{example} 

To show that depth $k+1$ formulas are always equivalent to formulas (\ref{main-normalform}) satisfying (\ref{eq:normalform-equiv}), we proceed by inducting on $k$, starting with the case of depth $1$ formulas ($k=0$) in free variables $\mathbf{y}=y_1,\dots, y_n$. The critical case is a count comparison: $$\#_x \bigvee \alpha(\mathbf{y},x) \succsim \#_x \bigvee \beta(\mathbf{y},x).
$$ As before, by (\ref{uniform-replace}) we can separate out the descriptions of $\mathbf{y}$ to obtain a formula \begin{equation}\bigvee \big(\gamma(\mathbf{y}) \wedge \#_x\bigvee \alpha^{\mathbf{y}}(x) \succsim \#_x\bigvee\beta^{\mathbf{y}}(x) \big),\label{eq:next} \end{equation} where we have a $\gamma(\mathbf{y})$ disjunct exactly when $\gamma(\mathbf{y}) = \alpha(\mathbf{y}) = \beta(\mathbf{y})$; that is, all disjuncts inside the $\#$ terms must agree on the characterization of variables $\mathbf{y}$. It is then routine to check, by exhaustive check of all cases, that the count comparison in each disjunct of (\ref{eq:next}), in context $\gamma(\mathbf{y})$, asserts $T_1 = T_2 +m$ or $T_1 > T_2+m$ for $m\leq 2n$ (or a disjunction of such comparisons). So this fits the form in (\ref{main-normalform}), and (\ref{eq:normalform-equiv}) is satisfied. 

\vspace{0.5ex}

In  general, the normal forms (\ref{main-normalform}) for a fixed $k$ and $n$ are closed under Boolean combinations, so we only need to consider the case of depth $k+1$ and $n$ variables. By Lemma \ref{lemma:depth1} we can assume all count comparison subformulas have depth $1$, so it suffices to consider an existential quantification, which by induction we assume is \begin{equation*} \exists z. \bigvee \big( \alpha(\mathbf{y},z) \wedge (\sigma)_{\alpha(\mathbf{y},z)}\big). \label{ex1} \end{equation*}  Such a formula will be equivalent to $$\bigvee \exists z. \big(\alpha(\mathbf{y},z) \wedge (\sigma)_{\alpha(\mathbf{y},z)}\big)$$ and indeed  to 
\begin{equation}\bigvee \Big(\alpha(\mathbf{y}) \wedge \exists z. \big(\alpha^{\mathbf{y}}(z) \wedge (\sigma)_{\alpha(\mathbf{y},z)}\big)\Big).\label{equation-now} \end{equation} It remains to be seen that (\ref{equation-now}) is of the form (\ref{main-normalform}) with each disjunct satisfying (\ref{eq:normalform-equiv}). 
By the inductive assumption we know that for any $s$, if $\mathcal{M},s \vDash \alpha(\mathbf{y},z) \wedge (\sigma)_{\alpha(\mathbf{y},z)}$ then $\mathcal{M}$ satisfies the $2(n+k+1)$ description $\sigma$. But if $\mathcal{M},s \vDash \alpha(\mathbf{y}) \wedge \exists z. \big(\alpha^{\mathbf{y}}(z) \wedge (\sigma)_{\alpha(\mathbf{y},z)}\big)$, then there is a $z$-variant $s'$ of $s$ such that $\mathcal{M},s' \vDash \alpha(\mathbf{y},z) \wedge (\sigma)_{\alpha(\mathbf{y},z)}$, which establishes the result. 
\end{proof}

\subsubsection{Connection to Integer Programming} \label{integer} As with ordinary monadic first-order logic, putting a sentence into normal form may result in a significantly longer formula. The satisfiability problem for monadic first-order logic (as for the two-variable fragment) is \textsc{NExpTime}-complete \citep{Lewis1980}, even though checking satisfiability of normal forms is in NP. As with monadic logic, checking satisfiability of a normal form in $\mathsf{MFO}^\phi(\#)$ is of relatively low complexity. In fact, it is of the same complexity. A set of (in)equalities of types   (\ref{inequality-1}) and (\ref{inequality-2}) give us an \emph{integer program}, whose solvability is known to be decidable in NP-time \citep{borosh}. Meanwhile, the special case of integer programming in which all coefficients are $1$ or $0$---in other words, the special case of inequalities like those in (\ref{simple-inequality})---was already included in Karp's \citeyearpar{Karp1972} original list of NP-complete problems. With this lower bound we can conclude that the satisfiability problem for normal forms in $\mathsf{MFO}^\phi(\#)$ is NP-complete. 

\subsection{Questions of Definability} \label{section:mfodef} Theorem \ref{normal} affords a refined understanding of the numerical relations that can be defined in $\mathsf{MFO}^\phi(\#)$, as well as $\mathsf{MFO}(\#)$. 

Where $T$ is a set of state-descriptions, let $|T|_\mathcal{M}$ denote the sum of cardinalities of extensions in $\mathcal{M}$ of state-descriptions in $T$. We will say that $\mathcal{M} \sim_k \mathcal{M}'$ if for all $T_1,T_2$ and all $m \leq k$: \vspace{0.5mm} \begin{eqnarray*} |T_1|_\mathcal{M} \geq |T_2|_{\mathcal{M}} + m & \mbox{ iff } & |T_1|_{\mathcal{M}'} \geq |T_2|_{\mathcal{M}'} + m \end{eqnarray*} \vspace{0.5mm} Then, where $\mathcal{M} \equiv_k \mathcal{M}'$ signifies that $\mathcal{M}$ and $\mathcal{M}'$ agree on all sentences up to depth $k$, Theorem \ref{normal} immediately gives:\vspace{0.5mm} 
\begin{corollary} $\mathcal{M} \sim_{2k} \mathcal{M'}$ iff $\mathcal{M} \equiv_{k+1} \mathcal{M}'$. \label{undefinability}
\end{corollary} As an initial example, we can characterize precisely the binary logical quantifiers definable in $\mathsf{MFO}^\phi(\#)$ (see \S\ref{section:gq} for a proof, and for further discussion of generalized quantifiers):\vspace{0.5mm}
\begin{restatable}{theorem}{binaryq} The binary quantifiers definable in $\mathsf{MFO}^\phi(\#)$ correspond exactly to  those expressible in the first-order theory of $\langle\mathbb{N}\,;>\rangle$. \label{binary-quant}\vspace{0.5mm}
\end{restatable} This includes many of the standard logical quantifiers: `most', `all', `some', `all but one', `at least two', etc. 
The following gives an example of a statement that cannot be expressed. 
\begin{fact} \emph{`There are twice as many $P$s as $Q$s'} cannot be expressed in $\mathsf{MFO}^\phi(\#)$. \label{fact:twice}
\end{fact}
\begin{proof} Supposing it could, such a sentence would have some depth $k+1$. In light of Corollary \ref{undefinability}, it suffices to show that, for any $k$, we can find $\mathcal{M},\mathcal{M}'$ that disagree on the statement and yet $\mathcal{M} \sim_{2k} \mathcal{M}'$. Define a first model $\mathcal{M}$ with $9k$ elements, such that $|P^{\mathcal{M}}| = 6k$ while $|Q^{\mathcal{M}}| = 3k$. The statement clearly holds of $\mathcal{M}$. But now define $\mathcal{M'}$ with $9k+1$ elements, such that $|P^{\mathcal{M}'}| = 6k+1$ and again $|Q^{\mathcal{M}'}| = 3k$. The statement fails in $\mathcal{M}'$, yet  $\mathcal{M} \sim_{2k} \mathcal{M}'$.
\end{proof}

For a second example, consider a natural rendering of the natural language expression `many', often taken to refer to a number above some contextual threshold. On a more sophisticated, but not uncommon, reading (cf. \citealt{Westerstahl1985,Rett}), `Many $Q$s are $P$' amounts to a comparison between the \emph{proportion} of $P$s among the $Q$s and the proportion of $P$s overall, which we might symbolize as \vspace{1mm} \begin{eqnarray} \label{eq:many}
\frac{\#_x\big(P(x) \wedge Q(x)\big)}{\#_x Q(x)} & \succ & \frac{\#_x P(x)}{\#_x \top} .
\end{eqnarray}\vspace{-1mm}
\begin{fact} \emph{`Many $Q$s are $P$'} cannot be expressed in $\mathsf{MFO}^\phi(\#)$. \label{fact:many}
\end{fact}
\begin{proof} Again, for any $k$, we must find two models $\mathcal{M},\mathcal{M}'$ that disagree on the statement and yet $\mathcal{M} \sim_{2k} \mathcal{M}'$. It suffices to specify the cardinalities of four regions within the model: $p = |P \cap \overline{Q}|$, $q = |Q \cap \overline{P}|$, $r = |P \cap Q|$, $s = |\overline{P \cup Q}|$.

In both models let $r=k$, $q=3k$, and $p=4k$. In $\mathcal{M}$ let $s = 15k$, while in $\mathcal{M}'$ let $s=11k$. In both cases $s > p+q+r+2k$, so $\mathcal{M} \sim_{2k} \mathcal{M}'$, and $\mathcal{M} \equiv_{k+1} \mathcal{M}'$. However, in $\mathcal{M}$ we have $\frac{r}{r+q} > \frac{p+r}{p+q+r+s}$, while in $\mathcal{M}'$ the inequality fails. 
\end{proof}

We will return to more analysis of natural language constructions  in \S\ref{section:gq}. Note that Corollary \ref{undefinability} can be used to derive undefinability results in $\mathsf{MFO}(\#)$ as well:
\begin{fact} \label{fact:successor} The successor function on infinite cardinals is not expressible in $\mathsf{MFO}(\#)$.
\end{fact}
\begin{proof} Every two models that agree on the \emph{order} of cardinalities for infinite definable sets will stand in the relation $\sim_k$ for all $k$. 
\end{proof}

\subsubsection{Interpolation Failure} Another consequence of   Theorem \ref{normal} is a particularly simple normal form  result for the ``letterless'' fragment of $\mathcal{L}_\#^1$, that is, the fragment with no predicate symbols, built up from atomic formulas $\top$ and $\bot$. In fact, the  normal forms are identical to those for monadic first-order logic with the infinity quantifier \citep{Carreiro}:
\begin{lemma} Every letterless sentence is equivalent in $\mathsf{MFO}(\#)$ to a disjunction of formulas having one of the following forms $\exists^\infty x. \top$, 
$\forall^\infty x. \bot \wedge \exists^{\geq k} \top$, or  $\exists^{=k} x. \top$. \label{letterless-normalform}
\end{lemma} For the restriction $\mathsf{MFO}^\phi(\#)$ to finite models, this simplifies even further to include only statements of the form $\exists^{>k}.\top$ and $\exists^{=k} x. \top$. As a consequence we can show: 
\begin{proposition} Neither $\mathsf{MFO}^\phi(\#)$ nor $\mathsf{MFO}(\#)$ enjoys the interpolation property. \label{prop:interpolation}
\end{proposition}
\begin{proof} Let $\varphi(P)$ be the formula: $$\forall^\infty x. \bot \wedge \#_x(P(x)) \approx \#_x(\neg P(x)),$$ which is only true in finite models of even sizes. Let $\psi(Q)$ be the formula: $$\exists x.\#_y( y \neq x \wedge Q(y) ) \approx \#_y( y \neq x \wedge \neg Q(y) ),$$ which in finite models requires the domain to be odd. Evidently $\varphi(P) \vDash \neg \psi(Q)$. Let $\chi$ be a purported interpolant: $\varphi(P) \vDash \chi \vDash \neg \psi(Q)$. As $\chi$ must be letterless, Lemma \ref{letterless-normalform} implies that it must be a disjunction of sentences with one of the three specified forms. Furthermore, as it is entailed by $\varphi(P)$ we can assume that $\exists^\infty x. \top$ is not a disjunct. A straightforward case analysis shows that $\chi$ must either be true only in models up to some fixed size---in which case it cannot be entailed by $\varphi(P)$---or it is true in all finite models from some finite size onward---in which case it cannot entail $\neg \psi(Q)$. 
\end{proof}

A familiar way of extending a language to guarantee interpolation is to allow second-order quantification. We will  turn to such an extension below in \S\ref{section:second-order}. 

\vspace{1ex} 

But first, we analyze the reasoning content of our normal form analysis a bit further, since this will illustrate the sort of combined system with counting and logic that we are after.

\subsection{Questions of Axiomatization} \label{section:axioms} What is the calculus of valid reasoning suggested by our current systems? For both of our basic monadic systems, $\mathsf{MFO}^\phi(\#)$ and $\mathsf{MFO}(\#)$, we can locate a kind of separation between two components: 

\vspace{1ex}

(a) general, more ``logical'' principles that allow our normal form result (Theorem \ref{normal}) and

\vspace{0.5ex}

(b) more specific numerical reasoning  for solving systems of inequalities. 

\vspace{1ex}

We discuss each component in turn for the  system $\mathsf{MFO}^\phi(\#)$ which allows only finite domains.  The general system $\mathsf{MFO}(\#)$ involves one more component for dealing with infinite sets that we will remark on at the end.

\subsubsection*{Step I} The \emph{normal form principles} underlying our normal form result are as follows:\vspace{0.5ex}
\begin{enumerate}[label=(\alph*)]
    \item General validities of propositional and first-order predicate logic, \label{item:logic}\vspace{0.5mm}
    \item The two general principles (\ref{invariance-main}) and (\ref{uniform-replace}) highlighted earlier, \label{item:special}\vspace{0.5mm}
    \item The linear order properties of the relation $\succsim$. \label{item:linear}\vspace{0.5ex}
\end{enumerate}
Here the linearity in Principle \ref{item:linear}, used in our case distinctions, is worth high-lighting: \vspace{1mm}
\begin{equation} \tag{$\mathsf{COMP}$} \#_x\varphi \succsim \#_x\psi \vee \#_x\psi \succsim \#_x\varphi \label{eq:comp}\vspace{1mm}
\end{equation}
The soundness of (\ref{eq:comp}) in $\mathsf{MFO}(\#)$ depends on the axiom of choice. Indeed, (\ref{eq:comp}) is equivalent to the axiom of choice \citep{Hartogs}. Significantly, in the generalized semantics discussed in Section \ref{section:beyond}, Principles \ref{item:logic} and \ref{item:special} will remain valid, 
while the strong reasoning principle \ref{item:linear} is naturally replaced by just the pre-order properties for $\succsim$.

\subsubsection*{Step II} As a result of the normal form analysis, we are left with a satisfiability problem for inequalities all of whose variables denote natural numbers. This system can be solved effectively, e.g., using the well-known \emph{Fourier-Motzkin Algorithm} \citep[\S12.2]{Schrijver}.

\vspace{1ex}

At this stage, we might say that we have solved the reasoning problem in the spirit of this paper, having used a simple combination of logic and counting. The above calculus uses logic to reduce a reasoning problem to a numerical one that is most elegantly solved on its own terms. This is precisely the sort of combination that we find natural and insightful.

\begin{remark} \label{rem:fourier} Even so, we could go further in \emph{Step III}, and determine the exact arithmetical principles that drive the Fourier-Motzkin algorithm. Here is a sketch. 

The algorithm works as follows. One picks a variable $\mathsf{s}$ as long as still possible, and then considers one of three cases. \begin{enumerate}[label=(\roman*)]
\item \label{fm-1} The variable $\mathsf{s}$ occurs only to the right in inequalities of the system. Then $\mathsf{s}$ can be dropped from all inequalities: setting its value to 0 will always suffice. \item \label{fm-2} The variable $\mathsf{s}$ occurs only to the left in inequalities. Then all inequalities where $\mathsf{s}$ occurs can be dropped, since they can be made true at the end by choosing some suitably large value for $\mathsf{s}$. \item \label{fm-3} In case $\mathsf{s}$ occurs both to the left and to the right in inequalities, one groups the inequalities of the form $\mathsf{u} \geq \mathsf{s} + \mathsf{v}, \mathsf{w}> \mathsf{s} + \mathsf{z}$ and those of the forms $\mathsf{y} + \mathsf{s}\geq \mathsf{t}$, $\mathsf{r} + \mathsf{s}>\mathsf{x}$, and forms all sums as follows:
$\mathsf{u}\geq \mathsf{s} + \mathsf{v}$ with $\mathsf{y} + \mathsf{s}\geq \mathsf{t}$ gives $\mathsf{u}+\mathsf{y} \geq \mathsf{v} + \mathsf{t}$; $\mathsf{u} \geq \mathsf{s} + \mathsf{v}$ with $\mathsf{r} + \mathsf{s} > \mathsf{x}$ gives $\mathsf{u}+\mathsf{r} > \mathsf{v} + \mathsf{x}$, and so on. \end{enumerate}
In the end, a set of variable-free statements about concrete natural numbers remains, which can be inspected immediately for truth or falsity.

Now each step of this algorithm can be checked for the principles that guarantee its soundness. Here are a few representative illustrations. All steps involve evident  principles for inequalities, such as symmetry and associativity of addition, and  monotonicity inferences such as the implication from $\mathsf{u} \geq \mathsf{v}$ to $\mathsf{u} + \mathsf{z} \geq \mathsf{v}$. Step \ref{fm-1} also involves the equality $\mathsf{v} = \mathsf{v} + 0$, while step \ref{fm-2} involves $\mathsf{u} + \mathsf{v} \geq \mathsf{v}$. The key step \ref{fm-3}  involves principles like the equivalence of $\mathsf{z} + \mathsf{u} \geq \mathsf{v} + \mathsf{u}$ and $\mathsf{z} \geq \mathsf{v}$ and  addition principles such as the implication from $\mathsf{u} \geq \mathsf{v}$ and $\mathsf{w} \geq \mathsf{z}$ to $\mathsf{u}+\mathsf{w} \geq \mathsf{v} + \mathsf{z}$. The final inspection step involves some simple principles for the successor function, if we think of numbers as encoded in a unary format.
\end{remark}

The preceding observations amount to what one might call a  ``mixed''  axiomatization of the system $\mathsf{MFO}^\phi(\#)$, letting the logic do what it is good at: reducing assertions to normal form, and then letting the arithmetical  component do what \emph{it} is good at: solving  equational problems involving numbers. By itself, a two-stage 
analysis with reductions to syntactic normal forms 
plus a separate combinatorial analysis of the latter 
is a common practice in logic, e.g., in quantifier 
elimination arguments. However, the above 
specific division of labor between logic and counting is a perfect fit with the methodological spirit of this paper, and with the general empirical reasoning practices that we started with. We will return to such combinations of logic and (explicit) arithmetic a bit more systematically in  \S\ref{section:explicitarithmetic} below.

\vspace{1ex}

Even so, it is also natural to explore the road of greater purity, and ask for a purely logical axiomatization, or a purely numerical one. We consider each of these roads in turn.

\vspace{1ex}

Can the arithmetical steps in the Fourier-Motzkin algorithm be replaced by an illuminating \emph{purely logical proof system} that goes beyond routine transcription? There is an interesting conceptual issue here. The variable-elimination step in the algorithm typically forms sums of single variables in its step \ref{fm-3}, and these sums have no direct interpretation in our logical systems: in particular, $|P| + |P|$ has no defining expression  in our logical languages (Fact \ref{fact:twice}). There are ways of dealing with this problem, for instance, by adding special inference rules as is done in \cite{Ding2020}, which essentially axiomatizes the slightly smaller system $\mathsf{PL}(\#)$ (cf. the discussion in Appendix \ref{app:related}). Such inference rules can be seen as expressing the admissibility of certain \emph{model constructions} for the logic, such as taking disjoint unions. This makes sense in our case, since, while |$P$| + |$P$| may not be definable in a given model $\mathcal{M}$,  it does denote the extension of $P$ in the disjoint union of $\mathcal{M}$ with itself. Even so, we are not aware of obvious model constructions matching the invariants needed for $\mathsf{MFO}(\#)$, and therefore leave this issue as an open problem.


Equally well, in terms of purity, one could go to the opposite side and ask for a \emph{purely numerical calculus} for our systems. We could restrict ourselves to (the monadic fragment of) the small sublanguage $\mathcal{L}_\#^-$ consisting only of predication, variable inequality, and count comparison (\S\ref{section:countinglogic}). As all logical operators are definable there, such an axiomatization would be possible in principle. For example, the following numerical claims (suppressing individual variables as all are chosen fresh) capture the basic principles of propositional logic:
\vspace{0.5ex}
    \begin{enumerate}
        \item $\#(\#\varphi \succsim \#\psi) \succsim \#\varphi$
        \vspace{1ex}
        \item $\#\big( \#(\#\chi \succsim \#\varphi) \succsim \#(\#\psi \succ \#\varphi) \big) \succsim \#\big( \#(\#\chi \succsim \#\psi) \succsim \#\varphi \big)$
        \vspace{1ex}
        \item $\#(\#\varphi \succsim \#\psi) \succsim \#\big( \#(\mathsf{0} \succsim \#\psi) \succsim \#(\mathsf{0} \succsim \varphi) \big)$
    \end{enumerate}\vspace{0.5ex}
However, what one wants is not transcription, but an independently motivated numerical system that generates the logic. As with the purely logical axiomatization, we leave  providing an illuminating purely numerical  axiomatization of our systems as an open problem.

\subsubsection*{Infinite cardinalities} In the general system $\mathsf{MFO}(\#)$, we must also deal with infinite cardinalities. This makes no difference to the principles producing our normal forms, but it changes the subsequent phase of solving equations. The key observation is that there is a simple expression distinguishing the infinite from the finite extensions, namely $\mathsf{s} \geq \mathsf{s} + 1$. We can thus completely separate reasoning about inequalities among finite variables from reasoning about the variables denoting infinite sets (cf. \citealt{Ding2020}). This is done systematically below in \S\ref{section:soinf} for the second-order version of $\mathsf{MFO}(\#)$, to which we now turn. 








\section{Monadic Second-Order Counting Logic} \label{section:second-order} From a logical point of view, a natural extension of $\mathcal{L}_\#^1$ is to allow quantification over unary predicates denoting subsets of the domain.  Call the resulting language $\mathcal{L}_\#^2$, and the finitary and general systems $\mathsf{MSO}^\phi(\#)$ and $\mathsf{MSO}(\#)$, respectively. 

One immediate observation is that, while in the first-order language $\mathsf{MFO}$ count comparisons $\#_x \varphi \succsim \#_x \psi$ are not definable from the H\"{a}rtig quantifier $\#_x\varphi \approx \#_x\psi$ alone (see, e.g.,  \citealt{PetersWesterstahl}, p. 470), with second-order quantification present, providing such a definition is straightforward: \vspace{-0.5ex}
\begin{equation*}
    \exists X . \big(\#_x\varphi \approx \#_x(\psi \vee X(x))\big).
\end{equation*} How much more powerful will $\mathsf{MSO}^\phi(\#)$ and $\mathsf{MSO}(\#)$ be in comparison to $\mathsf{MFO}^\phi(\#)$ and $\mathsf{MFO}(\#)$? The question is of some interest, since it is known (at least since \citealt{Ackermann}) that adding second-order quantification to monadic first-order logic does not increase expressive power. At the same time, if we add quantification over \emph{finite} sets to $\mathsf{MFO}$ this becomes equivalent to monadic logic with the infinity quantifier (see \citealt{Vaananen} for the case without equality, or Appendix \ref{app:infinityquantifier} for the case of including equality).

The failure of interpolation (Proposition \ref{prop:interpolation} above) shows that we could not expect a similar collapse when adding monadic second-order quantification to our counting extensions of $\mathsf{MFO}$. We saw that $\mathsf{MFO}(\#)$ can already distinguish between finite and infinite, so in effect we automatically gain access to quantification over finite sets. In fact we gain much more in terms of arithmetical expressive power, as will be shown in what follows.

\begin{example} As a preview, within the finite setting, in contrast to $\mathsf{MFO}^\phi(\#)$ (Fact \ref{fact:twice} above), in $\mathsf{MSO}^\phi(\#)$ the statement `There are twice as many $P$s as $Q$s' now becomes expressible: \begin{equation*}
    \exists X . \big(\#_y (X(y) \wedge \neg Q(y)) \approx \#_y Q(y) \wedge \#_y P(y) \approx \#_y (X(y)\vee Q(y)) \big).
\end{equation*} This essentially asserts the existence of a set whose extension outside of $Q$ is the same size as $Q$, and that $P$ is the same size as the union of these two. \label{example:twice} \end{example}

It turns out that Example \ref{example:twice} is just the tip of the iceberg. In addition to obviously  guaranteeing interpolants, there is another sense in which these second-order systems, $\mathsf{MSO}^\phi(\#)$ and $\mathsf{MSO}(\#)$, ``fill in the gaps'' of $\mathsf{MFO}^\phi(\#)$ and  $\mathsf{MFO}(\#)$. While the latter systems could enforce a certain type of inequality between sums, namely those in Eqns.  (\ref{inequality-1}) and (\ref{inequality-2}) above, the second-order versions are capable of enforcing arbitrary linear constraints over cardinalities which involve addition of numbers.

We now proceed to make this more precise, first in the finitary case, then infinitary. \vspace{-1ex}

\subsection{Finitary Case} \label{section:finiteSO} We saw above that normal forms in $\mathsf{MFO}^\phi(\#)$ correspond to (disjunctions of) sets of inequality constraints, a class whose solvability problem is already NP-complete. In the general setting of integer programming, there is a close correspondence between sets of linear inequalities and quantifier-free formulas of \emph{Presburger Arithmetic}, that is, first-order logic with addition over the natural numbers (see, e.g., \citealt{OPPEN1978323}). The sets of solutions to such inequalities (or equivalently, assignments satisfying Presburger formulas with free variables) are exactly the \emph{semi-linear} sets \citep{Ginsburg}, a generalization of the ``ultimately periodic'' sets of numbers:

Here are the relevant definitions and more precise results.

 \begin{definition}A set $\mathcal{V}\subseteq \mathbb{N}^n$ of $n$-ary vectors is called \emph{linear} if there is a system of equations over variables $\mathsf{v}_1,\dots,\mathsf{v}_n,\mathsf{u}_1,\dots,\mathsf{u}_m$ and constants $b_1,\dots,b_n,a_{1,1},\dots,a_{n,m}$,  \begin{equation}
 \begin{pmatrix}
\;\mathsf{v}_1\;  \\
\vdots  \\
\mathsf{v}_n  
\end{pmatrix} \hspace{.1in}= \hspace{.1in}
 \begin{pmatrix}
b_1 + a_{1,1}\mathsf{u}_1 + & \dots  & + a_{1,m}\mathsf{u}_m  \\
 \vdots & & \vdots \\
b_n + a_{n,1}\mathsf{u}_1 + & \dots & +  a_{n,m}\mathsf{u}_m
\end{pmatrix}
\label{first-semi} \end{equation} such that $\mathbf{x} \in \mathcal{V}$ if and only if there exist values of $\mathsf{u}_1,\dots,\mathsf{u}_m$ for which $\mathbf{v}=\mathbf{x}$ is a solution to (\ref{first-semi}). We say $\mathcal{V} \subseteq \mathbb{N}^n$ is \emph{semi-linear} if it is a finite union of linear sets.\label{def:linear} \end{definition}

\begin{restatable}{definition}{defdef} \label{define:define}
Suppose $S_1,\dots,S_n$ are some state-descriptions over predicates $\mathbf{P}$, and that $\varphi$ is an $\mathcal{L}_\#^2$ sentence in these same predicates $\mathbf{P}$. We say that $\varphi$ \emph{defines} a set $\mathcal{V} \subseteq \mathbb{N}^n$ just in case, for any model $\mathcal{M}$, we have \begin{eqnarray}
  \mathcal{M} \vDash \varphi & \mbox{ iff } & [|S_1|_{\mathcal{M}},\dots,|S_n|_{\mathcal{M}}] \in \mathcal{V}. \label{eq:df}
\end{eqnarray}
\end{restatable}
\begin{remark} Other notions of definability will make sense for our analysis of later systems. See Section~\ref{twodefinitions} for an alternative notion of `trace-definability' and a comparison.

\end{remark}
\begin{lemma} Every semi-linear set is definable in $\mathsf{MSO}^\phi(\#)$. \label{lemma:semi-linear}
\end{lemma}
\begin{proof} As $\mathcal{L}_\#^2$ closes under disjunction, it suffices to show that every linear set is definable. So we describe how to encode any linear set of the form in (\ref{first-semi}) by an $\mathcal{L}_\#^2$ sentence.  In words: \begin{enumerate}[label=(\roman*)]
    \item \label{eq:stepone} For all $i\leq n$, assert the existence of:\vspace{0.5mm} \begin{itemize} 
    \item Sets $Z_{i,j,1} \dots, Z_{i,j,a_{i,j}}$ (none if $a_{i,j}=0$) for all $j \leq m$,\vspace{0.5mm}
    \item Individuals $z_{i,1}, \dots, z_{i,b_i}$ (none if $b_i=0$). \end{itemize}\vspace{1mm}
    \item Add conjuncts for:\vspace{0.5mm} \begin{itemize} \item  $\#_x\big(Z_{i,j,p}(x) \wedge Z_{i,j',p'}(x)\big) \approx \mathsf{0}$, whenever $j\neq j'$ or $p \neq p'$\vspace{0.5mm} \item \label{basic-conditions} $\neg Z_{i,j,p}(z_{i,l})$, for all $i,j,p,l$\vspace{0.5mm}
    \item $\#_xZ_{i,j,p}(x) \approx \#_xZ_{i',j,p'}(x)$, for all $j$ and all $i,i',p,p'$ \end{itemize}\vspace{1mm}
    \item \label{eq:mainseme} Finally, conjoin these together with the claim that each state-description $S_i$ has the same cardinality as the union of all $Z_{i,j,p}$, together with $z_{i,1},\dots,z_{i,b_i}$: \begin{eqnarray*}\#_xS_i(x) & \approx & \#_x \big( \bigvee_{j\leq m \atop p \leq a_{i,j}} Z_{i,j,p}(x) \vee \bigvee_{l \leq b_i} x = z_{i,l}  \big).
    \end{eqnarray*} 
\end{enumerate} For a given $i\leq n$ and $j \leq m$, the sets $Z_{i,j,p}$ correspond to $a_{i,j}$-many  copies of the variable $\mathsf{u}_j$ in (\ref{first-semi}). The individual variables $z_{i,l}$ count the constant ``base'' number $b_i$. The numerical equalities stated in \ref{eq:mainseme} guarantee that each ``variable'' $|S_i|_\mathcal{M}$ has the right cardinality according to (\ref{first-semi}), under the conditions specifies by \ref{basic-conditions}. By existentially quantifying all variables \ref{eq:stepone} the resulting formula defines the linear set in (\ref{first-semi}) in the sense of (\ref{eq:df}). 
\end{proof}
We now want to show the other direction, that all $\mathcal{L}_\#^2$ sentences in fact define semi-linear sets. Toward this result we first note that $\mathsf{MSO}(\#)$ possesses a prenex normal form. 
\begin{lemma} \label{lemma:prenex}
Every $\mathcal{L}_\#^2$ sentence in predicates $\mathbf{P}$ is equivalent---in $\mathsf{MSO}(\#)$ as well as in $\mathsf{MSO}^\phi(\#)$---to one in prenex form, that is, of the form $Q_1X_1,\dots,Q_nX_n. \varphi(\mathbf{P},X_1,\dots,X_n)$, where $\varphi(\mathbf{P},X_1,\dots,X_n)$ is a first-order $\mathcal{L}_\#^1$ sentence (treating $X_1,\dots,X_n$ as additional predicates) and $Q_1,\dots,Q_n$ are second-order quantifiers.
\end{lemma}
\begin{proof}[Proof Sketch] The argument is as usual for prenex normal forms in first-order logic. As in case of $\mathsf{MFO}(\#)$, the soundness of (\ref{invariance-main}) allows us to extract any first-order or second-order quantifier from the scope of a $\#$-term. The only case we need to consider here is a first-order (universal) quantifier scoping directly over a second-order quantifier. The point is to convert the first-order universal quantifier into a universal second-order quantifier restricted to singleton sets. That is, a formula $\forall x. QY. \varphi$, with $Q$ a second-order quantifier, will be equivalent to $\forall X. QY. \forall x. \big(\forall z(X(z) \leftrightarrow z=x) \rightarrow \varphi  \big)$.
\end{proof}

By Theorem \ref{normal} we know that $\varphi(\mathbf{P},X_1,\dots,X_n)$ has a normal form involving  expressions $T_1=T_2+m$ and $T_1 > T_2+m$, where $T_1,T_2$ are cardinalities of state-descriptions over $\mathbf{P}$ and additional predicates $X_1,\dots,X_n$. Such formulas are thus easily seen to be semi-linear, and indeed linear. Now, as semi-linear sets are closed under Boolean combinations, and second-order quantifiers distribute over disjunction, the main goal is to show the following closure property for second-order quantifiers:
\begin{lemma} Let $X$ be a predicate variable, and $O_1,\dots,O_n$ be either predicate letters or predicate variables. Suppose that $\varphi(O_1,\dots,O_n,X)$ defines a linear set in state-descriptions $S_1,\dots,S_{2^{n+1}}$ over $O_1,\dots,O_n,X$. Then $\exists X. \varphi(O_1,\dots,O_n,X)$ defines a linear set in state-descriptions $S_1',\dots,S_{2^n}'$ over just $O_1,\dots,O_n$.
\end{lemma}
\begin{proof} As $\varphi(O_1,\dots,O_n,X)$ is linear, we can assume it defines the solutions to  \vspace{0.5ex} \begin{equation}
 \begin{pmatrix}
\;|S_1|\;  \\
\vdots  \\
|S_{2^{n+1}}|  
\end{pmatrix} \hspace{.1in}= \hspace{.1in}
\begin{pmatrix}
b_1 + a_{1,1}\mathsf{u}_1 + & \dots  & + a_{1,m}\mathsf{u}_m  \\
 \vdots & & \vdots \\
b_{2^{n+1}} + a_{2^{n+1},1}\mathsf{u}_1 + & \dots & +  a_{2^{n+1},m}\mathsf{u}_m
\end{pmatrix}. \label{semi-lin}
\end{equation}\vspace{0.5ex} To show that $\exists X. \varphi(O_1,\dots,O_n,X)$, too, is linear, we define another linear set of equations by ``projecting out'' the variable $X$. 

\vspace{1ex}

Specifically, note for each state-description $S_k'$, both $S_k' \wedge X$ and $S_k' \wedge \neg X$ are (equivalent to) some state-descriptions, $S_i$ and $S_j$, and in fact $S_k'$ is equivalent to $S_i \vee S_j$. The new linear system in $2^n$ variables is then as follows for each $k\leq 2^n$: \begin{eqnarray} |S'_k| & = &  b_i + b_j + (a_{i,1}+a_{j,1})\mathsf{u}_1+ \dots + (a_{i,m} + a_{j,m})\mathsf{u}_m . \label{ex-line}\end{eqnarray} It remains only to show that:\vspace{0.5ex} \begin{eqnarray}
   \mathcal{M} \vDash \exists X. \varphi(O_1,\dots,O_n,X) 
   & \Leftrightarrow & [|S_1'|_{\mathcal{M}},\dots,|S_{2^n}'|_{\mathcal{M}}] \mbox{ is a solution to equations }(\ref{ex-line}).\vspace{1ex} \nonumber
\end{eqnarray} 

$(\Rightarrow)$: If $\mathcal{M} \vDash \exists X. \varphi(O_1,\dots,O_n,X)$ then for some subset $A$ of the domain, we have $\mathcal{M},s^X_A \vDash \varphi(O_1,\dots,O_n,X)$. Treating $X$ now as a predicate constant, we have a model $\mathcal{M}'$ for which $X^{\mathcal{M}'}=A$, and by assumption this gives a solution $[|S_1|_{\mathcal{M}'},\dots,|S_{2^{n+1}}|_{\mathcal{M}'}]$ to (\ref{semi-lin}). But each state-description $S_k'$ is equivalent to a disjunction $S_i \vee S_j$, whose cardinality is the sum $|S_i|+|S_j|$. Therefore $\mathcal{M}'$ will satisfy each of the constraints in (\ref{ex-line}). As the state-descriptions $S_1',\dots,S_{2^n}'$ are independent of $X$, this means  $[|S_1'|_{\mathcal{M}},\dots,|S_{2^n}'|_{\mathcal{M}}]$ also gives a solution to (\ref{ex-line}).

\vspace{1ex}

$(\Leftarrow)$: Suppose $[|S_1'|_{\mathcal{M}},\dots,|S_{2^n}'|_{\mathcal{M}}]$  gives a solution to (\ref{ex-line}) for some particular choices $\mathsf{u}_1,\dots,\mathsf{u}_m$. We need to find a set $A$ such that $\mathcal{M},s^X_A \vDash \varphi(O_1,\dots,O_n,X)$. Since the extensions of $S_1',\dots,S_{2^n}'$ are all disjoint, to define $A$ it suffices to identify subsets of each $\semantics{S_k'}_{\mathcal{M}}$. As above, suppose $S'_k$ is equivalent to $S_i \vee S_j$, so that $S_i$ is equivalent to $S'_k \wedge X$ and $S_j$ is equivalent to $S_k' \wedge \neg X$. Then let $B_k$ be any subset of $\semantics{S_k'}_{\mathcal{M}}$ of size $b_i + a_{i,1}\mathsf{u}_1 + \dots + a_{i,m}\mathsf{u}_m$, such that the complement $\semantics{S'_k}_{\mathcal{M}} - B_k$ has size $b_j + a_{j,1}\mathsf{u}_1 + \dots + a_{j,m}\mathsf{u}_m$. This is always possible since $|S'_k|$ is simply the sum of these two numbers. Finally let $A = \bigcup_{k\leq 2^n} B_k$. 

Once again absorbing $X$ into the language and defining $\mathcal{M}'$ to be just like $\mathcal{M}$ but with $X^{\mathcal{M}'} = A$, the tuple $[|S_1|_{\mathcal{M}'},\dots,|S_{2^{n+1}}|_{\mathcal{M}'}]$ gives a solution to (\ref{semi-lin}) with the same choices $\mathsf{u}_1,\dots,\mathsf{u}_m$. Hence, $\mathcal{M}'\vDash \varphi(O_1,\dots,O_n,X)$, from which it easily follows that $\mathcal{M},s^X_A \vDash \varphi(O_1,\dots,O_n,X)$ and finally $\mathcal{M} \vDash \exists X. \varphi(O_1,\dots,O_n,X)$.
\end{proof}

The foregoing thus establishes:
\begin{theorem} The numerical relations definable in $\mathsf{MSO}^\phi(\#)$ are the semi-linear sets. In other words, $\mathsf{MSO}^\phi(\#)$ expresses the same numerical relations as Presburger Arithmetic. \label{thm:semi-lin} 
\end{theorem}
\begin{remark} \label{remark:soc} In $\mathcal{L}_\#^2$ we allow arbitrary second-order quantification. However, we saw in Lemma \ref{lemma:semi-linear} that we only needed an initial block of \emph{existential} second-order quantifiers to encode any (semi-)linear set. The fact that every sentence in $\mathsf{MSO}^\phi(\#)$ defines a semi-linear set demonstrates a collapse of $\mathsf{MSO}^\phi(\#)$ into its purely existential fragment. 
\end{remark}

As in the first-order case, numerous undefinable results again follow. For example: 
\begin{corollary} The expression `many' is still not definable in $\mathsf{MSO}^\phi(\#)$.
\end{corollary}
\begin{proof} Adopting the notation from the proof of Fact \ref{fact:many}, the constraint on state-descriptions, $\frac{r}{r+q} > \frac{p+r}{p+q+r+s}$, is not semi-linear. 
\end{proof}

Indeed, the theory of definability for Presburger Arithmetic carries over exactly to $\mathsf{MSO}^\phi(\#)$, thanks to Theorem \ref{thm:semi-lin}. Moreover, since there is an algorithmic means of putting a formula of $\mathsf{MSO}^\phi(\#)$ into normal form and finding a suitable semi-linear form, decidability follows from the decidability of Presburger Arithmetic.
\begin{corollary} The validity problem for $\mathsf{MSO}^\phi(\#)$ is decidable. 
\end{corollary}

\subsection{Infinitary Case} \label{section:soinf} Allowing second-order quantification does increase the expressive power of our initial system $\mathsf{MFO}^\phi(\#)$. While the latter essentially amounts to a proper fragment of Presburger Arithmetic, $\mathsf{MSO}^\phi(\#)$ gave us precisely Presburger Arithmetic. How does this look for the system $\mathsf{MSO}(\#)$ over models of arbitrary cardinality? One immediate difference is that, in contrast to $\mathsf{MFO}(\#)$ (Fact \ref{fact:successor}), the \emph{successor} function on cardinal numbers can now be easily expressed: $$\forall X.\big( \#_yP(y) \succ \#_y X(y) \rightarrow \#_y Q(y) \succsim \#_yX(y)\big).$$ This formula states that there is no cardinality strictly in between that of $P$ and that of $Q$. How much more cardinal arithmetic does $\mathsf{MSO}(\#)$ encode? 

As in the case of $\mathsf{MSO}^\phi(\#)$, we can calibrate this by appeal to additive first-order (now cardinal) arithmetic. Consider the elementary theory of the structure $\langle C_{\aleph_\omega};+ \rangle$, addition over the cardinals numbers less than $\aleph_\omega$. This is the theory of cardinal numbers in a first-order language with one binary function symbol, namely addition. 

We show in Appendix \ref{app:infinityaddition} that this theory admits quantifier elimination provided we augment the language with constants for the (definable in $\mathsf{MSO}(\#)$) functions and relations:\vspace{1ex}
\begin{itemize}
    \item $\{0\}$ and $\{\aleph_0\}$\vspace{1mm}
    \item $s$ --- the successor function\vspace{1mm}
    \item $>$ --- the ``greater than'' relation \vspace{1mm}
    \item $\equiv_k$ --- equivalence modulo $k$ for each $k>1$
\end{itemize} 
Furthermore, we can derive a normal form result for this language:
\begin{proposition} \label{cardinal:normalform} Every first-order sentence is equivalent over the structure $\langle C_{\aleph_\omega};+ \rangle$ to a disjunction of conjunctions $\delta \wedge \iota \wedge \phi$ each specifying:\vspace{-0.3ex} \end{proposition} \begin{itemize}
    \item which (ordinary first-order) variables in the disjunct are finite or infinite ($\delta$)
    \item a description of a linear set for the finite variables ($\phi$)
    \item a description of a set of infinite cardinals using $0$, $s$, and $>$ over $\aleph$-number indices, for the infinite variables ($\iota$).\vspace{0.5ex}
\end{itemize}
This can be understood as a kind of separation result. The finitary part, Presburger Arithmetic, is simply ordinary addition. As for the infinitary part, observe that there is an isomorphism from $\langle \mathbb{N};0,s,>\rangle$ onto $\langle \{\aleph_k\}_{k\in\mathbb{N}}; \aleph_0,s,> \rangle$, sending $k$ to $\aleph_k$. In other words, the additive structure of cardinals less than $\aleph_\omega$ amounts to a  ``product'' of $\langle \mathbb{N};+ \rangle$ and $\langle \mathbb{N};> \rangle$.\footnote{A version of this result can be traced back at least to \cite{MostowskiTarski}. A general notion of product that subsumes this case is that of \cite{FefermanVaught}.}

Our aim is to show that $\mathsf{MSO}(\#)$ possesses the same normal forms as in Proposition \ref{cardinal:normalform}.  To see that any statement of the form  $\delta \wedge \iota \wedge \phi$ can be expressed, note that $\delta$ merely requires distinguishing finite and infinite sets (recall Eq. (\ref{infinity-quantifier})), while definability of any linear set (specified by $\phi$) was shown already in Lemma \ref{lemma:semi-linear}. Meanwhile, $\iota$ is a conjunction of formulas of types $\mathsf{v}=s^k(\mathsf{u})$, $\mathsf{v}>s^k(\mathsf{u})$, $\mathsf{v} = \aleph_k$, and $\mathsf{v}>\aleph_k$. We noted above that successor is expressible, and, for instance, we can assert that $P$ has cardinality $\aleph_0$ simply by stating that $P$ is infinite and there is no infinite set with smaller cardinality. Thus, any such statement is expressible.

To show that this exhausts what is definable in $\mathsf{MSO}(\#)$, given Lemma \ref{lemma:prenex}, it remains to observe that the $\mathcal{L}_\#^2$-definable sets are closed under ``projection'' by existentially quantifying one of the variables. Thus, suppose we have an $\mathcal{L}_\#^2$ formula $\varphi(\mathbf{O},Y)$ with $Y$ a predicate variable and $\mathbf{O}=O_1,\dots,O_n$ all either predicate variables or letters. We will assume $\varphi(\mathbf{O},Y)$ has the form  $\delta(\mathbf{O},Y) \wedge \iota(\mathbf{O},Y) \wedge \phi(\mathbf{O},Y)$, analogously to the additive language: $\delta(\mathbf{O},Y)$ describes which state descriptions over variables $O_1,\dots,O_n,Y$ are (in)finite, $\iota(\mathbf{O},Y)$ characterizes the infinite state descriptions, while $\phi(\mathbf{O},Y)$ describes a linear set. We need to analyze $\exists Y.\big(\delta(\mathbf{O},Y) \wedge \iota(\mathbf{O},Y) \wedge \phi(\mathbf{O},Y)\big)$. 

We can replace $\delta(\mathbf{O},Y)$ with a formula $\delta'(\mathbf{O})$ specifying that a state description $S$ over $O_1,\dots,O_n$ is finite iff $S \wedge Y$ and $S \wedge \neg Y$ were both finite according to $\delta(\mathbf{O},Y)$. List the finite state descriptions according to $\delta(\mathbf{O},Y)$ as $S_1,\dots,S_k$. The subformula $\phi(\mathbf{O},Y)$ defines a linear set over the possible (finite) cardinalities:
\begin{equation}
 \begin{pmatrix}
\;|S_1|\;  \\
\vdots  \\
|S_{k}|  
\end{pmatrix} \hspace{.1in}= \hspace{.1in}
\begin{pmatrix}
b_1 + a_{1,1}\mathsf{u}_1 + & \dots  & + a_{1,m}\mathsf{u}_m  \\
 \vdots & & \vdots \\
b_{k} + a_{k,1}\mathsf{u}_1 + & \dots & +  a_{k,m}\mathsf{u}_m
\end{pmatrix} \label{semi-lin-6}
\end{equation} Suppose $S_i = S \wedge Y$ is finite but $S \wedge \neg Y$ is infinite. Then the constraint in (\ref{semi-lin-6}) on $|S_i|$ is no constraint at all: since $|S|$ must be infinite, carving out a finite portion $S \wedge Y$ of any size will always be possible. So in this case we can simply drop the equation for $S_i$. Otherwise, if $S_i = S\wedge Y$ and $S_j=S \wedge \neg Y$ are both finite, then we can repeat the argument from \S\ref{section:finiteSO} above, again combining these two equations into a single equation for $|S|$. The result is a set of equations in (cardinalities of) state descriptions over $\mathbf{O}$, all asserted finite in $\delta'(\mathbf{O})$.

The subformula $\iota(\mathbf{O},Y)$ represents constraints of the form $\mathsf{v}=\mathsf{w}$ and $\mathsf{v}>\mathsf{w}$, where $\mathsf{v}$ and $\mathsf{w}$ are either ``infinite'' state descriptions over $\mathbf{O},Y$, $k$-fold successors of such  state descriptions, or aleph-numbers. In view of the isomorphism between $\langle \mathbb{N};0,s,> \rangle$ and 
$\langle \{\aleph_k\}_{k \in \mathbb{N}};\aleph_0,s,>\rangle$, we can construe these conjuncts as describing relations on natural numbers. As these relations are a special case of linear sets and can thus be encoded as in (\ref{semi-lin-6}), we once again run the argument to ``merge'' the equations for $S \wedge Y$ and $S \wedge \neg Y$ into a single equation for $S$, provided all of these are infinite. (If only one of $S \wedge Y$ and $S \wedge \neg Y$ is infinite, then that equation remains as before since $S$ will have the same cardinality.) 

The resulting formula will in general involve addition. But as discussed further in Appendix \ref{app:infinityaddition}, since all variables are infinite we can eliminate all explicit sums, using equivalences such as $\mathsf{t} = \mathsf{v}+\mathsf{w} \Leftrightarrow  (\mathsf{t}=\mathsf{v} \wedge \mathsf{v} \geq \mathsf{w})\vee (\mathsf{t}=\mathsf{w} \wedge \mathsf{w} \geq \mathsf{v}) $.


\begin{theorem} The definable relations on cardinal numbers in $\mathsf{MSO}(\#)$ are exactly the same as those definable in additive first-order logic. 
\end{theorem}
In effect, we have shown how to reduce a sentence $\varphi(\mathbf{P})$ in $\mathcal{L}_\#^2$ to an additive first-order formula $\alpha(\mathbf{x})$, with a variable $x_i$ in $\mathbf{x}$ corresponding to each state-description over $\mathbf{P}$. Moreover, $\varphi(\mathbf{P})$ is satisfiable if and only if $\exists \mathbf{x}.\alpha(\mathbf{x})$ is true in $\langle C_{\aleph_\omega};+ \rangle$. Thus, from decidability of the elementary theory of $\langle C_{\aleph_\omega};+ \rangle$ (see Theorem \ref{thm:decidablecardinal} in Appendix \ref{app:infinityaddition}) we obtain: 
\begin{corollary} \label{cor:mso} The validity problem for $\mathsf{MSO}(\#)$ is decidable.\vspace{-1ex}
\end{corollary}

\vspace{0.5ex}

\subsection{Addendum: Notions of Definability}\label{twodefinitions} In the preceding results for the systems $\mathsf{MFO}(\#)$ and $\mathsf{MSO}(\#)$ we adopted an intuitive notion of what might be called \emph{direct definability} of arithmetical predicates by formulas of our combined count logics. Stated more precisely, recall that this notion works as follows:

\defdef*

Adopting this definition has some general consequences. The following observation holds for all logical systems considered in this paper, whether first-order or second-order, and more generally, for \emph{any counting logic having a decidable model-checking problem}. 

\begin{fact} All directly definable arithmetical predicates are decidable. \label{fact1} \end{fact}

\begin{proof}[Proof sketch] Suppose $\varphi$ defines $n$-ary numerical predicate $\mathcal{V}$. To check whether a given tuple $[m_1,\dots,m_n]$ of numbers is in $\mathcal{V}$ it suffices to consider the (unique) model with these cardinalities for the ``zones'' given by state descriptions $S_1,\dots,S_n$, with all other zones empty. By assumption, it is decidable  whether $\varphi$ is true in this model. 
\end{proof}


 However, other notions of definable can be motivated as well. For instance:
 \begin{definition}[Trace definability] \label{def:trace} Let us say that $\varphi$ \emph{trace defines} $\mathcal{V}$ just in case the ``trace'' of $\varphi$ is $\mathcal{V}$: for every tuple $[m_1,\dots,m_n]$ in $\mathcal{V}$ there is \emph{some} model $\mathcal{M}$ such that $\mathcal{M} \vDash \varphi$ with $m_i = |S_i|_{\mathcal{M}}$ for all $i \leq n$. \end{definition} 
 In contrast with direct definability, a trace-definable arithmetical predicate $P$ holds if some finite model makes its defining formula $\varphi$ true, where the non-coding state descriptions may now define non-empty subsets of the domain in a way that is crucial to the truth of $\varphi$. Given this, it is easy to see the following for our logics, and again the larger class mentioned earlier.

\begin{fact} All trace-definable arithmetical predicates are recursively enumerable.  \label{fact2} \end{fact}

The distinction between these two notions of definability has no effect on the above arithmetical definability results for the counting logics $\mathsf{MFO}(\#)$ and $\mathsf{MSO}(\#)$. The reason is, essentially, that the first-order theory of  $\langle\mathbb{N};+\rangle$ enjoys \emph{quantifier elimination}, where unbounded existential quantifiers over numbers can be replaced by bounded ones in suitable normal forms (systems of modulo-equalities, or semilinear forms). For instance, existentially quantifying over a semi-linear predicate in some variables again has an equivalent semi-linear definition. Put differently, recursive enumerability reduces to decidability in additive arithmetic. 

However, the two notions of definability diverge when we move to richer counting logics that also involve multiplication, where quantifier elimination no longer holds. We now turn to an extension of our counting logics where multiplication enters.

\section{Counting Sequences} \label{section:sequences} We have so far considered a base monadic system, $\mathsf{MFO}^\phi(\#)$, and a second-order extension, $\mathsf{MSO}^\phi(\#)$, both of which are essentially restricted to reasoning about \emph{sums} of numbers. The same theme carries over to the setting of infinite models, with $\mathsf{MFO}(\#)$ and $\mathsf{MSO}(\#)$. 
These previous systems involve unary variable binding operators, which count \emph{sets} of objects. But it is also very natural from a logical point of view to count \emph{sequences} of objects. Indeed, polyadic quantifiers are  ubiquitous across natural language; cf. \S\ref{section:gq} below. We now consider such an extension, essentially moving from sets to products of sets. We would like to understand what additional arithmetical capacity this affords. 

Let $\mathcal{L}^1_{\sharp}$ be the first-order monadic language with polyadic counting terms $\sharp_{\mathbf{x}}\varphi$, where $\mathbf{x} = x_1,\dots,x_k$ is a sequence of variables, which may appear in $\varphi$.  Then: \vspace{0.5ex}\begin{eqnarray*} \mathcal{M},s \vDash \sharp_{\mathbf{x}}\varphi \succsim \sharp_{\mathbf{y}}\psi & \mbox{ iff }  & |\{\mathbf{d} \in D^n: \mathcal{M},s^{\mathbf{x}}_{\mathbf{d}} \models \varphi\}| \geq |\{\mathbf{d} \in D^m: \mathcal{M},s^{\mathbf{y}}_{\mathbf{d}} \models \psi\}|. \end{eqnarray*} \vspace{0.5ex} Over finite models let us call the resulting system $\mathsf{MFO}^{\phi}(\sharp)$, and $\mathsf{MFO}(\sharp)$ for the general case.

It is known that polyadic counting over full first-order logic is more expressive than unary counting (i.e., our $\mathsf{FO}^\phi(\#)$; see, e.g., \citealt{Otto}, Example 4.13). In our monadic fragment this is particularly dramatic, as shown by the following example. 
\begin{example} Consider the earlier `Many $Q$s are $P$', defined in (\ref{eq:many}) and repeated here: \begin{eqnarray*}
\#_x\big(P(x) \wedge Q(x)\big) \times \#_x \top & \succ & \#_x P(x) \times \#_x Q(x).
\end{eqnarray*} We can express this as follows: $$ \sharp_{x,y} \big(P(x) \wedge Q(x)\big) \succ \sharp_{x,y} \big(P(x) \wedge Q(y)\big).  $$ In a finite model, the term $\sharp_{x,y} \big(P(x) \wedge Q(x)\big)$ gives us the product of the model's total cardinality and the region in which $P$ and $Q$ both hold, while the term on the right gives us the product of cardinalities for $P$ and $Q$. \end{example}
 Evidently $\mathsf{MFO}^{\phi}(\sharp)$ incorporates some reasoning about \emph{multiplication}. 
\begin{example} We can encode Pythagorean triples of cardinalities for state-descriptions $S_1,S_2,S_3$, i.e., the statement that $|S_1|^2 + |S_2|^2 = |S_3|^2$:
$$ \sharp_{x,y}\big((S_1(x) \wedge S_1(y)) \vee (S_2(x) \wedge S_2(y))\big) \approx \sharp_{x,y}(S_3(x) \wedge S_3(y)). $$ The multiplication again comes from taking products, while the addition in this example arises from disjunction, just as in our initial system $\mathsf{MFO}^\phi(\#)$. \end{example} 
The next examples involves a different combination of multiplication and addition:
\begin{example} \label{ex:complicatedone} This sentence expresses the constraint that $|P|\times 2 = |Q|^3 + 2$. 
\begin{gather*}
     \exists x,y.\Big(\neg (Q(x) \vee Q(y)) \wedge x \neq y \wedge \sharp_{u,v}\big( P(u) \wedge (v=x \vee v=y)  \big)\approx \\ \sharp_{z,u,v} \big((Q(z) \wedge Q(u) \wedge Q(v)) \vee z=u=v=x \vee z=u=v=y   \big)   \Big)
\end{gather*} Note the use of variables $x,y$ for both $\sharp$ terms. In the first, $\sharp_{u,v}\big( P(u) \wedge (v=x \vee v=y)  \big)$, we simply want to multiply the cardinality of $P$ by $2$---the fact that $Q$ holds of neither $x$ nor $y$ does not matter here. In the second $\sharp$ term we consider all triples of points satisfying $Q$, i.e., $|Q|\times|Q|\times|Q|$-many points, and we add two points $x,y$---here it is important that $Q$ holds of neither, since this guarantees that we indeed add $2$ to the product $|Q|^3$ in the second $\sharp$ term. 
\end{example}
For a visualization, see Figure \ref{fig:multiply}. 
All of these examples certainly go beyond what can be expressed in $\mathsf{MFO}^\phi(\#)$, and even $\mathsf{MSO}^\phi(\#)$. What is the full scope of $\mathsf{MFO}^{\phi}(\sharp)$?

\begin{figure} \centering 
   \begin{tikzpicture}
\draw [black,fill=blue!05,label=m2] (-3,.5) ellipse (2.2cm and .3cm);
\filldraw [label=p1] (-3.25,.5) circle (1.25pt);
\filldraw [label=p2] (-2.75,.5) circle (1.25pt);
\filldraw [label=p3] (-2.25,.5) circle (1.25pt);
\filldraw [label=p4] (-1.75,.5) circle (1.25pt);
\filldraw [label=p5] (-1.25,.5) circle (1.25pt);
\filldraw [label=p6] (-3.75,.5) circle (1.25pt);
\filldraw [label=p7] (-4.25,.5) circle (1.25pt);
\filldraw [label=p8] (-4.75,.5) circle (1.25pt);
\filldraw [label=x] (-3.5,3) circle (1.25pt);
\filldraw [label=y] (-2.5,3) circle (1.25pt);
\node at (-4,3) {$x$};
\node at (-2,3) {$y$};
\node at (0,1.75) {\huge{}$=$};
\node at (-5.5,.5) {$P$};
\node at (1.25,.5) {$Q$};
\node at (4.8,2.7) {copy};
\node at (4.8,2.3) {of $Q$};
\draw[-,thick,blue!50] (-3.25,.5) -- (-3.5,3);
\draw[-,thick,blue!50] (-2.75,.5) -- (-3.5,3);
\draw[-,thick,blue!50] (-2.25,.5) -- (-3.5,3);
\draw[-,thick,blue!50] (-1.75,.5) -- (-3.5,3);
\draw[-,thick,blue!50] (-1.25,.5) -- (-3.5,3);
\draw[-,thick,blue!50] (-3.75,.5) -- (-3.5,3);
\draw[-,thick,blue!50] (-4.25,.5) -- (-3.5,3);
\draw[-,thick,blue!50] (-4.75,.5) -- (-3.5,3);
\draw[-,thick,blue!50] (-3.25,.5) -- (-2.5,3);
\draw[-,thick,blue!50] (-2.75,.5) -- (-2.5,3);
\draw[-,thick,blue!50] (-2.25,.5) -- (-2.5,3);
\draw[-,thick,blue!50] (-1.75,.5) -- (-2.5,3);
\draw[-,thick,blue!50] (-1.25,.5) -- (-2.5,3);
\draw[-,thick,blue!50] (-3.75,.5) -- (-2.5,3);
\draw[-,thick,blue!50] (-4.25,.5) -- (-2.5,3);
\draw[-,thick,blue!50] (-4.75,.5) -- (-2.5,3);
\draw [black,fill=red!06,label=m2] (3,.5) ellipse (1.5cm and .3cm);
\draw (3,3) ellipse (1.5cm and .3cm);
\filldraw  (3.35,.5) circle (1.25pt);
\filldraw  (2.65,.5) circle (1.25pt);
\filldraw  (4.05,.5) circle (1.25pt);
\filldraw  (1.95,.5) circle (1.25pt);
\filldraw  (3.35,3) circle (1.25pt);
\filldraw  (2.65,3) circle (1.25pt);
\filldraw  (1.95,3) circle (1.25pt);
\filldraw  (4.05,3) circle (1.25pt);
\draw[-,thick,red!50] (3.35,.5) -- (3.35,3);
\draw[-,thick,red!50] (3.35,.5) -- (2.65,3);
\draw[-,thick,red!50] (3.35,.5) -- (4.05,3);
\draw[-,thick,red!50] (3.35,.5) -- (1.95,3);
\draw[-,thick,red!50] (2.65,.5) -- (3.35,3);
\draw[-,thick,red!50] (2.65,.5) -- (2.65,3);
\draw[-,thick,red!50] (2.65,.5) -- (4.05,3);
\draw[-,thick,red!50] (2.65,.5) -- (1.95,3);
\draw[-,thick,red!50] (1.95,.5) -- (3.35,3);
\draw[-,thick,red!50] (1.95,.5) -- (2.65,3);
\draw[-,thick,red!50] (1.95,.5) -- (4.05,3);
\draw[-,thick,red!50] (1.95,.5) -- (1.95,3);
\draw[-,thick,red!50] (4.05,.5) -- (3.35,3);
\draw[-,thick,red!50] (4.05,.5) -- (2.65,3);
\draw[-,thick,red!50] (4.05,.5) -- (4.05,3);
\draw[-,thick,red!50] (4.05,.5) -- (1.95,3);
\node at (0,-.6) {$\exists x,y.\big(x \neq y \wedge  \sharp_{u,v}(P(u) \wedge (v=x \vee v=y)) \approx \sharp_{u,v} (Q(u) \wedge Q(v))  \big)$};
\end{tikzpicture} 
    \caption{A visualization of the formula expressing that  $2$ times the number of $P$ points (blue) is exactly the number of $Q$ points squared (red), i.e., $ |P|+|P| = |Q|^2$. The formula asserts that the blue lines are equal in number to the red lines. This is a simplified version of Example \ref{ex:complicatedone} and of the more general construction in Lemma \ref{lemma:diophantus}. In the case pictured, $|P|=8$ and $|Q|=4$.}\label{fig:multiply}
\end{figure}

\subsection{Diophantine Inequalities}

To start, consider any polynomial inequality  \begin{equation} m_1(\mathbf{v}) + \dots + m_k(\mathbf{v}) \geq m_1'(\mathbf{v}) + \dots + m_j'(\mathbf{v}), \label{in-diophantus}\end{equation} where $m_1,\dots,m_k,m'_1,\dots,m'_j$ are all monomials in variables $\mathbf{v} = \mathsf{v}_1,\dots,\mathsf{v}_n$. Each monomial $m(\mathbf{v})$ is of the form $a \mathsf{v}_1^{e_1}\dots \mathsf{v}_n^{e_n}$, with $a,e_1,\dots,e_n$ all natural numbers and $a>0$. 

We would like to show that sentences in $\mathsf{MFO}^{\phi}(\sharp)$ express all the Diophantine inequalities of type (\ref{in-diophantus}). The first result generalizes the observations above:
\begin{lemma} Every Diophantine inequality can be expressed in $\mathsf{MFO}^{\phi}(\sharp)$. \label{lemma:diophantus}
\end{lemma}
\begin{proof} Let $a^*$ be the sum of all the coefficients of $m_1,\dots,m_k,m_1',\dots,m_j'$, and let $e^*$ be the maximum over all the sums $\sum_{i \leq n} e_i$. Then our sentence will take the form: \begin{equation} \exists \mathbf{z}.\Big(\mathsf{diff}(\mathbf{z}) \wedge \sharp_{\mathbf{x}} \bigvee_{1\leq i \leq k} \alpha_i \succsim \sharp_{\mathbf{x}} \bigvee_{1 \leq i \leq j} \beta_i\Big), \label{eq:} \end{equation}\vspace{0.5ex} with $\mathbf{z} = z_1,\dots,z_{a^*}$ and $\mathbf{x} = x_0,x_1,\dots,x_{e^*}$. We need to ensure that each tuple of values for $\mathbf{x}$ satisfies at most one of the $\alpha_i$ formulas (and same for the $\beta_i$ formulas), and that each $\alpha_i$ contributes exactly $a \times |S_1|^{e_1}\times \dots \times |S_n|^{e_n}$ to the overall sum, when $m_i(\mathbf{v}) = a \mathsf{v}_1^{e_1}\dots \mathsf{v}_n^{e_n}$. To that end let $\alpha_i$ be the conjunction of the following formulas:\vspace{0.5ex} \begin{enumerate}[label=(\roman*)]
    \item $S_1(x_1) \wedge \dots \wedge S_1(x_{e_1})$, with a similar conjunct for each of $S_2,\dots,S_n$, predicating $e_2,\dots,e_n$ variables, respectively.\vspace{0.5mm}
    \item For any remaining $x$ up to $x_{e^*}$, include a conjunct $x= x_1$.\vspace{0.5mm}
    \item As a final conjunct: $(x_0 = z_{i_1} \vee \dots \vee x_0 = z_{i_a})$, for $a$ variables $z_{i_1},\dots,z_{i_a}$ from among $z_1,\dots,z_{a^*}$, guaranteed unique to this disjunct $\alpha_i$. \label{third-diop}
\end{enumerate}\vspace{0.5ex}

The last conjunct \ref{third-diop} ensures that each $\alpha_i$ contributes $a$ multiplied by the number of tuples satisfying $|S_1|^{e_1}\times \dots \times |S_n|^{e_n}$, since each such tuple appears with exactly $a$ (unique) values of $x_0$. Defining the $\beta_i$s analogously produces a formula whose models capture precisely the same solutions as (\ref{in-diophantus}), provided the sum of these numbers is at least $a^*$. 

There may of course be solutions that together add up to less than $a^*$, in which case (\ref{eq:}) will fail; however, there will be at most finitely many. For each such solution $b_1,\dots,b_n$ we can simply disjoin (\ref{eq:}) with the statement $|S_1|=b_1 \wedge \dots \wedge |S_n|=b_n$, the latter being easily definable (even in $\mathsf{MFO}^\phi(\#)$). 
\end{proof}

Conjunctions of inequalities in (\ref{in-diophantus}) give us the well studied class of Diophantine equations. The Matiyasevich-Robinson-Davis-Putnam (MRDP) Theorem shows that there can be no decision procedure to determine whether a given Diophantine equation has a solution. So:
\begin{proposition} The satisfiability problem for $\mathsf{MFO}^{\phi}(\sharp)$---that is, the problem of determining whether a given formula in $\mathcal{L}^1_{\sharp}$ has a finite model---is undecidable. 
\end{proposition} 
Moreover, while it is possible to enumerate the satisfiable formulas in an effective way, the valid sentences of $\mathsf{MFO}^{\phi}(\sharp)$---those whose negations define equations with no solutions---are not computably enumerable. Therefore:
\begin{proposition} $\mathsf{MFO}^{\phi}(\sharp)$ is not computably axiomatizable. 
\end{proposition}


\subsection{Normal Forms} In the direction of a normal form for $\mathsf{MFO}^{\phi}(\sharp)$, a first observation is that a version of the invariance principle (\ref{invariance-main}) from \S\ref{section:normalform} holds in the present setting as well:\vspace{1mm} \begin{equation*} \exists \mathbf{x} \big(\alpha^{\mathbf{y}}(\mathbf{x}) \wedge \varphi(\mathbf{x})\big) \rightarrow \sharp_{\mathbf{x}} \big( \alpha^{\mathbf{y}}(\mathbf{x}) \wedge \varphi(\mathbf{x})\big) \approx \sharp_{\mathbf{x}} \alpha^{\mathbf{y}}(\mathbf{x}),
    \label{invariance:multiple}
\end{equation*}\vspace{1mm} where now $\alpha^{\mathbf{y}}(\mathbf{x})$ is a complete description of the list $\mathbf{x}$ of variables (relative to $\mathbf{y}$). 

By an analogous argument we can then show that every formula in $\mathcal{L}^1_{\sharp}$ is equivalent to one with no embedded $\sharp$ comparisons. Similar to Lemma \ref{lemma:depth1}, any $\sharp$ comparison embedded within a $\sharp$ comparison can be removed at the expense of a string of existential quantifiers.  

As in Theorem \ref{normal}, we have: 

\begin{theorem} The definable sets of $\mathsf{MFO}^{\phi}(\sharp)$ are exactly those definable by Boolean combinations of Diophantine inequalities. \label{thm:mfosharp}
\end{theorem}
\begin{proof}[Proof Sketch] We want to show more generally that every formula of $\mathcal{L}^1_{\sharp}$ in free variables $\mathbf{y}$ is equivalent to a disjunction \begin{equation} \bigvee \big(\alpha(\mathbf{y}) \wedge (\sigma)_{\alpha(\mathbf{y})}\big),
    \label{eq:variable-normal}
\end{equation} where $\alpha(\mathbf{y})$ ranges over possible descriptions of $\mathbf{y}$, and $\sigma$ is a conjunction of (strict and weak) Diophantine inequalities (\ref{in-diophantus}). Specifically, each disjunct is such that, for all $s$: \begin{eqnarray} \mathcal{M},s \vDash \alpha(\mathbf{y}) \wedge  (\sigma)_{\alpha(\mathbf{y})} & \Rightarrow & \mathcal{M}\mbox{ satisfies the description }\sigma. \label{eq:normalform-equiv1}
\end{eqnarray}

Here in the preliminary reduction to formulas with  single-count depth, the earlier method appealing to Automorphism Invariance works essentially as before, be it that the invariance now refers to tuples of objects in regions, leading to the extraposition of cases involving a prefix of one or more existential quantifiers for tuples of objects.

Next, as in the proof of Theorem \ref{normal}, we show that every formula is equivalent to one of the form (\ref{eq:variable-normal}) satisfying (\ref{eq:normalform-equiv1}) by induction on the quantifier depth of formulas. In the base case, with no quantifiers and just a single $\sharp$-comparison, our normal forms will be disjunctions of conjunctions $\alpha(\mathbf{y}) \wedge (\sigma)_{\alpha(\mathbf{y})}$ where $(\sigma)_{\alpha(\mathbf{y})}$ takes the form: \begin{equation} \bigwedge \Big(  \sharp_{\mathbf{x}} \bigvee \alpha_i^{\mathbf{y}}(\mathbf{x}) \succsim \sharp_{\mathbf{x}} \bigvee \alpha_j^{\mathbf{y}}(\mathbf{x})  \Big) . \nonumber
\end{equation} It is straightforward to check that each such disjunct corresponds to a set of Diophantine inequality constraints satisfying (\ref{eq:normalform-equiv1}). 

The inductive case is essentially as in the proof of Theorem \ref{normal}, reducing to the claim that normal forms in (\ref{eq:variable-normal}) are closed under existential quantification. The Diophantine form arises as follows here: products of sizes of regions are needed because of the multiple counting, the exponents stay fixed by the arity of the multiple count operators in the formula, and apart from this, a fixed tuple of objects distributed over regions just contributes fixed numerical factors which do not go beyond the Diophantine format.
\end{proof}

\begin{remark} \label{remark:trac} Recall the notion of trace definability given in Def. \ref{def:trace}. In light of Theorem \ref{thm:mfosharp} and the MRDP Theorem, it is easy to see that the trace definable sets in $\mathsf{MFO}^\phi(\sharp)$ are exactly the recursively enumerable sets. In other words, in light of Fact \ref{fact2}, $\mathsf{MFO}^\phi(\sharp)$ is already as expressive as it can be when it comes to systems whose model-checking problem is decidable. 
\end{remark}

\subsection{Second-Order Extensions} The jump from $\mathsf{MFO}^\phi(\#)$ to $\mathsf{MSO}^\phi(\#)$, incorporating second-order quantification, was relatively minor. Arithmetically speaking, it simply allowed ``filling out'' the class of linear inequalities we were able to encode. Given that $\mathsf{MFO}^{\phi}(\sharp)$ already defines all Diophantine sets (Lemma \ref{lemma:diophantus}), and predicates defined in any of these systems will be decidable (Fact \ref{fact1}), it is an interesting question what class of arithmetical predicates will be defined by the second order extension $\mathsf{MSO}^{\phi}(\sharp)$.\footnote{\emph{Note:} The current section corrects an earlier statement in the published version of this paper.}  

Recall that the trace definable sets in $\mathsf{MSO}^{\phi}(\sharp)$ will again be the recursive enumerable sets (Remark \ref{remark:trac}). The more difficult question is what the \emph{directly definable} predicates are. In particular, how close might $\mathsf{MSO}^{\phi}(\sharp)$ take us toward the set of all decidable arithmetical predicates? It is well-known that the decidable arithmetical predicates cannot have a reursively enumerable set of indices (a simple diagonal argument will show this), and this precludes the existence of a syntactic format that captures all decidable arithmetical properties. Still there are some natural syntactic subclasses. One of the most well-studied is the first-order language $\Delta_0$ of `bounded arithmetic', which contains all equalities between polynomial terms, Boolean operations, and bounded numerical quantifiers of the form $\exists x < t. \varphi(x, \textbf{y})$ where $t$ is a polynomial term. Here are a few observations and conjectures.

First note that the sizes of the subsets of the object domain over which a second-order quantifier ranges are bounded by the sum of the sizes for the coding state descriptions of the variables \textbf{y}, as we can model-check the definition for a given predicate in the unique model described earlier where are all non-coding state descriptions have an empty denotation (recall the proof of Fact \ref{fact1}). This suggests the following:

\begin{conjecture} The arithmetical predicates that are directly definable in $\mathsf{MSO}^\phi(\sharp)$ can be defined in $\Delta_0$.  \end{conjecture}

Presumably, we need only a restricted fragment of the full language $\Delta_0$.

It seems clear that venturing further into $\Delta_0$ requires expressive resources beyond $\mathsf{MSO}(\sharp)$. For instance, an arithmetical predicate $\exists x < y^{2}: x + y > z$ may not be bounded in $y + z$ because of its quadratic term. To define a predicate with a bound like this, we seem to need a non-monadic second-order counting logic beyond $\mathsf{MSO}(\sharp)$ which can also quantify over \emph{binary relations}, giving us the required exponent 2 without having to increase the domain size of a model. Higher fixed exponents in polynomial terms will require  higher predicate arities.


We leave a precise syntactic determination of the arithmetical defining power of $\mathsf{MSO}^\phi(\sharp)$ and its higher-order extensions as an interesting problem for further investigation.

\subsection{Infinitary Counting} For both systems, $\mathsf{MFO}^\phi(\sharp)$ and $\mathsf{MSO}^\phi(\sharp)$, we can consider their more general versions, $\mathsf{MFO}(\sharp)$ and $\mathsf{MSO}(\sharp)$, where we allow infinite models. It is shown in Appendix \ref{app:infinityaddition} that, similar to the purely additive case, cardinal arithmetic (say, up to $\aleph_{\omega}$) with addition and multiplication separates cleanly into the finitary and infinitary components, with the infinitary component effectively reducing to the first-order theory of $\langle \mathbb{N};> \rangle$.  Meanwhile, $\mathsf{MFO}(\sharp)$ and $\mathsf{MSO}(\sharp)$ can also express multiplicative relationships among cardinal numbers. 

\section{An Alternative Route: Explicit Arithmetical Operators} \label{section:explicitarithmetic}

The sequence of systems so far studied was motivated primarily by natural operations in logic, viz. second-order quantification and polyadicity. We were then able to calibrate the arithmetical content of these operations over our base monadic system $\mathsf{MFO}(\#)$. Another approach to extending $\mathsf{MFO}(\#)$ in the spirit of logic and counting would rather strengthen the counting component in natural ways, in particular, by   allowing complex terms built directly out of arithmetical operations. Instead of comparisons involving terms like $\#_x\varphi$ we might allow comparing, for instance, sums of terms $\#_x\varphi + \#_y\psi$, and in general allow inequalities $\mathbf{t}_1 \succsim \mathbf{t}_2$ between complex terms. We can then study the consequences of different choices of complex term building operators. Most salient are of course addition and multiplication, and for these it turns out that, speaking abstractly, we would have arrived at the same systems. 

\subsection{Addition}

Let $\mathsf{MFO}(\#,+)$ be the system that results by allowing arbitrary finite sums of basic $\#$ terms. In other words we allow terms of the form $\#_{x_1}\varphi_1 + \dots + \#_{x_n}\varphi_n$. We already know that $\mathsf{MSO}(\#)$ can express all such inequalities. Conversely, the normal form result for $\mathsf{MSO}(\#)$ by means of linear inequalities shows that this system and  $\mathsf{MFO}(\#,+)$ are in fact equally expressive when it comes to defining relations on cardinal numbers.

Note also that the numerical reasoning involved in Fourier-Motzkin  (Remark \ref{rem:fourier}) can be transcribed into this language without any further ado. For instance, we can encode the crucial step \ref{fm-3} by a simple scheme: $$\big(|S_1| \succsim |S_2| + |S_3| \wedge |S_4|+|S_3| \succsim |S_2|\big) \rightarrow |S_1|+|S_4| \succsim |S_2|+|S_2|.$$ Thus, similar to analogous work on (rational) linear programming \citep{Fagin}, we could codify the steps of the algorithm into axioms of a formal system.  

\subsection{Multiplication}

From an arithmetical point of view, it is natural to allow arbitrary finite \emph{products} of basic $\#$-terms as well. How would such a system relate to our systems for counting sequences, such as $\mathsf{MFO}^\phi(\sharp)$ or $\mathsf{MSO}^\phi(\sharp)$? Needless to say, if we had explicit multiplication and addition we would be able to encode all arithmetical relations. 

Similar to the case of $\mathsf{MSO}^\phi(\#)$, even without explicit addition we can simulate addition if we avail ourselves of second-order quantification. Indeed, let $\mathsf{MSO}^\phi(\#,\times)$ be the second-order monadic fragment with products of $\#$-terms (in fact, binary products suffice). Echoing observations dating back to \cite{Skolem}, we can thereby encode arbitrary Diophantine inequalities. Indeed, consider any such
\begin{equation} m_1(\mathbf{v}) + \dots + m_k(\mathbf{v}) \geq m_1'(\mathbf{v}) + \dots + m_j'(\mathbf{v}), \label{in-diophantus-2}\end{equation} over variables $\mathbf{v}$ corresponding to state-descriptions over $\mathbf{P}$. To express (\ref{in-diophantus-2}) in $\mathsf{MSO}^\phi(\#,\times)$ we introduce $k+j$ predicate variables $X_1,\dots,X_{k+j}$ and consider the statement that all $X_i$ are disjoint but that $\#(X_1 \vee \dots \vee X_k) \succsim \#(X_{k+1} \vee \dots \vee X_{k+j})$. Each monomial $m_i(\mathbf{v})$ can clearly be expressed as a product of $\#$-terms (possibly using first-order quantification) in the original variables $\mathbf{P}$, so we set each of these equal to the corresponding term $\# X_i$. 
To define the same set of solutions as (\ref{in-diophantus-2}) (in the sense of Definition \ref{define:define}, so that the formula includes no free predicate variables), we existentially quantify all the variables $X_1,\dots,X_{k+j}$. 

We leave more fine-grained analysis of fragments of these systems (e.g., the purely first-order fragment) for future explorations. 

\subsection{Other Arithmetical Operations} Aside from addition and multiplication, we could naturally consider a host of other common arithmetical functions and relations. As an example, unlike addition and multiplication, \emph{exponentiation} does not trivialize in the infinitary setting. Indeed, whereas in the finitary setting exponentiation is definable from addition and multiplication (e.g., by G\"{o}del's famous $\beta$ function), the Generalized Continuum Hypothesis can already be stated succinctly in $\mathsf{MSO}(\#)$ with exponentiation:\vspace{0.5ex} $$\forall X,Y,Z.\big( |X|\approx \mathbf{2}^{|Y|} \wedge |Y|\succsim \aleph_0 \wedge |Z|\succ |Y| \rightarrow |Z| \succsim |X|\big),\vspace{0.5ex} $$ where $\mathbf{2}$ abbreviates a set with cardinality two. 
It could be illuminating to study the properties of such a system across different models of set theory.

Another natural example is the relation of \emph{divisibility}, which also arises in the study of natural language quantifiers and automata hierarchies (see \S\ref{section:automata} and in particular Proposition \ref{prop:aut}\ref{prop:div}). A non-trivial observation in the finite case (due to Julia Robinson) is that first-order logic with divisibility and the successor function already provide the full suite of arithmetically definable relations \citep{Robinson}. At the same time, the \emph{existential fragment} of Presburger Arithmetic with divisibility is known to be decidable \citep{Lipschitz}, which leaves open the possibility that some of our systems may too remain relatively well behaved (recall Remark \ref{remark:soc}). We defer these and further explicit arithmetical excursions for another occasion. 
\subsection{Interim Summary} 
In the present monadic setting, we assess each system's ability to reason about counting by analyzing the arithmetical content of its family of definable relations. All of these systems speak about unary predicates and their Boolean combinations, but we have been most interested in the abstract relations over \emph{cardinal numbers} that sentences in these systems can define, essentially taking cardinalities of state-descriptions (or more generally, non-overlapping predicates) as numerical variables. We have seen that the landscape here is quite rich, naturally calibrated by familiar first-order arithmetical languages. 

With this grasp of the pure monadic fragment, we now move on to consider well-behaved fragments employing \emph{relational} reasoning.

\section{Modal Logic of Binary Relations} \label{section:modal}

We started with adding counting operators to the full language of first-order logic, and found a system $\mathsf{FO}(\#)$ with very high complexity. We then moved our base level to monadic fragments, which were decidable and allowed us to see combinations of logic and counting at work in more controlled settings. Even so, many simple intuitive examples of reasoning with numerical aspects go further than this, and involve binary relations.

\begin{example} The well-known Pigeonhole Principle says that, if we put $n$ objects into $k < n$ boxes, then at least one box must contain two or more objects. For all particular values of $k, n$, this principle can be expressed in monadic first-order logic using unary predicates for boxes (recall (\ref{eq:fophp})). But for a generic formulation, we need to go to binary relations, which admit of the following elegant statement.  Consider any binary relation $R$ whose domain has a larger cardinality than its range. Then at least one object must have two or more predecessors in the relation. In formal notation, $\#_x\exists y. Rxy \succ \#_x\exists y. Ryx$ implies $\exists x. \exists^{\geq 2} y. Ryx$.
\end{example}

In this light, it makes sense to study count versions of  fragments of $\mathsf{FO}(\#)$ that allow for some reference to binary relations, though without running into the high complexity noted earlier with the full language $\mathcal{L}_\#$. To this end, we will explore some count versions  of {\it modal languages} in some detail, starting with a simplest case, and returning to further extensions suggested by the Pigeonhole Principle later. For the basic notions and results of modal logic needed in this section, we will refer to the literature at appropriate places.


\subsection{Language and Semantics}  The language of \emph{propositional modal logic with counting}, $\mathcal{L}_\#^{\mathsf{ml}}$, has a syntax defined inductively as follows:\vspace{0.1ex}
$$\varphi \quad :=  \quad p \;\; \mid \;\; \neg \varphi \;\; \mid \;\; \varphi \wedge \psi \;\;\mid \;\; \#\varphi \succsim \#\psi \;. \vspace{0.1ex}$$
The \emph{depth} of formulas is defined recursively as for our earlier logics, with standard clauses for atoms $p$ and Booleans, while $\mathsf{d}(\#\varphi \succsim  \#\psi) = \mbox{max}(\mathsf{d}(\varphi), \mathsf{d}(\psi)) + 1$.

The semantics of this language uses standard modal relational models $\mathfrak{M} = (W, R, V)$. At points in these models, we define truth of formulas, and term values in a mutual recursion. For a point $s$ we write $R_s = \{t: Rst\}$ for its  $R$-successors. 

Here are the two semantic key clauses: \vspace{.1in} 
\begin{itemize}
\item $\semantics{\#\varphi}^{\mathfrak{M}, s} = | R_s \cap \semantics{\varphi}^{\mathfrak{M}} |$ \vspace{.1in}
\item $\mathfrak{M}$, $s \models \#\varphi \succsim  \#\psi $ \, iff \, 
$\semantics{\#\varphi}^{\mathfrak{M}, s}  \geq \semantics{\#\psi}^{\mathfrak{M}, s}$ \vspace{.1in}
\end{itemize}
Given this, we define an \emph{existential modality} $\Diamond\varphi$ as $\#\varphi \succ \#\bot$, 
and using negation we can then also define its universal dual $\Box\varphi$. There is also some definability for the Booleans, as we saw with $\mathsf{MFO(\#)}$, but we will let this rest here. Call the resulting system $\mathsf{ML}(\#)$. As before, we denote the logic interpreted over finite models by $\mathsf{ML}^\phi(\#)$.

\begin{remark} As was the case with the variety of quantifiers in $\mathsf{MFO(\#)}$, there are also further natural counting modalities such as `in most successors', `in almost all successors', but we will not study their logic separately here.\end{remark}

As for expressive power, iterated counting in this simple modal language can produce non-trivial assertions. The reader might consider the formula $\#(\#\neg p \succ \# p) \succ \#(\#p \succsim \#\neg p)$, and determine what this says numerically, for instance, on finite trees. An example model is depicted in Figure \ref{fig:modal}. One can also enforce infinity of some sets of successors. 
E.g., the modal formula $\Diamond q \wedge \#(p \wedge \neg q) \approx \#(p \vee q)$ requires an infinity of successors satisfying $p$.

 \begin{figure} \centering 
   \begin{tikzpicture}
   
   
  
  \node (s1) at (0,0) [circle,draw=black,scale=0.5] {};
\node (s2) at (-1,1) [circle,draw=black,scale=0.5] {};
\node (s3) at (0,1) [circle,draw=black,scale=0.5] {};
\node (s4) at (1,1) [circle,draw=black,scale=0.5] {};
\node (s5) at (-2,2) [circle,draw=black,scale=0.5,fill=blue!50] {};
\node (s6) at (-1.5,2) [circle,draw=black,scale=0.5] {};
\node (s7) at (-1,2) [circle,draw=black,scale=0.5] {};
\node (s8) at (-.35,2) [circle,draw=black,scale=0.5] {};
\node (s9) at (.35,2) [circle,draw=black,scale=0.5] {};

\node (s11) at (1,2) [circle,draw=black,scale=0.5,fill=blue!50] {};
\node (s12) at (1.5,2) [circle,draw=black,scale=0.5,fill=blue!50] {};
\node (s13) at (2,2) [circle,draw=black,scale=0.5] {};

\path (s1) edge[->] (s2);
\path (s1) edge[->] (s3);
\path (s1) edge[->] (s4);

\path (s2) edge[->] (s5);
\path (s2) edge[->] (s6);
\path (s2) edge[->] (s7);

\path (s3) edge[->] (s8);
\path (s3) edge[->] (s9);

\path (s4) edge[->] (s11);
\path (s4) edge[->] (s12);
\path (s4) edge[->] (s13);
\end{tikzpicture} 
    \caption{An example model in which $\#(\#\neg p \succ \# p) \succ \#(\#p \succsim \#\neg p)$ holds at the root point. The points where $p$ holds are colored blue. The number of successors with more non-$p$ successors than $p$ successors is greater than the number of successors with at least as many $p$ successors as non-$p$ successors.}\label{fig:modal}
\end{figure}
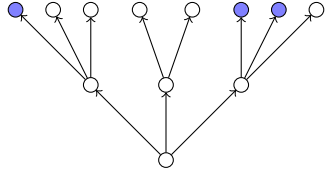


\subsection{Some Basic Model Theory}

Here are some invariance properties of our modal counting language that are useful for studying its expressive power.

A \emph{generated submodel} of a modal model is a submodel that is closed under taking $R$-successors (see, e.g., \citealt{Blackburn2001}). Given the ``forward-looking'' nature of the modal counting language along the order $R$, the following is a counterpart of the analogous invariance property for the basic modal language.

\begin{proposition} \mbox{(a)} Formulas of $\mathcal{L}_\#^{\mathsf{ml}}$ are invariant for generated submodels, \mbox{(b)} Terms of $\mathcal{L}_\#^{\mathsf{ml}}$  have the same value in generated submodels.
\end{proposition}

\begin{proof} A straightforward mutual induction on formulas and numerical terms. \end{proof}

The  Finite Depth Property of modal logic  also goes through, where ``finite depth'' refers to the following cut-off versions of our models: $\mathfrak{M}|_n, s$ is the submodel of $\mathfrak{M}$ consisting of only those points that can be reached from $s$ in at most $n$ relational steps.

\begin{proposition} For any model $\mathfrak{M}$ and $\mathcal{L}_\#^{\mathsf{ml}}$-formula $\varphi$, $\mathfrak{M}, s \models \varphi$ iff  $\mathfrak{M}|_{\mathsf{d}(\varphi)}, s \models \varphi$. 
\end{proposition}

The following invariance property refers to the standard \emph{tree unraveling}  of arbitrary relational models yielding  tree-like models in basic modal logic \citep{Blackburn2001}.

\begin{proposition} (a) Formulas of $\mathcal{L}_\#^{\mathsf{ml}}$ are invariant for tree unraveling under the map taking finite branches to their end-points, (b) Terms of $\mathcal{L}_\#^{\mathsf{ml}}$  have the same value under tree unraveling at points related by this same map.
\end{proposition}

\begin{proof} The proof is a straightforward induction on formulas and numerical terms, using the fact that the immediate successors of a branch in the tree are in one-one correspondence with the successors of the end-point in the original model.  \end{proof}

Finally define \emph{duplication} of a tree as making copies of all immediate successors of the root: each successor $t$ splits into $t_1, t_2$ each heading a disjoint copy of the original subtree at $t$. This construction can also be defined for models in general, and it can also be iterated going down the tree \citep{vanBenthem1998}, but we will not use this generality here.

\begin{proposition} $\mathcal{L}_\#^{\mathsf{ml}}$-formulas at the root  are invariant for  tree duplication. \end{proposition}

\begin{proof} The crucial cases are the numerical modal comparison statements $\#\varphi \succsim  \#\psi$  of our language, and these are obviously closed under taking multiples. \end{proof}

These invariance properties put limits on expressive power. For instance, our counting logic $\mathsf{ML(\#)}$ does not contain the well-known system of \emph{graded modal logic} that describes specific finite numbers of successors \citep{Fine}.

\begin{corollary} \label{cor:grade} The  graded modality \emph{`in at most one successor'} is not definable in $\mathsf{ML(\#)}$, \\as it is not invariant under tree duplication.
\end{corollary} In fact, $\mathsf{ML(\#)}$ and graded modal logic are incomparable in expressive power; see \cite{DEMRI2010233} for a more powerful system that subsumes both.

 \subsection{Bisimulation}  Behind the above preservation facts lies a general notion of {\it bisimulation} for $\mathcal{L}_\#^{\mathsf{ml}}$. For convenience, we  define this to be a standard modal bisimulation \citep{Blackburn2001} satisfying a further requirement of cardinality comparison between sets satisfying a structural property matching modal definability.
 
\begin{definition}[$\#$-bisimulation] \label{def:bisim} Let $Z$ be a  modal bisimulation between two points in two models $\mathfrak{M}$,  $\mathfrak{N}$ satisfying the usual conditions of (a) atomic harmony for proposition letters at $Z$-connected points, and (b) the standard back and forth clauses for matching relational successors of $Z$-connected points. 

Next, we define an auxiliary relation $\sim_Z$ between points in $\mathfrak{M}$ as follows: $x \sim_Z y$ iff for some $z \in \mathfrak{N}$: $x Z z$ and $y Z z$. The relation $\sim_Z$ in the model $\mathfrak{N}$ is defined likewise. Now, $Z$ is a \emph{$\#$-bisimulation} if the following comparative cardinality conditions hold. \vspace{0.5ex}

\begin{enumerate}[label=(\alph*)] \item \label{bisim-11} Whenever $s Z t$ and $X, Y$ are $\sim_Z$-closed sets of successors of $s$ with $X \succsim Y$ in our cardinality sense, then $Z[X] \cap R_t \succsim Z[Y] \cap R_t$.\footnote{Here, $Z[X]$ is the set $\{y \in \mathfrak{N}: xZy\mbox{ for some }x\in X.\}$.} \item \label{bisim-22}\vspace{1mm}
The same requirement in the opposite direction. \end{enumerate} 
\end{definition}

See Figure \ref{fig:bisim}. Note: in Clause \ref{bisim-11}, we mean that the sets $X, Y$ are $\sim_Z$-closed with respect to successors of $s$, not necessarily in the whole model $\mathfrak{M}$, and likewise in Clause \ref{bisim-22}. 

To understand what the map $Z[X] \cap R_t$ does, note that  $R_t - Z[X]$\, =\, $Z[R_s -X]$, given the $\sim_Z$-closedness of $X$ and the fact that $Z$ is a modal bisimulation. 
\begin{proposition} Formulas of $\mathcal{L}_\#^{\mathsf{ml}}$ are invariant for $\#$-bisimulation.
\end{proposition}
\begin{proof} The only non-routine part of the inductive argument is checking that $\#$-bisimulations preserve truth values of atomic formulas $\#\varphi \succsim \#\psi$ both ways for points $s, t$ with $s Z t$. 

To see this, first note that the set of all $\varphi$-successors of a point $s$ in a model $\mathfrak{M}$ satisfies the closure condition for $\sim_Z$ (using the inductive assumption on bisimulation invariance for the formula $\varphi$), and the same is true for the set of $\psi$-successors. We apply the comparison clause for our $\#$-bisimulation to these sets $X, Y$ and get that $Z[X] \, \cap \, R_t\, \succsim \,Z[Y] \,\cap \, R_t$ in $\mathfrak{N}$. 

Next, we show that $Z[X]$ is the set of successors of $t$ satisfying  $ \varphi$. By definition, each point in $Z[X]$ is $Z$-connected to some point in $X$, and so it satisfies $\varphi$ by the inductive hypothesis. Moreover, each point in $R_t - Z[X]$ was $Z$-connected to some point in $R_s - X$, and again by the inductive hypothesis, it then fails to satisfy $\varphi$. The same reasoning works for $Y$ and $\psi$. It follows that $\#\varphi \succsim \#\psi$ is true at $t$ in $\mathfrak{N}$.

Given the symmetry in the above comparative clause for a $\#$-bisimulation, the argument also works in the opposite direction.
\end{proof}

 \begin{figure} \centering 
   \begin{tikzpicture}
   
   \draw [draw=blue!90] (-.375,1) ellipse (.7cm and .25cm);
   
   \draw [draw=red!80] (.75,1) ellipse (.25cm and .25cm);
   
      \draw [draw=blue!90] (3.6,1) ellipse (1.1cm and .25cm);
      
\draw [draw=red!80] (5.2,1) ellipse (.25cm and .25cm);
  
  \node (s1) at (0,0) [circle,draw=black,scale=0.5] {};
\node (s2) at (-.75,1) [circle,draw=black,scale=0.5,fill=blue!50] {};
\node (s3) at (0,1) [circle,draw=black,scale=0.5,fill=blue!50] {};
\node (s4) at (.75,1) [circle,draw=black,scale=0.5] {};

\node (s0) at (0,-.7) {$\mathfrak{M}$};

\path (s1) edge[->] (s2);
\path (s1) edge[->] (s3);
\path (s1) edge[->] (s4);

  \node (t1) at (4,0) [circle,draw=black,scale=0.5] {};
\node (t2) at (2.8,1) [circle,draw=black,scale=0.5,fill=blue!50] {};
\node (t3) at (3.6,1) [circle,draw=black,scale=0.5,fill=blue!50] {};
\node (t4) at (4.4,1) [circle,draw=black,scale=0.5,fill=blue!50] {};
\node (t5) at (5.2,1) [circle,draw=black,scale=0.5] {};

\path (t1) edge[->] (t2);
\path (t1) edge[->] (t3);
\path (t1) edge[->] (t4);
\path (t1) edge[->] (t5);

\node (t0) at (4,-.7) {$\mathfrak{N}$};

\path (s1) edge[dotted,thick,bend right] (t1);
\path (s2) edge[dotted,thick,bend left] (t4);
\path (s3) edge[dotted,thick,bend left] (t4);
\path (s3) edge[dotted,thick,bend left] (t3);
\path (s3) edge[dotted,thick,bend left] (t2);
\path (s4) edge[dotted,thick,bend right] (t5);

\vspace{1ex}
\end{tikzpicture} 
    \caption{An ordinary modal bisimulation $Z$ between $\mathfrak{M}$ and $\mathfrak{N}$ is depicted by the dotted line. In both of these models the root point has four $\sim_Z$-closed sets of successors: the empty set, the whole set, and the two sets encircled in blue and in red. To be a $\#$-bisimulation (Definition \ref{def:bisim}), the same ordering of these sets by cardinality must hold in each, as it does here.}\label{fig:bisim}
\end{figure}
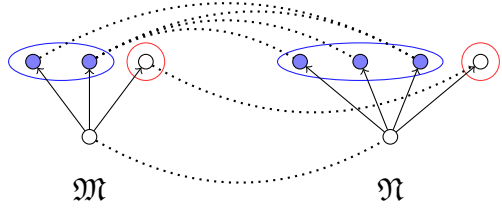

\vspace{1ex}

Bisimulation invariance can be used to show that certain notions are not definable. 

\begin{example} Infinity of a set of successors is not definable in $\mathsf{ML}(\#)$. Consider two models: one with a root and one successor, the other with a root with infinitely many successors. All proposition letters are true at all points. Connecting the two roots while also connecting all successors across the models is easily seen to be a bisimulation in the above sense.
\end{example}

As in general modal logic, converse results require additional conditions. We formulate two versions, starting  with a Hennessy-Milner result for ``image-finite'' models where each point has only finitely many relational successors.

\begin{proposition} On points in two image-finite modal relational models, the relation $E$ of $\mathcal{L}_\#^{\mathsf{ml}}$-equivalence is a $\#$-bisimulation.\end{proposition}

\begin{proof} By a standard argument from the modal literature, since $\mathcal{L}_\#^{\mathsf{ml}}$ contains the basic modal language, we have that $E$ is an ordinary modal bisimulation. 

Now for the set-comparison clause.
Start with $s E t$. We first show that any  $\sim_E$-closed set $X$ of successors of $s$ 
is definable among the successors of $s$. This follows by a well-known model-theoretic definability argument if we can  show that this set is closed under $\mathsf{ML}(\#)$-equivalence in the finite set of successors of $s$. But the latter fact can be seen as follows. Suppose that in $\mathfrak{M}$, $x \in X$ is  $\mathcal{L}_\#^{\mathsf{ml}}$-equivalent to $x'$ in $R_s$. By the ordinary forth clause for a modal bisimulation, $x$ is $E$-related to some $u$ in $R_t$ in $\mathfrak{N}$. But then $x'$, too, is $E$-related to $u$, and by  the assumed $\sim_E$-closure,
$x'$ must be in $X$. 

Next, assume that $|X| \geq |Y|$ is true in $\mathfrak{M}$ at $s$. Given the preceding observation, this  shows in the truth of some formula $\#\alpha \succsim \#\beta$ at $s$ where $\alpha$ defines $X$ and $\beta$ defines $Y$. Given the definition of $E$, this formula will also be true at $t$ in $\mathfrak{N}$, and then it suffices to note, using the above definability in $\mathfrak{M}$ plus the inductive hypothesis, that the set of $\alpha$-successors of $t$ is just $E[X] \cap R_t$, and likewise for the $\beta$-successors. \end{proof}


Still, the common assumption of image-finiteness runs counter to the fact that $\mathsf{ML}(\#)$ can also compare infinite cardinalities among successors. To reach a perfect correspondence we can employ another device, passing to an  \emph{infinitary} modal language, allowing conjunctions and disjunctions of arbitrary sets of formulas. Call this language $\mathcal{L}_\#^{\infty\mathsf{ml}}$.

\begin{theorem} The following are equivalent: (a) There exists a $\mathsf{ML}(\#)$-bisimulation connecting $\mathfrak{M}, s$ to $\mathfrak{N}, t$, (b) $\mathfrak{M}, s$ and $\mathfrak{N}, t$ satisfy the same formulas of $\mathcal{L}_\#^{\infty\mathsf{ml}}$.
\end{theorem}
\begin{proof} The inductive proof of the invariance assertion from (a) to (b) is essentially as before. From (b) to (a), we can use the earlier reasoning for image-finite models almost literally, noting now that in specific models, we only have \emph{sets of} successors, which are ``small'' w.r.t. the class of formulas of $\mathcal{L}_\#^{\infty\mathsf{ml}}$, and then using the closure of the latter  language under arbitrary set conjunctions and disjunctions. \end{proof}

\begin{remark} \label{remark:bisim} The bisimulation analysis presented here puts the crucial count comparisons of $\#$-languages in the back and forth clauses by brute force. A more refined notion of bisimulation might directly relate the {\it counting procedures} that underlie the comparative cardinality judgments in the two models. We leave this here as a further desideratum.\end{remark}

For a first study of standard modal themes like bisimulation and frame correspondence for $\mathsf{ML}(\#)$ and its  extensions, see \cite{FuZh23b}.

\subsection{Normal Forms for $\mathsf{ML}(\#)$}  The modal counting language admits syntactic normal forms that combine standard normal forms for modal logic with the numerical equational normal forms that we found   for $\mathsf{MFO(\#)}$. The idea is to start with the earlier state descriptions, and then describe inductively, for successors at increasing distance, which types occur with which multiplicity. In what follows, we fix a finite vocabulary of proposition letters.

\begin{definition} A \emph{0-type} is a complete conjunction of literals. An \emph{n+1-type} is a conjunction of a 0-type and a complete set of inequalities describing a linear order on count terms   $\#T$ that describe all unions of $n$-types.
\end{definition}

The inductive step in this definition makes sense because it is easy to show inductively that the set of $n$-types is finite for each $n$. 

To understand this definition, note that  modal types record inductively which types of lower rank are present and absent   on successors of given points. $\mathsf{ML}(\#)$ merely enriches this  to more precise numerical information.
\vspace{0.5ex}
\begin{fact}
Each formula of $\mathsf{ML}(\#)$ with count depth k is equivalent to a  disjunction of k-types. \vspace{0.2ex}
\end{fact}

\begin{proof}
For $k = 0$, this is the  disjunctive normal form of propositional logic. For the case $k+1$, a formula of this depth is equivalent to a Boolean combination of proposition letters and formulas $\# \varphi \succsim \# \psi$ with $\varphi,\psi$ of depth at most $k$. By the inductive hypothesis,  these formulas are equivalent to disjunctions of $k$-types. So, the whole formula is equivalent to a disjunction of conjunctions of such statements, where negations of comparison atoms can be replaced by strict inequalities.  Thus, a certain number of comparisons between regions is already given, and all we need to do is replace this formula by the disjunction of all completions to fill in comparisons between all regions, which is possible by  the linearity of $\succsim$.
\end{proof}

Here are two points of comparison with the earlier normal forms for monadic counting logic. First, we cannot compress our modal  normal forms further to depth 1 as we did for $\mathsf{ML}(\#)$ in \S\ref{section:basicmonadic}. Their simplicity is rather in that each level of counting refers to points farther away in the modal accessibility structure, so, intuitively, nested count information refers to different positions. Next, normal forms for the monadic counting language are ``loose'' in that they do not necessarily contain complete information about all regions in the model. The difference is slight, since, by the linearity of the cardinality order,  loose forms can be expanded to disjunctions of complete forms. (We implicitly invoked this fact already when stating Corollary \ref{undefinability} in \S\ref{section:basicmonadic}.) In the modal case, we could also allow loose forms, but we chose the complete version because of the following point.

\begin{remark} In modal logic, normal form results are often proved semantically, showing that a formula $\varphi$  of depth $k$ is equivalent to the disjunction of all $k$-types occurring in pointed models where $\varphi$ holds. This semantic argument involves a finite restriction of the notion of modal bisimulation to $k$ back-and-forth steps, and a similar argument can also be given with the more complex notion of bisimulation identified above.

\end{remark}

Normal forms are  related to \emph{Scott sentences} in infinitary languages, describing, up to suitable ordinal depths,  what syntactic types in the language are realized in a given model \citep{Scott1965}. Modal Scott sentences in $\mathcal{L}_\#^{\infty\mathsf{ml}}$  also include information on numbers of occurrences of types, and can define given pointed models up to bisimulation.







\begin{proposition} $\mathsf{ML(\#)}$ is decidable.  \label{prop:mldec}  \end{proposition}

\begin{proof} We show that SAT is decidable for normal forms. At depth 0, this is trivial. At depth $k + 1$, we proceed by means of the following pseudo-algorithm. 

Working outside in, we check, successively, that (a) the given atomic description for the root point is satisfiable, (b) the system of inequalities for the types of level $k$ occurring among the successors of the root is numerically satisfiable, say by the Fourier-Motzkin algorithm allowing infinite cardinalities as described in \S\ref{section:soinf}, and finally, (c) for each non-zero term that occurs in the solution in point (b), i.e., each relevant type of level $k$, we test for satisfiability again of these simpler types. 

This simple decision procedure is correct because stages (b) and (c) of the procedure are largely separate. If we can satisfy one of the types in stage (c) at all, then, by copying and taking disjoint subtrees, we can satisfy it at any desired number of successors for the root as described by the in equalities of stage (b). That no truth values are disturbed in this procedure is precisely the content of the earlier observations about invariance of modal counting formulas for generated submodels.
\end{proof}

\begin{remark} The preceding analysis is constructive, and it also contains information about the \emph{reasoning system} for $\mathsf{ML}(\#)$. We will not spell this out in further detail here.

Note also that, by results of \cite{DEMRI2010233} on a strictly larger modal fragment (their Thm. 1), Proposition \ref{prop:mldec} can in fact be sharpened to decidability in polynomial space. 
\end{remark}

\subsection{Language Extensions} 

Our modal language is still quite weak in some respects. For instance, unlike the system $\mathsf{MFO}(\#)$, it cannot talk about specific finite numbers of relevant objects: in our case, successors of the current point. In addition, we still lack the resources to express further features of the earlier-mentioned Pigeonhole Principle  
\begin{equation} \label{eq:rphp} \#_x\exists y. R(y,x) \succ \#_x\exists y. R(x,y) \rightarrow \exists x. \exists^{\geq 2} y. R(y,x). \end{equation}  
A modal rendering of the $\mathsf{FO}(\#)$-style  syntax with explicit variables  in this principle requires (a) numerical  \emph{graded} modalities, (b) both forward $\Diamond^{\rightarrow}$ and \emph{backward} $\Diamond^{\leftarrow}$ modalities along the ordering $R$, and also, (c) the notion of counting involved is not local to some current point, but involves a \emph{global} operator $\#_g$, referring to the whole domain. With these modal devices, the Pigeonhole Principle will come to look like this:
\begin{equation} \label{eq:rphp2}\#_g \Diamond^{\rightarrow}\top   \succ \#_g \Diamond^{\leftarrow}\top   \rightarrow E \Diamond^{\leftarrow, \geq 2} \top \end{equation}
where the ``existential modality'' $E$ is defined as having global count greater than 0. We briefly discuss these various extensions in turn. 

Adding graded modalities to  $\mathsf{ML}(\#)$ seems natural, so as to give the system the same expressive power as $\mathsf{MFO}(\#)$ over sets of successors. 
In fact, \cite{DEMRI2010233} present a generalized graded modal logic with Presburger-like constraints on types of successors which extends our system $\mathsf{ML}(\#)$. Given our analysis of additive arithmetic in terms of second-order counting logics, there may be connections here with the second-order version of $\mathsf{ML}(\#)$ mentioned below in Example \ref{ex:problems}.

Next, adding backward modalities for the converse of the accessibility relation leads to a \emph{tense-logical} version of $\mathsf{ML}(\#)$. While such an extension seems straightforward, several earlier notions would need non-trivial adaptation. Moreover, the typical valid tense-logical principles $ \varphi \rightarrow \Box^{\rightarrow}\Diamond^{\leftarrow}\varphi,\varphi \rightarrow \Box^{\leftarrow}\Diamond^{\rightarrow}\varphi $ relating the two directions of $R$ suggest a more systematic analysis of the connections between counting in the forward and backward directions. 

Also of interest is adding global counting operators, which, as noted above, can define the usual global existential modality over the whole domain. Extending the well-known fact that standard modal $\mathsf{S5}$ provides  an alternative notation for monadic first-order logic without identity, we could also consider global counting as a device by itself, yielding another presumably simple modal counterpart  to the system $\mathsf{MFO}(\#)$.

Of course, beyond $\mathsf{MFO}(\#)$, the other systems considered in  earlier sections, too,  suggest modal extensions. For instance, an  analogue to the notion of \emph{multiple counting} in $\mathsf{MFO}(\sharp)$ might involve ``multi-dimensional'' modal counting logics \citep{MarxVenema}.  

Perhaps more importantly in representing natural patterns of reasoning,  one can add \emph{second-order quantifiers} over sets, on the analogy of the earlier system  $\mathsf{MSO}(\#)$. This would  result in   a second-order version of  $\mathsf{ML}(\#)$ comparable to basic modal logic with quantifiers over propositions \citep{Fine1970}. In fact, if we add quantification over proposition letters to $\mathsf{ML}(\#)$ with \emph{global} counting, then this gives us an alternative, modal notation for sentences of $\mathsf{MSO}(\#)$, thanks to Theorem \ref{normal} and Lemma \ref{lemma:prenex}.\footnote{Strictly speaking we also need to add the statement $\#\varphi \approx 1$, expressing that there is exactly one $\varphi$-point.} By the same argument as in \S\ref{section:soinf}, such a system will be decidable (Corollary \ref{cor:mso}). One concrete use for it will be found in \S\ref{section:vocab}, when discussing quantifier constructions in natural language. 

\section{Generalizing the Counting Semantics} \label{section:beyond} The systems we studied in \S\ref{section:basicmonadic}-\S\ref{section:modal} all dealt with \emph{syntactic} fine-structure and tractable fragments of the natural, but excessively rich and complex system $\mathsf{FO}(\#)$. However, there is another, complementary means of recovering from intractability; that is to change the \emph{semantics} (cf. \citealt{vanBenthem2005}). In the present setting, at least two possibilities suggest themselves, each with their own motivation: (1) broaden the interpretation of $\#$-terms, so that they may denote elements of a more general class of algebraic structures, and (2) generalize the logical semantics in ways known to reduce complexity, e.g., by allowing variation in the space of allowable variable assignments \citep{Nemeti}. We discuss each in turn, with an emphasis on the generalized-value approach. In this more exploratory section, we will not provide the same level of detail as in our earlier presentation. \vspace{-0.5ex}

\subsection{Beyond Counting} We can break our  interpretation of count terms $\#_x\varphi$ into two steps. In a model $\mathcal{M}$ with domain $D$ and variable assignment $s$, we first consider the set $\semantics{\varphi}^{\mathcal{M},s}_x =\{d \in D: \mathcal{M}, s^x_d \models \varphi\}$. In the second step we map subsets $S \subseteq D$ of the domain to \emph{numbers}. Thus we have a map $f:\wp{(D)} \rightarrow  \{\kappa: |D| \geq \kappa\}$, with $S \mapsto |S|$. Ultimately we set\vspace{0.5ex} \begin{equation} \label{eq:Fvariation} \semantics{\#_x\varphi}^{\mathcal{M},s} = f(\semantics{\varphi}^{\mathcal{M},s}_x). \end{equation}\vspace{0.5ex}
We now want to consider generalizing Eq. (\ref{eq:Fvariation}) by allowing a broader class of functions $f:\wp{(D)} \rightarrow \mathbb{P}$, where $\mathbb{P} = (P;\geq)$ may be some other poset than a set of cardinal numbers. 

\subsection{Probability and Proportionality} \label{section:probprop}

As a first example, let $\mathbb{P} = ([0,1];\geq)$ be the real unit interval. (The rational interval $([0,1]\cap\mathbb{Q};\geq)$ would also suffice for much of what we will say.) Over finite models a  natural map to consider is the function $f:\wp{(D)} \rightarrow [0,1]$ sending $S$ to the ratio $|S|/|D|$. It is straightforward to verify that the valid reasoning principles in  the systems $\mathsf{MFO}^\phi(\#)$ and $\mathsf{MSO}^\phi(\#)$ will remain unchanged. The basic propositional and modal systems, $\mathsf{PL}^\phi(\#)$ and $\mathsf{ML}^{\phi}(\#)$, can also be given such a proportionality interpretation whereby $\semantics{\#\varphi}^{\mathfrak{M},s}$ is the proportion of (successor) points where $\varphi$ holds. 

On this interpretation, terms $\#_x\varphi$ (or $\#\varphi$) can be construed as specifying the \emph{probability} that $\varphi$ is satisfied, a connection between elementary logics of counting and  probability made explicit in \cite{Hoek1996b}. The probability measures obtained in this way are all \emph{regular} in that they assign every non-empty set non-zero probability (cf. \citealt{Ding2021}). 

What happens when we move to polyadic systems $\mathsf{MFO}(\sharp)$ and $\mathsf{MSO}(\sharp)$? Some of our work has natural analogues here. For instance, recall again our analysis of `Many $Q$s are $P$' (Eq. (\ref{eq:many})). Interpreted as a probabilistic statement about a measure $\mu$, this says, $\mu(P\mid Q) > \mu(P)$, i.e., that $Q$ ``confirms'' $P$ \citep{reichenbach:56}. However, the two interpretations---counting and proportion---no longer agree on general logical principles when we allow poladicity:
\begin{example} \label{example:probprop} Consider the formula
$$\exists y,z.\big(y \neq z \wedge \sharp_x(x= y \vee x = z) \approx \sharp_{x,w}(P(x) \wedge P(w))\big).$$ This is not satisfiable under our counting interpretation since it would require $|P| = \sqrt{2}$. By contrast, on the proportionality interpretation: it merely requires that $2 \times |D| = |P|^2$. \end{example}
This echoes a broader theme that reasoning about conditional probabilities already amounts to general reasoning about real fields \citep{Mosse,Ibeling}.

\subsection{Mass, Weight, and Abstract Values} \label{section:mass} Probability and proportionality are still clear quantitative numerical measures of sizes and ratios. However, generalizing beyond these, our logics also support more qualitative interpretations as calculi of ``weight'' or ``mass'' \citep{Link}. In the above two-step set up, we can think of terms $\#_x\varphi$ as denoting a collective entity in some intuitive sense, say, like in  the semantics of \emph{plural} expressions and mass terms in natural language. The values assigned to these might lie in some qualitative mereological algebra. The minimum needed for interpreting our $\#$-languages is then some pre-order on this mereological algebra, while further structure may come in the form of, not addition or multiplication, but \emph{fusion}, and perhaps other mereological primitives \citep{Lesniewski}. 

We will not pursue this more general perspective here, which deserves  a separate development of its own. Instead, we only note two changes from our earlier logical analysis.

\subsection{Non-classical Logics} Recall that in \S\ref{section:countinglogic}, we pointed at different logics, classical or non-classical, to come out of the counting component of our systems. In the current generalized setting, ways of inducing logical operations multiply. For instance, as long as $f(D) \geq f(\varnothing)$ and $f(\varnothing) \ngeq f(D)$, we will recover at least the classical Booleans in the same way as before via (\ref{implication}) and (\ref{negation}). However, if we merely drop the requirement that $f(D) \geq f(\varnothing)$, we can already  invalidate ``paradoxes of material implication'' such as $\varphi \rightarrow (\psi \rightarrow \varphi)$, while still validating some principles typical of relevant logics (see \citealt{Restall}), such as $\neg\neg \varphi \leftrightarrow \varphi$. We leave further exploration of this way of inducing logical systems to future work.

\subsection{Embedding into Multisorted $\mathsf{FO}$} Our second logical observation is more technical. With a generalized semantics, some of our earlier conclusions about system behavior and complexity need to be reconsidered. Provided we only stipulate \emph{first-order} conditions on the partial order $\mathbb{P}$ and on the map $f:\wp{(D)} \rightarrow \mathbb{P}$---both natural requirements in abstract mereological semantics---we can show that the set of valid principles for $\mathsf{FO}(\#)$ becomes \emph{computably enumerable}. The same style of  analysis also applies to the extended system $\mathsf{FO}(\sharp)$ which allows counting tuples. 


\begin{theorem} 
Over generalized models, satisfiability in $\mathsf{FO}(\sharp)$ can be translated faithfully into satisfiability in an associated three-sorted first-order language.
\end{theorem}
\begin{proof}[Proof Sketch]
The idea here is as follows. The above generalized semantics works on three-sorted structures with a domain $D$ of objects, a domain $P$ of collectives or ``predicates'' for the denotations of $\#$-terms, and a value domain $V$, with a binary relation $E$ between objects and predicates, a function $f:P\rightarrow V$ sending predicates to values, and a binary relation $\geq$  on the value domain. 

We can state what is needed for this to work in first-order terms, on the analogy of Henkin models for second-order logic: (a) an \emph{Extensionality} principle stating that predicates standing in the $E$ relation to the same objects are the same, (b) a set of \emph{Comprehension} principles making sure that the domain of predicates is closed under definitions in our language with finitely many parameters.  

With this in place, we can  translate our $\#$-languages into this three-sorted first-order language. In particular, the $\mathsf{Tr}$-translation of an expression $\#_x \varphi \succsim \#_x \psi$ will read\vspace{0.5ex}   $$\exists p, q.\, \forall x \big(E(x,p)\, \leftrightarrow\, \mathsf{Tr}(\varphi)(x)\big) \,\land\, \forall x. \big(E(x,q) \leftrightarrow \mathsf{Tr}(\psi)(x)\big) \, \land\, f(p) \geq f(q).$$ 

\noindent The translation extends to multiple count operators, where available predicates can now have arbitrary finite arities, with natural closure properties  still described in first-order style.

Now it is straightforward to show that a formula $\varphi$ of our $\#$-language is satisfiable in abstract value semantics iff its translation $\mathsf{Tr}(\varphi)$ is satisfiable in a three-sorted model for the above effectively axiomatized first-order theory consisting of Extensionality and Comprehension.\vspace{-1ex}\end{proof}


Incidentally, the same strategy can also  bring down the complexity of  second-order versions of our $\#$-languages, as we let second-order quantifiers range, Henkin-style, over the special set of available predicates in the above three-sorted models.

The shift to a first-order perspective  has noteworthy repercussions for the meta-properties of $\#$-logics. Consider the failure of \emph{Compactness} observed earlier (Proposition \ref{prop:compactness}): this property will hold now, because of the first-order reduction outlined above. 

\begin{example} It is of interest to see how this works with the standard counterexample to Compactness. The finitely satisfiable set $\{\neg \exists^\infty x. \top\} \cup \{ \exists^{\geq n}x.\top: n < \omega \}$  
is not satisfiable in  standard cardinality semantics, but it is  in generalized semantics. Concretely, we take a language with only the identity predicate, dropping unary predicates for convenience. Now consider finite models $\mathcal{M}_n$ of all cardinalities $n$ and take an ultrapower over these with respect to a free ultrafilter. In the resulting model, all first-order properties of our finite models still hold, and we can say concretely how the generalized function $f$ works. The value domain will be an uncountable linear order consisting of one copy of the standard natural numbers followed by copies of the integers with, at the end, one copy of the negative integers. On finite subsets, $f$ gives sizes in the standard natural numbers, and on cofinite sets, it will give values in the final copy of the negative integers, counting down from the infinite largest element.  \end{example}

\subsection{Generalized Dependence  Semantics} There is also a more logic-oriented approach to lowering the complexity of our initial system $\mathsf{FO}(\#)$. In so-called \emph{generalized assignment semantics} for first-order logic, models come with a range of admissible or available assignments, where gaps in the full space of all functions from variables to objects encode \emph{dependencies} or correlations between variables \citep{Nemeti,BaltagvanBenthem}. The main truth condition is now that $\mathcal{M}, s \models \exists x \varphi$\, iff \,there exists some admissible assignment $t$ in the model which is equal to $s$ except for the value of $x$, and  such that    $\mathcal{M}, t \models  \varphi$.

It is known that the set of validities on generalized assignment models is decidable \citep{Nemeti}, while additional first-order principles impose further existential closure conditions on the admissible assignments, such as Church-Rosser style confluence properties that support the encoding of undecidable tiling problems, thus elucidating the assumptions underlying the undecidability of $\mathsf{FO}$ on standard models. Moreover, generalized assignment models  support various  decidable language extensions, such as polyadic tuple quantifiers \citep{vanBenthem2005}, and explicit atoms expressing functional dependence of variables $y$ on sets of variables $X$ \citep{BaltagvanBenthem}.

\vspace{1ex}

To extend the semantics of $\mathsf{FO}(\#)$ to generalized assignment models, we need a stipulation as to how we are going to \emph{count} in this setting. Various options may be considered given the richer environment of available vs. arbitrary assignments, but here is one that seems natural. At an assignment $s$ in a model $\mathcal{M}$, 

\vspace{1.5ex}

 \quad $\#^{X}_Y \varphi$ denotes the cardinality of the set of all tuples of values taken by the set $Y$ in
 
 \vspace{0.3ex}
 \quad those assignments in $\mathcal{M}$ that (a) agree with $s$ on their $X$-values, and (b) make $\varphi$ true. 

\vspace{1.5ex}

\noindent This counts ranges of values for some variables conditional on the current values of other fixed variables, leaving open  the status of yet other variables occurring in the formula $\varphi$. 

In terms of this count notion, for instance, the existential quantifier (and the dependence modalities of \citealt{BaltagvanBenthem}) can easily be defined in the style that was introduced in \S\ref{section:countinglogic}. Moreover, functional dependence of $y$ on $X$ can be written as    $\#^{X}_y \top(y) \approx \#^{\{x\}}_x \top(x)$. But we can also express other notions of correlation, up to forms of \emph{independence}. For instance, $\#^{X}_y \top(y) \approx \#^{\emptyset}_y \top(y)$ says that the local values of $X$ leave a value range for $y$ whose cardinality is equal to the total value range for $y$ in the model.

We submit that this combination of generalized assignment semantics and count terms is interesting, but exploring its natural open problems  is beyond the scope of this paper. 


\section{Generalized Quantifiers and Natural Language} \label{section:gq} The preceding section concludes our analysis of elementary combinations of logic and counting in terms of a standard hierarchy of designed formal systems. Let us now return to the setting of our Introduction, and take a look at how these issues manifest in natural language, the vehicle for our broader daily practices of logical reasoning and counting.

An obvious source for such a comparison is Generalized Quantifier Theory \citep{BarwiseCooper,vanBenthem1985,PetersWesterstahl}, an area where logic and counting have always co-existed, even though the field often places the emphasis on logic in the formal syntax, while the counting aspect resides in the semantics. We will develop this interface with more empirical practice in some detail, and show how quantification in natural language and the theory developed around it connect in interesting ways with the earlier systems. As it happens, new questions will arise both ways.

\vspace{0.5ex}

{\it A point of notation:} Throughout this section we will be using letters $A,B,C,\dots$ for subsets of a domain $D$. Following the semantic literature, we will use the same letters interchangeably as predicate symbols in a formal language, provided no confusion can arise.


\vspace{-1ex}

\subsection{Quantifier Expressions in Logical Semantics} \label{section:quantifiers}

\vspace{1ex}

There is a wide range of quantifier words and quantificational constructions in natural language. The quantifier vocabulary includes first-order expressions such as `all', `some', `no', but also numerals like `one', `two', or combinations like `all except two'. But there are also higher-order expressions such as `most', and expressions whose meaning is highly context-dependent such as `many', `enough', and so on. Moreover, the quantifier vocabulary includes comparative expressions such as `More $A$ than $B$ are $C$' or `As many $A$ as $B$ are $C$', or even `Twice as many $A$ as $B$ are $C$'.

For a  basic pattern, one usually takes the binary format $Q AB$ of $\langle 1,1 \rangle$ quantifiers, with $Q$ a quantifier expression and $A, B$  unary predicates denoting sets of objects. It is generally assumed that quantifiers in natural languages satisfy some universal constraints, such as,\vspace{0.5ex}
\begin{itemize}
    \item[] \emph{Conservativity} \quad $Q AB$ holds \,iff\, $Q A(A\cap B)$ holds.
\end{itemize}\vspace{0.5ex}
In what follows, in line with the literature, we will assume Conservativity throughout, though most of our results could be given more general formulations. 
Other widely assumed constraints  hold in many cases. An important one is \emph{Extension} saying that the  relation $Q AB$ does not depend on the total universe of objects inside which the sets $A, B$ are located. Finally, true quantifier expressions are purely numerical in the sense of satisfying the following constraint, which also played a key role in \S\ref{section:basicmonadic}-\ref{section:sequences}:\vspace{1ex}
\begin{itemize}
    \item[]  \emph{Permutation Invariance} \quad $Q AB$ holds iff $Q \pi[A]\pi[B]$ for any permutation $\pi:D\rightarrow D$.
\end{itemize}\vspace{1ex}
The total effect of these three  conditions ties quantifiers closely to counting. To specify the meaning of a quantifier expression $Q$, it suffices to list its acceptance behavior on all pairs of numbers $(a, b)$ where $a = |A - B|, b = |A \cap B |$ for some $A, B$ such that  $Q AB$ holds (see Figure \ref{fig:quant}(a)). Accordingly, Generalized Quantifier Theory studies quantifiers equally well in numerical terms as in logical ones. A typical tool in Generalized Quantifier Theory for visualizing this double perspective is the so-called {\it tree of numbers} for representing quantifiers graphically in terms of the pairs $(a,b)$ representing cardinalities $(|A-B|, |A\cap B|)$ (Figure \ref{fig:quant}(b)). 
Here a quantifier can be seen as a subset of the tree.  Further special properties of quantifiers then show up as geometrical patterns in the tree. For instance, if $Q$  is upward monotonic in its right-hand argument, $Q$ will be closed when moving upward along an upward diagonal line from any point where it holds. Upward monotonicity in the left-hand argument shows as acceptance of the 
sub-quadrant generated by points $(a, b)$ where $Q$ holds. These geometric descriptions lead to simple characterizations of all possible monotonic quantifiers, or all first-order definable quantifiers \citep{vanBenthem1985}.

\begin{remark} From our  perspective, the tree of numbers approach is interesting for its twists. It arises by passing from logical syntax to numerical content of quantifier expressions, but with that in place, it again geometrizes that numerical content, something that one might see as a further move to qualitative  geometric logic.\end{remark}

\begin{figure} \centering 
\subfigure [Effect of Conservativity, Extension, and Invariance] {
\begin{tikzpicture}[framed]
        \node at (0,1.25) {};
       \draw [fill=blue!10] (0,0) circle (1cm);
       \draw (1,0) circle (1cm);
       \node at (-1,-1.05) {$A$};
       \node at (2,-1.05) {$B$};
       \node at (-.45,0) {$a$};
       \node at (.5,0) {$b$};
       \node at (1.5,0) {$c$};
       \node at (-1.2,1.2) {$e$};
\end{tikzpicture}
   } \hspace{.25in} 
\subfigure [The tree of numbers $(a,b)$ with `all' highlighted in red] {
   \begin{tikzpicture}
   \node (a) at (5,5) {\footnotesize{} $0,0$};
   \node at (4.5,4.25) {\footnotesize{} $1,0$};
   \node at (5.5,4.25) {\footnotesize{} $0,1$};
      \node at (4,3.5) {\footnotesize{} $2,0$};
   \node at (5,3.5) {\footnotesize{} $1,1$};
      \node at (6,3.5) {\footnotesize{} $0,2$};
    \node at (3.5,2.75) {\footnotesize{} $3,0$};
    \node at (4.5,2.75) {\footnotesize{} $2,1$};
    \node (r) at (5.5,2.75) {\footnotesize{} $1,2$};
    \node (p) at (6.5,2.75) {\footnotesize{} $0,3$};
    \node at (5,2.25) {$\vdots$};
  \pgfmathanglebetweenpoints{\pgfpointanchor{a}{center}}{\pgfpointanchor{p}{center}}
        \pgfmathsetmacro{\myangle}{\pgfmathresult}
        \node[fill=red!5, rounded corners=2pt, rotate fit=\myangle, fit=(a) (p)] {};
\node at (5,5) {\footnotesize{} $0,0$};
   \node at (5.5,4.25) {\footnotesize{} $0,1$};
         \node at (6,3.5) {\footnotesize{} $0,2$};
             \node at (6.5,2.75) {\footnotesize{} $0,3$};
\end{tikzpicture} 
}\hspace{.25in}
\subfigure [Arithmetical expressions for some binary quantifiers] {\begin{tabular}{|l | c|} \vspace{-1.25in} 
 &  \\ \hline
\footnotesize{}`all' & \footnotesize{}$a=0$ \\
\footnotesize{}`some' & \footnotesize{}$b \neq 0$ \\
\footnotesize{}`at least two' & \footnotesize{}$b\neq 0 \wedge b\neq s0$ \\
\footnotesize{}`most' & \footnotesize{}$b>a$ \\
\footnotesize{}`even number of' & \footnotesize{}$\exists x. b=x+x$ \\
\footnotesize{}`many' & \footnotesize{}$b \times e > a \times c$ \\ \hline
\end{tabular} 
}
\caption{(a) Assuming Conservativity, Extension, and Invariance, we need only be concerned with $a=A-B$ and $b=A\cap B$. (b) The tree of numbers consists of all pairs $(a,b)$. Highlighted in light red are pairs in the quantifier `all'. (c) Examples of quantifiers and their arithmetic expressions. Note that, in addition to requiring multiplication, the quantifier `many' violates Extension.} \label{fig:quant}
\end{figure}

\vspace{-1ex}

\subsection{Linguistic  Vocabulary and $\#$-Logics} \label{section:vocab} For a start, the logical system $\mathsf{MFO}(\#)$ shows various analogies with the preceding style of analysis. First in terms of general constraints, it enjoys the following well-known logical property.

\begin{fact} \label{fact:relativization} $\mathcal{L}_\#^1$ is closed under  \emph{relativization} to definable subdomains. 
\end{fact}
\begin{proof} The crucial step of defining a relativization map $\varphi \mapsto (\varphi)^A$ is given by the transformation of $\#_x(\varphi) \succsim \#_y(\psi)$ into $\#_x(A(x) \wedge (\varphi)^A) \succsim \#_y(A(y) \wedge (\psi)^A)$. \end{proof}
This property allows us to explore valid reasoning principles in $\mathsf{MFO}(\#)$ that assume Conservativity and Extension, which did not yet come to the fore in our earlier analysis in \S\ref{section:basicmonadic}. An illustration is the following principle of \emph{Quantity}:
$$\big((\varphi)^A(B) \wedge \#(A-B) \approx \#(C-D) \wedge \#(A \cap B) \approx \#(C \cap D)\big) \rightarrow (\varphi)^C(D)$$
Here $(\varphi)^{A}(B)$ makes a pure cardinality statement about $B$ inside the set $A$ (a strong form of Conservativity) and the conclusion is that this statement would hold for any set $C$ with the same numerical behavior with respect to $D$. This expresses a form of Permutation Invariance in terms of validities in $\mathsf{MFO}(\#)$.

\vspace{1ex}

As for inference patterns for specific quantifiers, our Introduction highlighted  numerical syllogisms with first-order quantifiers, numerals, exceptive quantifiers `all except  at most $k$', and comparative quantifiers such as `most' or `more... than...'. Numerical syllogisms are often analyzed in practice using Venn Diagrams with number information written into the zones, as in Figure \ref{fig:quant}(a).  This representation  was the intuitive basis for the normal forms for the system $\mathsf{MFO(\#)}$ (recall Figure \ref{fig:variables}) which encodes all of the above reasoning.


These informal comparisons can be made precise by means of  definability results. Here is an illustration using the model-theoretic definability analysis presented  in \S\ref{section:mfodef}.

\binaryq*
\begin{proof}

Consider the elementary theory of $\langle\mathbb{N}\,;>\rangle$, the natural numbers with the binary relation ``greater than''. For simplicity, assume we also have (the definable) function symbols $0$ and $s$. This logic then has quantifier elimination: every formula $\alpha(x,y)$ is equivalent to a Boolean combination of equalities and inequalities between terms of the form $s^{n_1}(0)$, $s^{n_2}(x)$, or $s^{n_3}(y)$. The key to the theorem is  that these are also exactly the normal forms for the binary quantifiers definable in $\mathsf{MFO}(\#)$.

To make this more precise, we first characterize exactly the type $\langle 1, 1 \rangle$ quantifiers that can be defined by formulas $\varphi^A(A,B)$. By relativization (Fact \ref{fact:relativization}) we only have two relevant state descriptions, $A\cap B$ and $A-B$. By Theorem \ref{normal},  $\varphi^A(A,B)$ is equivalent to a disjunction of $m$-inequalities with constant numbers $m$ involving $\#(A-B)$ and $\#(A\cap B)$. We can assume that such inequalities involve no sums, as they would denote the size of the whole domain, and we can then eliminate these. Of the remaining cases, the statement $k \succsim \#(A-B)+\#(A\cap B)$ can be rewritten as a large disjunction over all ways of dividing $k$ or less between $\#(A-B)$ and $\#(A\cap B)$. Likewise, $\#(A-B ) + \#(A\cap B) \succsim k$ is rewritten as a large disjunction over all pairs adding up to $k$. The remaining cases are handled similarly. The result is a Boolean compound of inequalities that is expressible in the first-order theory of $\langle \mathbb{N}; >\rangle$.

Next, we map the language of formulas $\varphi^A(A,B)$ to the arithmetical language by employing two distinguished variables $x$ and $y$, corresponding to the $\mathcal{L}_\#^1$-terms $\#(A\cap B)$ and $\#(A-B)$, respectively. By the foregoing it is easy to see that every expression $\varphi^A(A,B)$ corresponds to an arithmetic formula $\alpha(x,y)$ in the two free variables $x,y$. In the other direction, each arithmetical formula $\alpha(x,y)$ of the form produced by quantifier elimination is easily seen to  be expressed by an appropriate $\mathcal{L}_\#^1$ formula.
\end{proof}

Thus, the first-order quantifiers $Q\, AB$ of natural language are definable. But of course $\mathsf{MFO}(\#)$ can define non-first-order quantifiers too, such as `Most $A$ are $B$'. The normal form format of \S\ref{section:basicmonadic} cuts across standard first-order/second-order boundaries.

\begin{remark} \label{remark:ton} The binary quantifiers definable in $\mathsf{MFO(\#)}$ can be classified algebraically in terms of our normal forms, but a more geometrical perspective is provided by  the tree of numbers. This is a discrete version of the usual representations of solution sets for systems of linear inequalities (cf. \citealt[p. 85]{Schrijver} for a systematic treatment). In this special case with just two numbers $a, b$, the  inequalities occurring in our normal forms reduce to the following simpler types:

\vspace{1ex}

(i) $a = k$, (ii) $a > k$, (iii) $a + b = k$, (iv) $a + b > k$, (v)  $a = b + k$, (vi) $a > b + k$,

plus all versions of these with $a$, $b$ interchanged.

\vspace{1ex}

\noindent To see why all these forms can occur, note that terms $T_i$ in normal forms are disjunctions of state descriptions, with the empty disjunction allowed. On the other hand, we can suppress some  possible forms such as $k > a$, since these are finite disjunctions of equalities. 

Now, in the tree of numbers, these types correspond with simple  geometrical patterns. (i) describes a right- or left-sloping \emph{diagonal line}, (ii) an infinite downward \emph{triangle}, (iii) a horizontal line, (iv) a \emph{trapezoid} below a horizontal line, (v) a vertical line, and (vi) a \emph{slice}: a left- or rightward half triangle.  Analyzing up to finite disjunctions, we  look at the intersections produced by these. Here, as in earlier definability arguments, we focus on  what happens beyond some finite tree level, as a finite number of  points above that level can  be dealt with by adding  explicit definitions in terms of intersecting lines. Next,   shapes can be simplified further in term of finite unions: a horizontal line is a finite set of points, a trapezoid is a finite union of triangles, a full triangle is a union of two slices. We are left with the following basic shapes (above a certain tree level): \emph{diagonal lines}, \emph{vertical lines} and \emph{slices}. Intersections of these can produce finite unions of (a) single points, (b) diagonal lines from some point onward,  (c) horizontal lines from some point onward, (d) slices from some point onward.  For instance, intersecting two slices with different orientations produces an infinite ``band'' extending vertically downward,  but this is a finite union of vertical lines. 

In all, we are left with finite disjunctions of the following $\mathsf{MFO}(\#)$-definable types of quantifiers: (a) `Exactly $k\, A$ are $B$ and exactly $m \,A$ are  not-$B$', (b) `There are at least $k$ A and exactly $m$ of these are $B$', `There are at least $k$ A and exactly $m$ of these are not-$B$', (c) `There are at least $k$ A and among these, there are equally many  $B$ and not-$B$', and (d) `There are at least $k$ A with at least $m$ $B$ among these, and fewer $B$ than non-$B$' (for a left-looking slice), and vice versa for the other case.

This geometric analysis extends that  for first-order quantifiers in \cite{vanBenthem1985}, where the only basic shapes needed are diagonal lines and triangles.

\end{remark}

\begin{remark} It should be noted that the preceding analysis is about quantifiers on finite domains only. While this restriction is often assumed in the semantics of natural language, a generalization to \emph{infinite models} would be of interest. For an extension of the tree of numbers representation to infinite cardinalities, cf. \cite{vanDeemter}.\end{remark}


Clearly, the infinitely many quantifiers definable in the above manner are not all realized in natural language, though they can drive an interesting search for examples and non-examples. For instance, the simple pattern $b > a +2$ seems to defy a simple unforced linguistic description, say, in terms of `most' and `except'. Of course, with enough words, one can always paraphrase what this says in artificial ways (cf. (\ref{at-least}) below), but we are interested here in the quantifiers that have actually been \emph{lexicalized} in natural languages (roughly in the sense of being ``morphosyntactically simple''; see, e.g.,  \citealt{KeenanPaperno}).

On the other hand, there are also realistic  quantifiers in natural language that our base system cannot express.

\begin{corollary} $\mathsf{MFO(\#)}$ cannot express the quantifier \emph{`An even number of $A$s are $B$'} or the proportionality quantifiers \emph{`At least $1/n$ of the $A$s are $B$'} for $n>2$. \label{corollary:proportion}
\end{corollary} 

These additional quantifiers require the resources of our second-order system $\mathsf{MSO(\#)}$, which overshoots considerably compared to natural language, as it can define all Presburger definable logical quantifiers. The following is a direction consequence of Theorem \ref{thm:semi-lin}:
\begin{corollary} The binary quantifiers definable in $\mathsf{MSO(\#)}$ are exactly those expressible in the first-order theory of $\langle \mathbb{N}\,; + \rangle$.
\end{corollary} 

Here, the over-generation of the logic continues. Say, $2a = b$  says that  the number of $A$s that are $B$ equals twice the number of $A$s that are not, or rephrased: two thirds of the $A$s are $B$s. This is intelligible, but not part of natural basic quantifier vocabulary.


Finally, our move to the richer system $\mathsf{MFO}(\sharp)$ and Diophantine arithmetic raises even further issues. Here is what we found earlier.

\begin{corollary} The quantifiers definable in $\mathsf{MFO}(\sharp)$ are exactly those expressible by Boolean combinations of Diophantne inequalities, i.e., in a small bounded fragment  of the complete first-order theory of $\langle \mathbb{N}\,; +,\times \rangle$. 
\end{corollary}

It has been suggested in \cite{vanBenthem1985} that the arithmetical content of linguistic quantifiers is essentially restricted to addition. In that case, multiplication would be irrelevant to understanding the linguistic quantifier repertoire. However, our current analysis throws doubts on this picture. The natural  meaning of `many' involved multiplication, and natural language does have resources for comparing proportions. Moreover, it does form \emph{pairs of objects} in basic syntax, witness naturally occurring relational phrases such as `who married whom'. The resulting counting of pairs or longer tuples suggests  connections with our multiple count logic $\mathsf{MFO}(\sharp)$. However, the formulas that we used to define multiplication have somewhat artificial variable binding patterns that need not occur in natural language. The multiplicative content of natural quantifier expressions remains to be determined.

Finally, while the preceding discussion was about basic \emph{quantifier vocabulary},  natural language also has more complex \emph{quantifier constructions}. Well-known constructions of logical interest are ``cumulative'' and ``branching quantifiers'' (see \citealt{PetersWesterstahl}). A particular construction worth highlighting here is the role of particles qualifying meanings of quantifier combinations. Consider a sentence like:
\begin{equation} \label{family} \mbox{`Every family has a different problem'.} \end{equation}
This is not just a simple $\forall\exists$ combination,  demanding the existence of some choice function from families to problems. The particle `different' requires that choice function to be one-to-one, more like our cardinality comparison statements. However, there is a crucial difference. In this case, the one-to-one function must lie {\it inside a given relation}, in our concrete sentence: the relation \emph{having}. This seems a case where natural language poses a challenge.
\begin{remark} Linguists have been well aware of these and related phenomena, and have advanced relatively complex machinery to handle the full range of attested patterns. See, e.g., \cite{Brasoveanu,Bumford} for two notable dynamic accounts.

Notably, such constructions come up in the context of probabilistic reasoning as well \citep{HHI2016}, in a way that reverberates elsewhere in natural language, witness modal language about  probability and likelihood \citep{HIForthcoming}. \end{remark} We suspect that this notion of ``guarded injection'' is not even definable in the strong counting logic  $\mathsf{FO}(\#)$. However, for finite cardinalities there is a connection with the weaker logics considered in this paper: in this case, a modal system.

The  \emph{Hall marriage theorem} in graph theory \citep{Hall} says that there is an injection from a  set $A$ into a set $B$ contained in a relation $R \subseteq A \times B$ iff for each subset $C$ of $A$, $|R[C]| \geq |C|$. But this can be used to give a simple definition of `different' sentences like (\ref{family}) in our modal logic with global counting and one second-order quantification over sets: 
\begin{example}
Let $F$ be the unary predicate for family, $G$ for problems, and suppose $\Diamond$ moves along the relation `$x$ is had by $y$'. Then the required definition is\begin{equation*}
    \forall X \big(X \subseteq F \rightarrow \# (G \wedge \Diamond X) \succsim \# X\big)
\end{equation*} 
where $X \subseteq F$ is shorthand for $\#\bot \succsim \#(X \wedge \neg F)$. \label{ex:problems}\end{example}

This concludes our brief comparison of quantifier expressions in natural language with the expressive resources of our $\#$-logics. Clearly, this is not so much a matter of proving theorems as of exploring empirical fit. The hierarchy in our system design may suggest patterns in the architecture of natural language, while, precisely when the fit is not evident, common constructions in natural language may pose non-trivial questions concerning logical systems. We have just provided some illustrations here, a deeper investigation  of linguistic versus logical architecture would require another paper.

\subsection{Varieties of Monotonicity Reasoning} \label{section:varmon}

Next we move from quantifier vocabulary to  inference patterns in natural language. \emph{Monotonicity inferences} arise when  occurrences of a predicate in positive syntactic position are replaced ``upward'' by occurrences of a predicate with a larger denotation, or when in negative position, ``downward'' by a predicate with a smaller extension \citep{vanBenthem1985,Sanchez91,IcardMoss2014}. Monotonicity inference works all across natural language for many kinds of quantifiers, but just as well for other numerical  expressions, witness a valid inference like `If more $A$ than $B$ are $C$, and all $A$ are $E$, then more $E$ than $B$ are $C$'.

Monotonicity with \emph{inclusion premises} is also a valid inference form in logical systems, and in particular, in the ones studied here. Let us mark  syntactic positions as follows in  formulas of $\mathsf{MFO}(\#)$.
An atomic formula $P(x)$ occurs positively in $P(x)$ itself, positive and negative occurrences keep their polarity in conjunctions and disjunctions, their polarity switches under negations, and finally, 
in atoms $\#_x\varphi \succsim \#_x \psi$, occurrences in $\psi$ switch  polarity, while those in $\varphi$ keep their polarity. It is easy to show the following:

\begin{proposition} Positive occurrences in formulas of $\mathsf{MFO(\#)}$ support valid upward monotonicity inferences, negative occurrences downward monotonicity inferences.\end{proposition}

\begin{remark} It seems likely that $\mathsf{MFO(\#)}$ also satisfies a \emph{Lyndon Theorem} to the effect that semantic monotonicity amounts to positive definability up to logical equivalence (see \citealt{vanBenthem1991,IcardMossTune2017}). Our normal forms contain all information necessary for a constructive proof of  this result. However, we leave this as an open problem. \end{remark}

The syntax of $\mathsf{MFO(\#)}$ in fact suggests two kinds of monotonicity reasoning: the usual one with inclusion  premises, but also one with \emph{cardinality premises}, in forms such as
$$\varphi(B) \, \text{and}\, \# A \succsim \#B \,\, \text{imply}  \,\,  \varphi(A).$$ As it happens, the  inductive clauses for positive and negative occurrences  work here as before, the crucial failure is the atomic clause, as premises $Bx, A \succsim B$ obviously do not imply $Ax$. Clearly, numerical monotonicity implies its set-theoretic variant, but the converse can fail. The quantifier `Some $B$ are $C$' is upward set-monotonic in its argument $B$, but obviously not numerically monotonic in $B$, since the larger set $A$ may be disjoint from $B$ and $C$.

\begin{remark} \label{remark:mon} Numerical monotonicity as stated here has some interesting features as a mixture of logic and counting. As a special case, if $\varphi(A)$ is true and we replace $A$ by a predicate $B$ of the same cardinality, then $\varphi(B)$ is true. This very strong insensitivity property intuitively separates $\varphi$ into some purely numerical assertion about $A$ plus an assertion that is not about $A$ at all. This may be provable as a preservation theorem for formulas in first-order logic, and for  $\mathsf{MFO(\#)}$, a complete characterization of numerically monotonic formulas may be provable through our normal forms. However, we end with one small observation. 

Consider binary quantifiers $Q$ definable in the logic $\mathsf{MFO(\#)}$. In some cases, the two kinds of monotonicity are close. For instance, if $Q\, AB$ is upward set-monotone in the argument $B$, then it is also upward cardinality-monotone in the following sense, restricted to the set $A$. If $Q\, AB$ and \emph{at least as many $A$ are $C$ as $B$}, then $Q\,AC$. The crucial property here is  \emph{Permutation Invariance}: given $A$, the quantifier $Q$ is fixed by the set of all sizes of subsets $B$ which it accepts, and set-monotonicity plus permutation invariance imply that these sizes are upward closed. The restriction to comparing inside $A$ is necessary here, since cardinal monotonicity w.r.t. $B$ for arbitrary larger $C$  not inside $A$ can easily fail.  The same failure occurs with upward set-monotonicity in the left-hand argument $A$, where a larger set $C$ may change the context of evaluation. Even so, permutation invariance does support a valid second-order inference pattern for left-upward set-monotonic quantifiers $Q$:  if $Q\, AB$ and 
  \emph{there are at least as many $C$ as $A$}, then $\exists C' \approx C.\, Q\, C'B$, where $C'$ can be taken to be any set equinumerous to $C$ that contains $A$, so that left-upward monotonicity applies to it. 


\end{remark}


Cardinality monotonicity resembles monotonicity in numerical terms, where a variable $x$ occurs positively in $x$, retains its polarity across addition, multiplication, and the left hand side of inequalities $\mathbf{t}_1 \geq \mathbf{t}_2$, while switching polarity on the right hand side of these inequalities. Making this work using our normal forms takes some care though, since their numerical terms $T_i$ do not refer to sizes of predicates, as in the above, but of state descriptions.  A unified perspective on monotonicity in logical and arithmetical syntax has  been proposed in \cite{IcardMoss2014}. As for concrete examples,  \cite{vanBenthemLiu} note several different versions of set-based and size-based monotonicity inference that hold for the natural language expression `Many $A$ are $B$' that involve increasing or decreasing the size of relevant zones in the Venn diagram for $A, B$.

\begin{remark}[Natural logic] Monotonicity reasoning in natural language is an engine of ``natural logic'' \citep{vanBenthem1985,Sanchez91,Moss2015}: efficient forms of surface reasoning based on simple fragments and proof systems. Our $\#$-logics are more expressive than most of the calculi studied in this literature, and it would be of interest to locate natural logic fragments inside them (see, for example,  \citealt{Hartmann2008,Hartmann2009,Moss2016,moss_topal_2020,KisbyMoss}).
\end{remark}

\subsection{Dynamic Modalities} 

Monotonicity inference can also be viewed dynamically in terms of {\it model change}. One such change is internal to a current model: one merely changes the denotation of some predicate $A$ to a larger (or smaller) set $X$ of objects, turning the current $\mathcal{M}$ into a new model  $\mathcal{M}[A:= X]$. Other operations on models arise with different intuitive takes on what upward monotonicity inference is about. It could also mean that we add \emph{new objects} to the current model that satisfy the predicate $A$, in which case the relevant relation between models is extension. And this perspective can even be generalized. On the earlier analogy with monotonicity in numerical terms, since the latter stand for zones of the model in our normal forms, the replacement for, say, $x : = x + 1$ applies to \emph{regions} defined by state-descriptions, rather than single predicates.

In recent years, model change has been studied by adding dynamic modalities to logical languages, cf. the recent study of \cite{mlsr}. A standard example is $[!\varphi]\psi$ which says that $\psi$ is true after we \emph{relativize} the current model to the submodel of all objects satisfying $\varphi$. This fits the earlier discussion of Conservativity and Extension for quantifiers. Next, upward inclusion monotonicity in our first sense  suggests a modality $[+A]\psi$ which holds when $\psi$ is true in all models arising from the current one by increasing the denotation of $A$. Downward monotonicity may then refer to decreasing the denotation of $A$, or more drastically, to removing objects from the current model. The dynamic modality $[-\varphi]\psi$ for the latter model change says that each removal from the current model of an object satisfying $\varphi$ results in a model satisfying $\psi$.

\begin{proposition} $\mathsf{MFO(\#)}$ is closed under the dynamic modalities $[!\varphi]$ and $[-A]$. $\mathsf{MSO(\#)}$ is closed under the modality $[+A]$.
\end{proposition}
\begin{proof} The case of relativization can be dealt with by providing axioms that recursively analyze the possible syntactic shapes of the formula $\psi$. The proof for the deletion modality is by inspection of normal forms in a manner similar to that used in \cite[Thm. 6.1]{mlsr}. $\mathsf{MFO(\#)}$ is not closed under the predicate extension modality, since it can define having an even number of points with some property, but it does have a straightforward definition in the second-order $\mathsf{MSO(\#)}$. 
\end{proof}

Similar closure results can be obtained in our monadic $\#$-logics for dynamic modalities describing the effects of adding an object to the current model. 

\begin{remark} \label{remark:mlsr}
Another source for model change occurred with the discussion of  counting in modal languages in \S\ref{section:modal}. Instead of adding explicit numerical information like in graded modal languages, one can also count by ``setting  aside'' objects and then perhaps replacing them, removing or adding objects to a current model. For instance, having at  least $k+1$ successors with property $p$ is  definable using the deletion modality as $[-\top]\dots (k  \mbox{ times}) \dots[-\top]\Diamond p$ and thus \emph{counting in the syntax}. There is a link here with Remark \ref{remark:bisim} about possible finer notions of modal bisimulation that analyze counting procedures. A typical way of comparing sizes between two sets  picks an object in one set plus an  object from the other set, and then puts these two objects aside, iterating the process. But keeping track of  effects of removals of matched objects is exactly what  $\mathsf{MLSR}$-style bisimulations do (see \citealt{mlsr}). 
\end{remark}

Technical topics like dynamic modalities may seem far from natural language. But the distance is not that great. Natural language contain many verbs of \emph{fact change} that fit this setting. Indeed, \cite{LiuMonotone} give samples of logical reasoning in the ancient Chinese tradition 
that involve monotonicity inferences with dynamic verbs  such as  `increase'. 

\subsection{Semantic Automata} \label{section:automata}
Our final topic comes again from Generalized Quantifier Theory, and it brings one more entanglement of logic and counting. There is a natural way of classifying quantifiers in terms of the associated \emph{verification procedures} and determining their complexity in the Automata Hierarchy \citep{vanBenthem1985}. The word `count' is of course polysemous between a verbal use (the act of counting) and a nominal use (the total counted), and here the focus is on the former, dynamic aspects of counting.

 \emph{Semantic automata} read strings of symbols $\mathtt{a}, \mathtt{b}$ standing for types of relevant objects encountered when traversing a finite domain (Figure \ref{fig:quant}(a)). That is, each element of $A-B$ corresponds to an occurrence of $\mathtt{a}$ in the string, while each element of $A \cap B$ corresponds to an occurrence of $\mathtt{b}$. The automaton reads the string and accepts precisely when the pair $(a,b)$ is in the quantifier. These automata, and the complexity jumps predicted by them for quantifier denotations, have also been studied as models for the mixture of quantifier reasoning in the brain and cognitive sciences (see \citealt{Szymanik} for an overview).

 \begin{example} This acyclic finite automaton in Figure \ref{fig:aut} recognizes the quantifier `exactly one'. It accepts any pair $(a,1)$ with $a\geq 0$, and no other pairs. That is, there should be exactly one element in $A \cap B$; more or fewer should lead to non-acceptance. \end{example}
 
  \tikzset{every loop/.style={min distance=10mm,looseness=10}}
 
 \begin{figure} \centering 
   \begin{tikzpicture}
  
  \node (s1) at (0,0) [circle,draw=black] {};
\node (s2) at (3,0) [circle,draw=black,fill] {};
\node (s3) at (-1,0) {};
\node (s4) at (6,0) [circle,draw=black] {};
  \path (s1) edge[->,thick] (s2);
  \node (h1) at (0, 1.1) {\small{$\mathtt{a}$}};
  
    \node (h2) at (3, 1.1) {\small{$\mathtt{a}$}};
    
    \node (h3) at (1.5,-.3) {\small{$\mathtt{b}$}};
        \node (h4) at (6, 1.1) {\small{$\mathtt{a}, \mathtt{b}$}};
        
        \node (h5) at (4.5,-.3) {\small{$\mathtt{b}$}};
  
  \path (s1) edge[->,thick,loop above] (s1);
  
  \path (s3) edge[->,thick] (s1);
  \path (s2) edge[->,thick] (s4);
   \path (s2) edge[->,thick,loop above] (s2);
     \path (s4) edge[->,thick,loop above] (s4);
\end{tikzpicture} 
    \caption{Acyclic finite automaton recognizing   `exactly one'. The machine begins in the left-most state, and the middle is the only accepting state.}\label{fig:aut}
\end{figure}
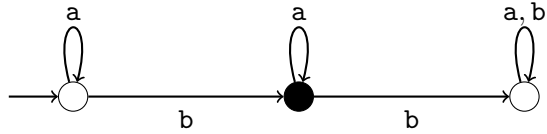

Moreover, familiar operations on quantifiers, such as \emph{iteration}, correspond systematically to natural operations on standard classes of automata \citep{Steinert}. We list some known results on the subject:

\begin{proposition}\label{prop:aut}\begin{enumerate}[label=(\alph*)] \item The first-order definable binary quantifiers are exactly those that are recognized by acyclic finite automata \citep{vanBenthem1985}.
  \item \label{prop:div} Finite automata with non-trivial cycles can recognize \emph{`An even number of A are B'} and related periodic quantifiers. In fact, finite automata recognize precisely the quantifiers definable in first-order logic with \emph{divisibility} \citep{Mostwoski1998}.
  \item The binary quantifier \emph{`most'} and related proportionality quantifiers are not computable by finite automata, but they are computable by pushdown store automata. In fact, pushdown automata recognize precisely the quantifiers definable in additive \emph{Presburger Arithmetic}, i.e., the \emph{semi-linear} sets \citep{vanBenthem1985}.
\end{enumerate}
\end{proposition}

We thus see numerous deep connections with our earlier systems. Most obviously, we saw the semi-linear sets in our analysis of $\mathsf{MSO}^\phi(\#)$ (Theorem \ref{thm:semi-lin}). Proposition \ref{prop:aut} adds a further computational dimension to this characterization: the quantifiers definable in $\mathsf{MSO}^\phi(\#)$ are precisely those that can be verified by pushdown automata. The counting procedures required for verifying claims of $\mathsf{MSO}^\phi(\#)$ are those that can be carried out with a pushdown store. 

Identifying such a computational analogue for our other systems could also be illuminating. For instance, our initial system, $\mathsf{MFO}^\phi(\#)$, misses some quantifiers definable even  by finite automata---`an even number of' being an illustrative example (Corollary \ref{corollary:proportion})---while capturing some  quantifiers that demand unbounded memory such as `most' or `exactly half'. It also makes sense to interrogate the other direction. What systems combining logic and counting would capture the quantifiers recognizable by intermediate classes such as \emph{counter automata}, or even weaker classes like those recognizing \emph{subregular languages} (cf. \citealt{Graf})? We leave such questions for further analysis, but end here with a final observation tying together several of our earlier themes, including \emph{permutation invariance}. 

As we have seen as multiple points (\S\ref{section:coreprinciples}, \S \ref{section:quantifiers}, \S\ref{section:varmon}), the theme of permutation invariance is paramount in the analysis of logic and counting. Given this assumption for quantifiers, the corresponding formal languages will also be closed under permutations. For instance, if $\mathtt{ababa}$ appears in the quantifier language, so will $\mathtt{aaabb}$. This is a relatively exceptional property for sets of strings:  the permutation closures of languages accepted by finite automata and by pushdown automata actually coincide---as it happens, they characterize the semi-linear sets \citep{Parikh}. 
It is therefore of interest to understand the permutation closed (or ``commutative'') languages in their own right. Such languages have been studied since the beginning of formal language theory (e.g., \citealt{Eilenberg}). Here our question is the following: restricting to permutation closed languages, which semi-linear sets are also accepted by \emph{finite} automata? This would give us a way of calibrating the counting capacity of finite-state machines, relative to semi-linear sets. 


With an alphabet of size two, recall that \emph{linear} sets (Definition \ref{def:linear}) are the solutions (for $\mathsf{v}_1,\mathsf{v}_2$) to equations given by constants $b_1,b_2,a_{1,1},\dots,a_{1,m},a_{2,1},\dots,a_{2,m}$: 
\begin{equation}
 \begin{pmatrix}
\;\mathsf{v}_1\;  \\
\mathsf{v}_2
\end{pmatrix} \hspace{.1in} = \hspace{.1in}
 \begin{pmatrix}
b_1   +  a_{1,1}\mathsf{u}_1+ \dots + a_{1,m}\mathsf{u}_m   \\
b_2   +  a_{2,1}\mathsf{u}_1+ \dots + a_{2,m}\mathsf{u}_m
\end{pmatrix} \label{eq:permsim}
\end{equation} for some choices of $\mathsf{u}_1,\dots,\mathsf{u}_m$; the semi-linear sets are finite unions of linear sets. 
\begin{definition} Let us call a set \emph{rectilinear} if it is of the form (\ref{eq:permsim}), but for all $i \leq m$ either $a_{1,i} = 0$ or $a_{2,i}=0$  (or both). A set is \emph{semi-rectilinear} if it is a finite union of rectilinear sets.
\end{definition}

It may be helpful to explain this notion in the earlier geometrical setting of the tree of numbers in Remark \ref{remark:ton} above. Linear forms in general can define both diagonal and horizontal lines, as well as more complex patterns such as triangles and slices. But there is a crucial difference. In order to produce a diagonal line, only one coordinate needs to be incremented, using a period $(0, i)$ or $(i, 0)$ with $i \neq 0$, but producing a horizontal line requires a simultaneous increment $(1, 1)$. This coordination is typically beyond the recognizing capacity of finite state machines. On the other hand, finite state machines are capable of performing counting tasks such as keeping track of cycles in the numbers of $\mathtt{a}$ (or of $\mathtt{b}$) read. This parity check can define quantifiers like `an even number of' which were beyond $\mathsf{MFO}(\#)$. The geometric meaning of these cycles shows in automorphisms between tree positions accepted by the quantifier whose precise nature is explained in the proof of the following result, which is our main offering.


\begin{theorem}
The binary quantifiers recognized by finite semantic automata are precisely those whose associated arithmetical definitions are semi-rectilinear.
\end{theorem}
This theorem follows from results of \cite{Kanazawa} (see also \citealt{EHRENFEUCHT1983311}), but for completeness we offer a full proof in Appendix \ref{app:automata}. Needless to say, this is just the beginning of a study of counting procedures and their relation to semantic meanings, as a natural complement to the logic and counting entanglements studied in this paper.

\section{Cognitive Questions} \label{section:cognition} We encountered in the previous section some examples of interleaving logic and counting in natural language. This entanglement is very much on display in psychology and neuroscience as well. As pointed out by \cite{carey2009origin}, children first learn explicit numerical terms as examples of quantifiers, and work such as \cite{BARNER2009195} has shown a strong correlation in development between comprehension of number terms and comprehension of (logical) quantifiers.\footnote{The psychologist Piaget famously argued that children's understanding of number was built out of logical primitives (thus, another version of ``logicism''). Subsequent research has revealed a more subtle entanglement. with numerical primitives arising much earlier. See \cite{carey2009origin,Dehaene} for discussion.} Early learning about basic logical and numerical constructs is evidently intertwined, and as we have argued this continues even through more mature ``grassroots mathematics'' and ordinary reasoning practices. 

But how, more specifically, might the logical systems we have studied here relate to cognition? The fundamental primitives we have assumed in all of our logical systems are numerical comparisons such as $\#\varphi \succ \#\psi$ or $\#\varphi \approx \#\psi$. The ability to make such comparisons is present across a wide range of species, and appears to be available in human infants from birth (see \citealt{Feigenson,Dehaene}). 
Unsurprisingly, `more' emerges as one of the first quantificational phrases children learn, alongside plurals and `a'/`some' \citep{carey2009origin}. 
There is also evidence for basic operations like addition and subtraction in preverbal infants \citep{Feigenson}, and in adults, researchers have even uncovered distinct brain areas for encoding addition and for making numerical comparisons \citep{Dehaene}.  This all raises the  question of how, computationally speaking, numerical comparisons are made.

A prominent theme throughout the empirical literature is the distinction between reasoning about \emph{individuals} and their properties, and reasoning about \emph{collections} or \emph{ensembles} and their properties. To solve a concrete task such as determining whether there are more $A$s than $B$s there are at least three conceivable families of strategies: \begin{enumerate}
    \item \label{most-second} Match each $B$ one-to-one with an $A$ and check whether there are any $A$s left over.
    \item \label{most-first} Explicitly count the numbers $\#A$ and $\#B$ and compare those numbers.
    \item \label{most-third} Perceptually approximate  $\#A$ and $\#B$ and compare those approximations.
\end{enumerate} (\ref{most-second}) and (\ref{most-first}) both require enumerating through the relevant objects in an explicit way---much like the semantic automata discussed in the previous section---while (\ref{most-third}) bypasses any explicit enumeration or counting procedure, relying instead on fast, parallel perceptual processing (such as when we visually estimate the number of balls in a bin). Such an \emph{approximate number system} (ANS) is in fact ubiquitous and phylogenetically ancient \citep{Dehaene}.

Much experimental work has gone into distinguishing hypotheses like these in specific instances \citep{carey2009origin}. A striking example investigates the psychological representation of quantifier expressions in natural language \citep{Pietroski,Lidz2011,Knowlton2,Knowlton}. Consider, for instance, verifying a sentence like `Most of the dots are blue' (see Figure \ref{fig:dots}). Any of these strategies, (\ref{most-second}), (\ref{most-first}), or (\ref{most-third}), could in principle be used, where $A$ is ``blue dots'' and $B$ is something like ``non-blue dots'' (though see \citealt{Lidz2011}). \cite{Pietroski} present convincing evidence that people in fact employ a strategy more like (\ref{most-third}), with the counts $\#A$ and $\#B$ likely determined by the  ANS. Queries involving `more' can also invoke the ANS, though the method people use appears distinct from that for `most' \citep{Knowlton2}. In  further work, \cite{Knowlton} show that different English expressions for \emph{universal} quantification in fact elicit  different representations altogether: while `all' and `every' prompt representations of ensembles and their cardinalities, `each' seems to elicit an individual-level procedural strategy more like  semantic automata. 

 \begin{figure} \centering 
   \begin{tikzpicture}
   \draw [black,fill=gray!50] (0,0) rectangle (4,3);
  \node at (2,2.5) [circle,scale=0.5,fill=blue] {};
  \node at (1,.4) [circle,scale=0.7,fill=yellow] {};
  \node at (.4,1) [circle,scale=1,fill=yellow] {};
  \node at (2.1,1.9) [circle,scale=.4,fill=yellow] {};
 \node at (.4,.3) [circle,scale=0.65,fill=blue] {};
\node at (2.3,.2) [circle,scale=0.5,fill=blue] {};
\node at (.6,2) [circle,scale=1.2,fill=blue] {};
\node at (.9,1.4) [circle,scale=.9,fill=blue] {};
\node at (2.8,1.1) [circle,scale=.7,fill=yellow] {};
\node at (3.3,.3) [circle,scale=.85,fill=yellow] {};
\node at (3.7,2.44) [circle,scale=.5,fill=yellow] {};
\node at (3.6,2.8) [circle,scale=.6,fill=blue] {};
\node at (3.55,.7) [circle,scale=.95,fill=blue] {};
\node at (2,1.2) [circle,scale=1.5,fill=yellow] {};
\node at (1.6,.5) [circle,scale=1.2,fill=blue] {};
\node at (3,2.4) [circle,scale=.65,fill=blue] {};
\node at (3.2,2.1) [circle,scale=.55,fill=blue] {};
\node at (2.8,1.7) [circle,scale=.8,fill=blue] {};
\node at (.2,2.56) [circle,scale=.65,fill=blue] {};
\node at (.65,2.75) [circle,scale=.75,fill=blue] {};
\node at (2.7,.5) [circle,scale=.9,fill=yellow] {};
\node at (1.3,2.2) [circle,scale=.77,fill=yellow] {};
\node at (1.6,1.6) [circle,scale=.65,fill=yellow] {};
\node at (3.8,1.1) [circle,scale=.55,fill=blue] {};
\node at (3.4,1.4) [circle,scale=.8,fill=yellow] {};
\node at (3.7,1.8) [circle,scale=.95,fill=yellow] {};
\node at (1.35,2.78) [circle,scale=1,fill=blue] {};
\node at (2.4,2.68) [circle,scale=.68,fill=blue] {};
\node at (2.6,2.2) [circle,scale=.46,fill=blue] {};
\end{tikzpicture} 
    \caption{A display of dots, where experimental participants might be asked to determine whether, `Most of the dots are blue' or `There are more blue dots than yellow dots' (see, e.g., \citealt{Pietroski,Knowlton2}).}\label{fig:dots}
\end{figure}
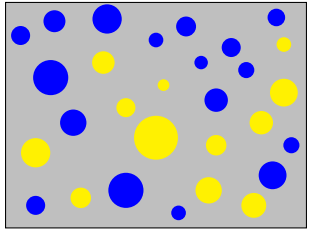


Relating these tasks to our logical systems, consider a first-order term $\#_x\varphi$.\footnote{Recall that, given (\ref{uniform-replace}), we need only consider subformulas in $\varphi$ that mention $x$.} We think of $\varphi$ as describing the constraints that determine what is to be counted. The availability of any of these strategies,  (\ref{most-second}), (\ref{most-first}), or (\ref{most-third}),  depends on the extent to which the mind can ``filter'' by $\varphi$.\footnote{As a special case, there has been interest in understanding which organisms can reason with the number \emph{zero} (i.e., $\#_x x\neq x$). Recent work suggests that this is  within range for crows \citep{Kirschhock4889}.}\footnote{Relevant is also the amount of memory required in principle to implement each of these strategies. For instance, it is possible to implement (\ref{most-second}) with less memory than (\ref{most-first}), provided $A$ and $B$ are represented as lists. Thanks to an anonymous reviewer for calling attention to this dimension.} For instance, successful application of the approximate number system (\ref{most-third}) depends on specific perceptual qualities such as spatial or temporal contiguity \citep{Dehaene}, while application of (\ref{most-second}) depends on how easy it is to match pairs one-to-one without repetition.

A logical property that is distinctive of our monadic first-order system $\mathsf{MFO}(\#)$ and its extensions is that we allow a kind of ``quantifying in'' to terms like $\#_x\varphi$ (recall, e.g., Figure \ref{fig:variables}). 
Consider a query such as, \begin{equation} \label{at-least} \mbox{`There are at least }2\mbox{ more blue dots than yellow dots',} \end{equation} i.e., $\#B \succsim \#Y+2$. In $\mathsf{MFO}(\#)$ this is encoded naturally as $$\exists y_1,y_2. y_1 \neq y_2 \wedge B(y_1) \wedge B(y_2) \wedge \#_x(B(x) \wedge x\neq y_1 \wedge x \neq y_2) \succsim \#_x Y(x),$$ 
whereby we ``remove'' two blue dots and then compare. Perhaps even more natural is the second-order version in $\mathsf{MSO}(\#)$ (with appropriate abbreviations as introduced earlier): \begin{equation} \label{subtraction} \exists Z. |Z|\approx \mathbf{2} \wedge Z \subseteq B \wedge \#_x\big(B(x) \wedge \neg Z(x)\big) \succsim \#_x Y(x). \end{equation} This essentially asks us to locate a subset of two blue dots and  \emph{subtract} those from the total number of blue dots before comparing. This type of predicate subtraction is consistent with observed patterns (e.g., \citealt{Lidz2011}), and while (\ref{subtraction}) does not yet specify a precise procedure, it seems an interesting question whether verification of sentences like (\ref{at-least}) would induce representations anything like (\ref{subtraction}). Exceptive phrases, such as `No one dared attempt the bonus question, except for a few of the best students', also seem to call for a means of ``removing'' subparts of a predicate (see, e.g., \citealt{PetersWesterstahl}, Chapter 8).

\vspace{1ex}

Moving beyond $\mathsf{MFO}(\#)$ and $\mathsf{MSO}(\#)$, what evidence is there for fundamental numerical representations involving polyadicity or multiplication? Of course, our running example of `many' (like its antonym `few') is exceedingly common, also appearing early in development, though there is still significant debate about how  these expressions should be analyzed \citep{Rett},\footnote{E.g., it has been suggested, based on examples like, ``his sins were many; his virtues were few'' \citep{Hoeksema}, that `many' should be understood grammatically not as a quantifier at all, but as an adjectival modifier.} and how closely they should be unified with their \emph{mass} counterparts like `much' and `little' (\citealt{Rothstein}; cf. our discussion in \S\ref{section:mass}). 

More direct evidence about polyadicity and multiplication comes from the surprising finding that 11-month infants can already compare proportions, for instance, preferring a ratio of $50/100$ to one of $100/500$ \citep{Denison}. Such phenomena appear consistent with a representation involving counts of pairs, perhaps like our $\mathsf{MFO}(\sharp)$, though it has also been suggested that the ANS can directly represent and compare rational numbers (see  \citealt{ClarkBeck}), which might look more like the probabilistic interpretation of our $\#$-terms described in \S\ref{section:probprop}. Teasing apart these different possibilities presents an exciting opportunity to interface between experimental inquiry and more theoretical explorations. 

As one last example of contacts between empirical cognitive science and the themes of this paper, let us return once again to the Pigeonhole Principle. In an experimental study of patterns resembling our opening example, repeated here:
\begin{itemize}
    \item[] \textit{Premise}: 20 farmers own at most 15 cows each.
    \item[] \textit{Conclusion}: At least 2 farmers own the exact same number of cows.
\end{itemize}
\cite{Sperber} found at most 30\% of participants realized that the conclusion definitely follows. The proposed explanation for this is that, to apply the Pigeonhole Principle we need to construe the numbers less than 15 as themselves forming categories, viz. ``the property of having exactly $k$ cows'' for $k\leq 15$. Thus, while each instance of the first-order encoding (\ref{eq:fophp}) of the Pigeonhole Principle may be clearly valid, realizing that the $P_i$ need to stand for these numerical predicates requires a further step of interpretation. 

Although the relational encodings of the Pigeonhole Principle---(\ref{eq:rphp}) and its modal variant (\ref{eq:rphp2})---enjoy an elegant generality lacking in the monadic formulation, the interpretive step from stimulus to representation is even more formidable here. The relation $Rxy$, meaning ``$x$ has $y$-many cows'', is not one that most people are accustomed to thinking about. The premise of (\ref{eq:rphp2}) then becomes something like, ``there are more cow-owners than numbers-of-cows-owned'', which again may not come so naturally or immediately to people. 

It is but a short way from reasoning puzzles and ``grassroots mathematics'' to even more subtle and abstract applications of such principles in more advanced topics. The Pigeonhole principle itself manifests throughout mathematics, often in surprising ways. 
For instance, it is used in a simple proof of the Erd\H{o}s-Szekeres Theorem in graph theory \citep{Seidenberg1}, and the infinitary version of the principle (recall Eq. (\ref{eq:fophpinf}) above) for the case of $k=2$ appears in proofs of the well-known Bolzano-Weierstrass Theorem.\footnote{In short: dividing an interval containing an infinite sequence into two subintervals will guarantee infinitely many points inside at least one of these subintervals.} While the principle itself is straightforward enough, just as in the experiments by \cite{Sperber}, the difficulty is often in choosing the relevant predicates so as to see that it applies in the first place.


Once we turn to infinitary patterns in logic and counting, a whole additional array of cognitive questions arise. Chief among these is the question of how our initial conceptions of numbers and counting can be extended to accommodate basic infinitary reasoning.

Some researchers have suggested that the individual developmental stages in mastering the modern concept of infinity actually mirror the \emph{historical} development of the concept (see \citealt{MathEdu}, echoing a broader theme familiar from \citealt{PiagetGarcia}). From Galileo's bewilderment that infinite sets could be matched one-to-one with their proper subsets (and thus that, in our terminology, $\mathsf{s}=\mathsf{s}+1$ could be satisfiable), to Bolzano's explicit introduction of \emph{infinity} as a potential feature of any set that we can describe (thus giving clear meaning to our notation $\#\varphi$ when the $\varphi$s are unbounded), and eventually to ``Cantor's paradise'', children undergo a surprisingly similar sequence of transitions \citep{MathEdu}. It is intriguing to consider whether any of the systems studied here might correspond to intermediate ``way-stations'' in this development, capturing only a suitably restricted range of more intuitive infinitary patterns. Because our monadic and modal systems involve at most addition and multiplication, the infinitary patterns in these systems are less complex than their finitary counterparts. Whether this type of logical complexity could be brought to match intuitive cognitive complexity is worth investigating further. 

This concludes our brief tour of just a few salient points of contact with empirical issues in the cognitive sciences. A deeper foray into such contacts would undoubtedly reveal many further connections and opportunities. 


\section{Conclusion}

This paper has presented a number of contributions to studying the interplay of \emph{logic and counting}, viewed as a basic phenomenon in human reasoning in its own right. In fact, we  encountered three perspectives on what it means to combine logic and counting. The main perspective  adopted here is one of consilience and synergistic co-existence. As a complement to the related bodies of research in the theory of generalized quantifiers \citep{BarwiseFeferman,PetersWesterstahl} and in computational logic \citep{Otto,Schweikardt}, we explored a hierarchy of  progressively richer formal systems exemplifying this perspective (summarized in Table \ref{table:overview}). A common theme running through all of these systems is the separation between logical reasoning patterns needed to derive meaningful \emph{normal forms}, and the varieties of numerical reasoning suggested by those normal forms.  The latter spanned from (fragments of) additive arithmetic to Diophantine inequalities and full elementary arithmetic, also encompassing basic counting along binary relations. In each case infinitary reasoning could be cleanly separated and, at least for the systems we considered here, revealed as a simplified version of the corresponding finitary patterns. Finally, we probed natural generalizations of these systems, obtained either by broadening the possible interpretations of numerical terms or by relaxing the logical semantics.

Parallel to this formal development, we explored entanglements between logic and counting in natural language and thought. Quantifier vocabulary alone provides a kind of microcosm illustrating many of our broader motifs, with rich logical, linguistic, psychological, and computational dimensions, all highlighting novel mixtures of logic and counting. We also touched on ontogenetically and phylogenetically more basic examples of ``number sense'', in addition to more sophisticated reasoning patterns on the cusp of mature mathematics, the famous Pigeonhole Principle being a paradigmatic instance.

Throughout these explorations the individual contributions of logic and of counting, while often still distinctly identifiable, nonetheless resist disentangelement. Take a system like $\mathsf{MFO}(\#)$, the starting point of our analysis. The count term $\#_x\varphi$ is assumed to denote a cardinal number, but \emph{under a logical description} specified  by $\varphi$. Meanwhile, a characteristically quantitative principle---permutation invariance---begets qualitative principles in the logical language such as (\ref{invariance-main}) and (\ref{uniform-replace}), which in turn allow for derivation of explicitly numerical normal forms that support familiar numerical algorithms. As \cite{Hilbert} once put it, ``a partly simultaneous development of the laws of logic and
 arithmetic is requisite'' (p. 347).

Similar patterns permeate our discussions around extensions of $\mathsf{MFO}(\#)$, and of the various empirical phenomena in language and cognition. Monotonicity inference, to take a typical example, operates at a level that abstracts away from logical or arithmetical details, for instance treating number lines and predicate hierarchies on a par.

The other two perspectives on logic and counting---less emphasized in the present treatment but historically at least as prominent---reflect an aspiration toward methodological purity. We briefly considered how much of logic could be extracted from ``pure'' counting. As we saw, classical logic emerges from remarkably austere numerical primitives, and non-classical systems can also be elicited. For instance, in place of the ``true'' universal quantifier $\neg \exists x. \neg \varphi$ we could entertain variants like $\#_x\varphi \approx \top \wedge \#_x\varphi \succ \#_x\neg\varphi$, which states that \emph{almost all} objects satisfy $\varphi$, except for a few that ``do not count''. In the other direction we considered some of the counting principles already implicit in (first-order) logical systems. The recurrent theme of \emph{counting in the syntax} is typical in this connection (developed further in Appendix \ref{appendix:syntax}). 

Even with the above exploration in place, the three angles on logic and counting pursued in this paper do not exhaust the rich and ubiquitous entanglement of logic and counting. To mention just one more instance, there are also natural and illuminating  \emph{computational} perspectives. We briefly explored one of these, in the form of a procedural semantics for logical expressions afforded by \emph{semantic automata} (\S\ref{section:automata}), that allow us to calibrate the counting content of meanings for quantifier expressions. But also more globally, we can measure the numerical content of an entire logical system in terms of the \emph{computational complexity} of its satisfiability problem. Indeed, there is a precise sense in which any NP-hard logical system---for instance, ordinary propositional logic---can be said to solve arbitrary integer programs, via a simple (viz. polynomial) SAT reduction. In a similar vein, any $\Sigma^1_1$-hard system---even one that is not overtly quantative such as first-order dynamic logic  \citep{Harel}---implicitly answers arbitrary arithmetical queries. This angle affords a relatively coarse-grained means of calibrating logical and numerical reasoning, and we have even seen in the present article how it would collapse expressively and intuitively distinct systems (e.g., $\mathsf{MFO}(\#)$ and $\mathsf{MSO}(\#)$). But entanglements via computational complexity can go even deeper, as seen in the methods of \emph{proof complexity} where logical encodings of numerical principles like Pigeonhole take center stage \citep{cook_reckhow_1979,Krajicek}. Research programs like this only reinforce the view of consilience and co-existence as a natural habitat.

In closing, it is important to acknowledge that reductive aspirations and methodological purity often originate from motivations that are not themselves logical or mathematical. The program of logicism, for instance, has been concerned with philosophical puzzles about the epistemology and metaphysics of ``number'' (e.g., \citealt{HaleWright}). Measurement theorists, meanwhile, have maintained that only ``qualitative (that is, nonnumerical) empirical laws'' have objective  significance, with numerical representations merely ``a matter of convention'', chosen for ``computational convenience'' \citep[pp. 12-13]{Krantz1971}. Whatever one's stance on these and other philosophical and methodological issues, we hope to have shown that the important borders and thresholds in understanding reasoning are not those between qualitative and quantitative reasoning, but between simple and complex combinations of logic and counting. Whatever we might lose in foundational purity by pursuing this path, we may gain a better understanding of human reasoning abilities in return.

\subsection*{Acknowledgements} For helpful feedback we would like to thank Xiaoxuan Fu, David Gonzalez, Erich Gr\"{a}del, Makoto Kanazawa, Phokion Kolaitis, Thomas Mayer, Paul Pietroski, Alexander Pruss, Stanislav Speranski, Rineke Verbrugge, Zhiguang Zhao, and audiences at the Nordic Online Logic Seminar, the UC Berkeley Logic Colloquium, the Stanford University Logic Seminar, and the Tsinghua University Logic Seminar. We are also grateful to the editors and referees at \emph{Bulletin of Symbolic Logic} for productive comments and suggestions. 

\bibliographystyle{apalike}
\bibliography{ms}

\newcommand{\SortNoop}[1]{}
\begin{thebibliography}{}

\bibitem[Ackermann, 1954]{Ackermann}
Ackermann, W. (1954).
\newblock {\em Solvable Cases of the Decision Problem}.
\newblock Studies in Logic and the Foundations of Mathematics. North Holland
  Publishing Company.

\bibitem[Antonelli, 2010]{antonelli2010}
Antonelli, G.~A. (2010).
\newblock Numerical abstraction via the {F}rege quantifier.
\newblock {\em Notre Dame Journal of Formal Logic}, 51(2):161--179.

\bibitem[Baader and De~Bortoli, 2019]{Baader}
Baader, F. and De~Bortoli, F. (2019).
\newblock On the expressive power of description logics with cardinality
  constraints on finite and infinite sets.
\newblock In Herzig, A. and Popescu, A., editors, {\em Frontiers of Combining
  Systems}, pages 203--219.

\bibitem[Bacchus, 1990]{Bacchus}
Bacchus, F. (1990).
\newblock {\em Representing and Reasoning with Probabilistic Knowledge}.
\newblock MIT Press.

\bibitem[Baltag and {\SortNoop{Benthem}}van~Benthem, 2021]{BaltagvanBenthem}
Baltag, A. and {\SortNoop{Benthem}}van~Benthem, J. (2021).
\newblock A simple logic of functional dependence.
\newblock {\em Journal of Philosophical Logic}.

\bibitem[Barcel\'{o} et~al., 2020]{Barcelo}
Barcel\'{o}, P., Kostylev, E.~V., Monet, M., P\'{e}rez, J., Reutter, J., and
  Silva, J.~P. (2020).
\newblock The logical expressiveness of graph neural networks.
\newblock In {\em Proceedings of the International Conference on Learning
  Representations (ICLR)}.

\bibitem[Barner et~al., 2009]{BARNER2009195}
Barner, D., Chow, K., and Yang, S.-J. (2009).
\newblock Finding one's meaning: A test of the relation between quantifiers and
  integers in language development.
\newblock {\em Cognitive Psychology}, 58(2):195--219.

\bibitem[Barwise and Cooper, 1981]{BarwiseCooper}
Barwise, J. and Cooper, R. (1981).
\newblock Generalized quantifiers and natural language.
\newblock {\em Linguistics and Philosophy}, 4(2):159--219.

\bibitem[Barwise and Feferman, 1985]{BarwiseFeferman}
Barwise, J. and Feferman, S. (1985).
\newblock {\em Model-Theoretic Logics}.
\newblock Association for Symbolic Logic.

\bibitem[Bednarczyk et~al., 2020]{Bednarczyk}
Bednarczyk, B., Demri, S., Fervari, R., and Mansutti, A. (2020).
\newblock Modal logics with composition on finite forests: Expressivity and
  complexity.
\newblock In {\em Proceedings of the 35th Annual ACM/IEEE Symposium on Logic in
  Computer Science}, page 167–180.

\bibitem[{\SortNoop{Benthem}}van~Benthem, 1986]{vanBenthem1985}
{\SortNoop{Benthem}}van~Benthem, J. (1986).
\newblock {\em Essays in Logical Semantics}.
\newblock Reidel, Dordrecht.

\bibitem[{\SortNoop{Benthem}}van~Benthem, 1991]{vanBenthem1991}
{\SortNoop{Benthem}}van~Benthem, J. (1991).
\newblock {\em Language in Action: Categories, Lambdas, and Dynamic Logic},
  volume 130 of {\em Studies in Logic}.
\newblock Elsevier, Amsterdam.

\bibitem[{\SortNoop{Benthem}}van~Benthem, 1998]{vanBenthem1998}
{\SortNoop{Benthem}}van~Benthem, J. (1998).
\newblock Program constructions that are safe for bisimulation.
\newblock {\em Studia Logica}, 60:311--330.

\bibitem[{\SortNoop{Benthem}}van~Benthem, 2005]{vanBenthem2005}
{\SortNoop{Benthem}}van~Benthem, J. (2005).
\newblock Guards, bounds, and generalized semantics.
\newblock {\em Journal of Logic, Language, and Information}, 14(3):263--279.

\bibitem[{\SortNoop{Benthem}}van~Benthem and Liu, 2020]{vanBenthemLiu}
{\SortNoop{Benthem}}van~Benthem, J. and Liu, F. (2020).
\newblock New logical perspectives on monotonicity.
\newblock In Deng, D., Liu, F., Liu, M., and Westerst{\aa}hl, D., editors, {\em
  Monotonicity in Logic and Language}. Springer.

\bibitem[{\SortNoop{Benthem}}van~Benthem et~al., 2020]{mlsr}
{\SortNoop{Benthem}}van~Benthem, J., Mierzewski, K., and Zaffora~Blando, F.
  (2020).
\newblock The modal logic of stepwise removal.
\newblock {\em The Review of Symbolic Logic}, page 1–28.

\bibitem[Blackburn et~al., 2001]{Blackburn2001}
Blackburn, P., de~Rijke, M., and Venema, Y. (2001).
\newblock {\em Modal Logic}.
\newblock Cambridge University Press, New York.

\bibitem[Borosh and Treybig, 1976]{borosh}
Borosh, I. and Treybig, L.~B. (1976).
\newblock Bounds on positive integral solutions of linear diophantine
  equations.
\newblock {\em Proceedings of the American Mathematical Society},
  55(2):299--304.

\bibitem[Brasoveanu, 2011]{Brasoveanu}
Brasoveanu, A. (2011).
\newblock Sentence-internal \emph{different} as quantifier-internal anaphora.
\newblock {\em Linguistics and Philosophy}, 34:93--168.

\bibitem[Bumford, 2015]{Bumford}
Bumford, D. (2015).
\newblock Incremental quantification and the dynamics of pair-list phenomena.
\newblock {\em Semantics and Pragmatics}, 8(9):1--70.

\bibitem[Burgess, 2010]{Burgess2010}
Burgess, J.~P. (2010).
\newblock Axiomatizing the logic of comparative probability.
\newblock {\em Notre Dame Journal of Formal Logic}, 51(1):119--126.

\bibitem[Cai et~al., 1992]{Immerman1992}
Cai, J.-Y., F\"{u}rer, M., and Immerman, N. (1992).
\newblock An optimal lower bound on the number of variables for graph
  identification.
\newblock {\em Combinatorica}, 12:389--410.

\bibitem[Carey, 2009]{carey2009origin}
Carey, S. (2009).
\newblock {\em The Origin of Concepts}.
\newblock Oxford University Press.

\bibitem[Carreiro et~al., 2021]{Carreiro}
Carreiro, F., Facchini, A., Venema, Y., and Zanasi, F. (2021).
\newblock Model theory of monadic predicate logic with the infinity quantifier.
\newblock {\em Archive for Mathematical Logic}.

\bibitem[Clarke and Beck, 2021]{ClarkBeck}
Clarke, S. and Beck, J. (2021).
\newblock The number sense represents (rational) numbers.
\newblock {\em Behavioral and Brain Sciences}, pages 1--57.

\bibitem[Cook and Reckhow, 1979]{cook_reckhow_1979}
Cook, S.~A. and Reckhow, R.~A. (1979).
\newblock The relative efficiency of propositional proof systems.
\newblock {\em Journal of Symbolic Logic}, 44(1):36–50.

\bibitem[Corcoran et~al., 1974]{Corcoran}
Corcoran, J., Frank, W., and Maloney, M. (1974).
\newblock String theory.
\newblock {\em The Journal of Symbolic Logic}, 39(4):625--637.

\bibitem[{\SortNoop{Deemter}}van~Deemter, 1984]{vanDeemter}
{\SortNoop{Deemter}}van~Deemter, K. (1984).
\newblock Generalized quantifiers: finite versus infinite.
\newblock In {\SortNoop{Benthem}}van~Benthem, J. and ter Meulen, A., editors,
  {\em Generalized Quantifiers in Natural Language}, pages 145--160. Foris.

\bibitem[Dehaene, 2011]{Dehaene}
Dehaene, S. (2011).
\newblock {\em The Number Sense}.
\newblock Oxford University Press.

\bibitem[Demri and Lugiez, 2010]{DEMRI2010233}
Demri, S. and Lugiez, D. (2010).
\newblock Complexity of modal logics with {P}resburger constraints.
\newblock {\em Journal of Applied Logic}, 8(3):233--252.

\bibitem[Denison and Xu, 2014]{Denison}
Denison, S. and Xu, F. (2014).
\newblock The origins of probabilistic inference in human infants.
\newblock {\em Cognition}, 130(3):335–347.

\bibitem[Ding et~al., 2020]{Ding2020}
Ding, Y., Harrison-Trainor, M., and Holliday, W.~H. (2020).
\newblock The logic of comparative cardinality.
\newblock {\em The Journal of Symbolic Logic}, 83(3):972--1005.

\bibitem[Ding et~al., 2021]{Ding2021}
Ding, Y., Holliday, W.~H., and Icard, T.~F. (2021).
\newblock Regularity for relative likelihood.
\newblock manuscript.

\bibitem[Ehrenfeucht et~al., 1983]{EHRENFEUCHT1983311}
Ehrenfeucht, A., Haussler, D., and Rozenberg, G. (1983).
\newblock On regularity of context-free languages.
\newblock {\em Theoretical Computer Science}, 27(3):311--332.

\bibitem[Eilenberg and Sch\"{u}tzenberger, 1969]{Eilenberg}
Eilenberg, S. and Sch\"{u}tzenberger, M.-P. (1969).
\newblock Rational sets in commutative monoids.
\newblock {\em Journal of Algebra}, 13(2):173--191.

\bibitem[Endrullis and Moss, 2019]{EndrullisMoss}
Endrullis, J. and Moss, L.~S. (2019).
\newblock Syllogistic logic with ``most''.
\newblock {\em Mathematical Structures in Computer Science}, 29(6):763--782.

\bibitem[Fagin et~al., 1990]{Fagin}
Fagin, R., Halpern, J.~Y., and Megiddo, N. (1990).
\newblock A logic for reasoning about probabilities.
\newblock {\em Information and Computation}, 87:78--128.

\bibitem[Feferman and Vaught, 1959]{FefermanVaught}
Feferman, S. and Vaught, R. (1959).
\newblock The first-order properties of products of algebraic systems.
\newblock {\em Fundamenta Mathematicae}, 47:57--103.

\bibitem[Feigenson et~al., 2003]{Feigenson}
Feigenson, L., Dehaene, S., and Spelke, E. (2003).
\newblock Core systems of number.
\newblock {\em Trends in Cognitive Sciences}, 8(7):307--314.

\bibitem[Fine, 1970]{Fine1970}
Fine, K. (1970).
\newblock Propositional quantifiers in modal logic.
\newblock {\em Theoria}, 36:336--346.

\bibitem[Fine, 1972]{Fine}
Fine, K. (1972).
\newblock {In so many possible worlds.}
\newblock {\em Notre Dame Journal of Formal Logic}, 13(4):516--520.

\bibitem[Fu and Zhao, 2023]{FuZh23b}
Fu, X. and Zhao, Z. (2023).
\newblock Modal logic with counting: Definability, semilinear sets and
  correspondence theory.
\newblock Unpublished manuscript, China University of Political Science and
  Law, Beijing and School of Mathematics and Statistics, Taishan University.

\bibitem[G\"{a}rdenfors, 1975]{Gardenfors1975}
G\"{a}rdenfors, P. (1975).
\newblock Qualitative probability as an intensional logic.
\newblock {\em Journal of Philosophical Logic}, 4(2):171--185.

\bibitem[Ginsburg and Spanier, 1966]{Ginsburg}
Ginsburg, S. and Spanier, E.~H. (1966).
\newblock {Semigroups, Presburger formulas, and languages.}
\newblock {\em Pacific Journal of Mathematics}, 16(2):285 -- 296.

\bibitem[Gr\"{a}del et~al., 1997]{GradelOtto}
Gr\"{a}del, E., Otto, M., and Rosen, E. (1997).
\newblock Two-variable logic with counting is decidable.
\newblock LICS '97. IEEE Computer Society.

\bibitem[Gr\"{a}del et~al., 1999]{Gradel1999}
Gr\"{a}del, E., Otto, M., and Rosen, E. (1999).
\newblock Undecidability results on two-variable logics.
\newblock {\em Archive for Mathematical Logic}, 38:313--353.

\bibitem[Graf, 2019]{Graf}
Graf, T. (2019).
\newblock A subregular bound on the complexity of lexical quantifiers.
\newblock In Schl\"{o}der, J.~J., McHugh, D., and Roelofsen, F., editors, {\em
  Proceedings of the 22nd Amsterdam Colloquium}, pages 455--464.

\bibitem[Grumbach and Tollu, 1995]{Grumbach}
Grumbach, S. and Tollu, C. (1995).
\newblock On the expressive power of counting.
\newblock {\em Theoretical Computer Science}, 149:67--99.

\bibitem[Grzegorczyk, 2005]{Grz}
Grzegorczyk, A. (2005).
\newblock Undecidability without arithmetization.
\newblock {\em Studia Logica}, (79):163–230.

\bibitem[Hale and Wright, 2001]{HaleWright}
Hale, B. and Wright, C. (2001).
\newblock {\em The Reason's Proper Study: Essays towards a Neo-{F}regean
  Philosophy of Mathematics}.
\newblock Oxford University Press.

\bibitem[Hall, 1935]{Hall}
Hall, P. (1935).
\newblock On representatives of subsets.
\newblock {\em Journal of the London Mathematical Society}, 10(1):26--30.

\bibitem[Halpern, 1990]{Halpern}
Halpern, J.~Y. (1990).
\newblock An analysis of first-order logics of probability.
\newblock {\em Artificial Intelligence}, 46:311--350.

\bibitem[Harel, 1985]{Harel}
Harel, D. (1985).
\newblock Recurring dominoes: Making the highly undecidable highly
  understandable.
\newblock {\em Annals of Discrete Mathematics}, 24:51--72.

\bibitem[Harrison-Trainor et~al., 2018]{HHI2016}
Harrison-Trainor, M., Holliday, W.~H., and Icard, T.~F. (2018).
\newblock Inferring probability comparisons.
\newblock {\em Mathematical Social Sciences}, 91:61--70.

\bibitem[Hartogs, 1915]{Hartogs}
Hartogs, F. (1915).
\newblock {\"{U}ber das Problem der Wohlordnung}.
\newblock {\em Mathematische Annalen}, 76:438--443.

\bibitem[Herre et~al., 1991]{hartig}
Herre, H., Krynicki, M., Pinus, A., and Väänänen, J. (1991).
\newblock The {H}\"{a}rtig quantifier: A survey.
\newblock {\em The Journal of Symbolic Logic}, 56(4):1153--1183.

\bibitem[Hilbert, 1905]{Hilbert}
Hilbert, D. (1905).
\newblock On the foundations of logic and arithmetic.
\newblock {\em The Monist}, 15(3):338--352.

\bibitem[Hoeksema, 1983]{Hoeksema}
Hoeksema, J. (1983).
\newblock Plurality and conjunction.
\newblock In ter Meulen, A., editor, {\em Studies in Model-Theoretic
  Semantics}, pages 63--83. Foris.

\bibitem[{\SortNoop{Hoek}}van~der Hoek, 1996]{Hoek1996b}
{\SortNoop{Hoek}}van~der Hoek, W. (1996).
\newblock Qualitative modalities.
\newblock {\em International Journal of Uncertainty, Fuzziness, and
  Knowledge-Based Systems}, 4(1):45--59.

\bibitem[{\SortNoop{Hoek}}van~der Hoek and de~Rijke, 1993]{Hoek1993}
{\SortNoop{Hoek}}van~der Hoek, W. and de~Rijke, M. (1993).
\newblock Generalized quantifier and modal logic.
\newblock {\em Journal of Logic, Language, and Information}, 2:19--58.

\bibitem[Hoffmann, 2019]{Hoffman}
Hoffmann, S. (2019).
\newblock Commutative regular languages -- properties and state complexity.
\newblock In {\'{C}}iri{\'{c}}, M., Droste, M., and Pin, J.-{\'E}., editors,
  {\em Algebraic Informatics}, pages 151--163. Springer.

\bibitem[Holliday and Icard, 2018]{HIForthcoming}
Holliday, W.~H. and Icard, T.~F. (2018).
\newblock Axiomatization in the meaning sciences.
\newblock In Ball, D. and Rabern, B., editors, {\em The Science of Meaning}.
  Oxford University Press.

\bibitem[Ibeling et~al., 2024]{Ibeling}
Ibeling, D., Icard, T., Mierzewski, K., and Moss\'{e}, M. (2024).
\newblock Probing the quantitative-qualitative divide in probabilistic
  reasoning.
\newblock {\em Annals of Pure and Applied Logic}, 175(9):103339.

\bibitem[Icard and Moss, 2014]{IcardMoss2014}
Icard, T.~F. and Moss, L.~S. (2014).
\newblock Recent progress on monotonicity.
\newblock {\em Linguistic Issues in Language Technology}, 9(7):167--194.

\bibitem[Icard et~al., 2017]{IcardMossTune2017}
Icard, T.~F., Moss, L.~S., and Tune, W. (2017).
\newblock A monotonicity calculus and its completeness.
\newblock In Kanazawa, M., de~Groote, P., and Sadrzadeh, M., editors, {\em
  Proceedings of the 15th Meeting on the Mathematics of Language}, pages
  75--87.

\bibitem[Kanazawa, 2013]{Kanazawa}
Kanazawa, M. (2013).
\newblock Monadic quantifiers recognized by deterministic pushdown automata.
\newblock In Aloni, M., Franke, M., and Roelofsen, F., editors, {\em
  Proceedings of the 19th Amsterdam Colloquium}, pages 139--146.

\bibitem[Karp, 1972]{Karp1972}
Karp, R.~M. (1972).
\newblock Reducibility among combinatorial problems.
\newblock In Miller, R.~E., Thatcher, J.~W., and Bohlinger, J.~D., editors,
  {\em Complexity of Computer Computations}, pages 85--103. Springer.

\bibitem[Keenan and Paperno, 2012]{KeenanPaperno}
Keenan, E. and Paperno, D. (2012).
\newblock Overview.
\newblock In {\em Handbook of Quantifiers in Natural Language}, volume~90 of
  {\em Studies in Linguistics and Philosophy}, pages 941--950. Springer.

\bibitem[Kiero\'{n}ski et~al., 2018]{Kieronski}
Kiero\'{n}ski, E., Pratt-Hartmann, I., and Tendera, L. (2018).
\newblock Two-variable logics with counting and semantic constraints.
\newblock {\em ACM SIGLOG News}, 5(3):22–43.

\bibitem[Kirschhock et~al., 2021]{Kirschhock4889}
Kirschhock, M.~E., Ditz, H.~M., and Nieder, A. (2021).
\newblock Behavioral and neuronal representation of numerosity zero in the
  crow.
\newblock {\em Journal of Neuroscience}, 41(22):4889--4896.

\bibitem[Kisby et~al., 2020]{KisbyMoss}
Kisby, C., Blanco, S.~A., Kruckman, A., and Moss, L.~S. (2020).
\newblock Logics for sizes with union or intersection.
\newblock In {\em Proceedings of AAAI}.

\bibitem[Knowlton et~al., 2021a]{Knowlton2}
Knowlton, T., Hunter, T., Odic, D., Wellwood, A., Halberda, J., Pietroski, P.,
  and Lidz, J. (2021a).
\newblock Linguistic meanings as cognitive instructions.
\newblock {\em Annals of the New York Academy of Sciences}.

\bibitem[Knowlton et~al., 2021b]{Knowlton}
Knowlton, T., Pietroski, P., Halberda, J., and Lidz, J. (2021b).
\newblock The mental representation of universal quantifiers.
\newblock {\em Linguistics and Philosophy}.
\newblock forthcoming.

\bibitem[Kraft et~al., 1959]{kps59}
Kraft, C.~H., Pratt, J.~W., and Seidenberg, A. (1959).
\newblock {Intuitive probability on finite sets}.
\newblock {\em The Annals of Mathematical Statistics}, 30(2):408--419.

\bibitem[Kraj\'{i}\v{c}ek, 2019]{Krajicek}
Kraj\'{i}\v{c}ek, J. (2019).
\newblock {\em Proof Complexity}.
\newblock Cambridge University Press.

\bibitem[Krantz et~al., 1971]{Krantz1971}
Krantz, D.~H., Luce, R.~D., Suppes, P., and Tversky, A. (1971).
\newblock {\em {Foundations of Measurement}}, volume~1.
\newblock Academic Press, New York.

\bibitem[Kuske and Schweikardt, 2017]{Kuske}
Kuske, D. and Schweikardt, N. (2017).
\newblock First-order logic with counting: At least, weak hanf normal forms
  always exist and can be computed!
\newblock In {\em Proceedings of the 32nd Annual ACM/IEEE Symposium on Logic in
  Computer Science}.

\bibitem[Lai et~al., 2016]{LaiEndrullisMoss}
Lai, T., Endrullis, J., and Moss, L.~S. (2016).
\newblock Majority digraphs.
\newblock {\em Proceedings of the American Mathematical Society},
  144(9):3701--3715.

\bibitem[Le\'{s}niewski, 1927]{Lesniewski}
Le\'{s}niewski, S. (1927).
\newblock O podstawach matematyki.
\newblock {\em Przegl\k{a}d Filozoficzny}, 30:164--206.

\bibitem[Lewis, 1980]{Lewis1980}
Lewis, H.~R. (1980).
\newblock Complexity results for classes of quantificational formulas.
\newblock {\em Journal of Computer and System Sciences}, 23(3):317--353.

\bibitem[Lidz et~al., 2011]{Lidz2011}
Lidz, J., Pietroski, P., Halberda, J., and Hunter, T. (2011).
\newblock Interface transparency and the psychosemantics of \emph{most}.
\newblock {\em Natural Language Semantics}, 19:227–256.

\bibitem[Lindstr\"{o}m, 1966]{Lindstrom}
Lindstr\"{o}m, P. (1966).
\newblock First order predicate logic with generalized quantifiers.
\newblock {\em Theoria}, 32(3):186--195.

\bibitem[Link, 1998]{Link}
Link, G. (1998).
\newblock {\em Algebraic Semantics in Language and Philosophy}.
\newblock Cambridge University Press.

\bibitem[Lipshitz, 1978]{Lipschitz}
Lipshitz, L. (1978).
\newblock The diophantine problem for addition and divisibility.
\newblock {\em Transactions of the American Mathematical Society},
  235:271--283.

\bibitem[Marx and Venema, 1997]{MarxVenema}
Marx, M. and Venema, Y. (1997).
\newblock {\em Multi-Dimensional Modal Logic}.
\newblock Springer.

\bibitem[Mayer, 2024]{Mayer}
Mayer, T.~L. (2024).
\newblock An investigation of the negationless fragment of the
  {R}escher-{H}{\"a}rtig quantifier.
\newblock In Meier, A. and Ortiz, M., editors, {\em Foundations of Information
  and Knowledge Systems}, pages 287--297.

\bibitem[Mercier et~al., 2017]{Sperber}
Mercier, H., Politzer, G., and Sperber, D. (2017).
\newblock What causes failure to apply the pigeonhole principle in simple
  reasoning problems?
\newblock {\em Thinking \& Reasoning}, 23(2):184--189.

\bibitem[Moreno and Waldegg, 1991]{MathEdu}
Moreno, L.~E. and Waldegg, G. (1991).
\newblock The conceptual evolution of actual mathematical infinity.
\newblock {\em Educational Studies in Mathematics}, 22(3):211--231.

\bibitem[Mortimer, 1975]{Mortimer}
Mortimer, M. (1975).
\newblock On languages with two variables.
\newblock {\em Mathematical Logic Quarterly}, 21(1):135--140.

\bibitem[Moss, 2015]{Moss2015}
Moss, L.~S. (2015).
\newblock Natural logic.
\newblock In {\em Handbook of Contemporary Semantic Theory, Second Edition},
  pages 646--681. Wiley-Blackwell.

\bibitem[Moss, 2016]{Moss2016}
Moss, L.~S. (2016).
\newblock Syllogistic logic with cardinality comparisons.
\newblock In Bimb{\'o}, K., editor, {\em J. Michael Dunn on Information Based
  Logics}, pages 391--415. Springer.

\bibitem[Moss and Topal, 2020]{moss_topal_2020}
Moss, L.~S. and Topal, S. (2020).
\newblock Syllogistic logic with cardinality comparisons, on infinite sets.
\newblock {\em The Review of Symbolic Logic}, 13(1):1–22.

\bibitem[Moss\'{e} et~al., 2024]{Mosse}
Moss\'{e}, M., Ibeling, D., and Icard, T. (2024).
\newblock Is causal reasoning harder than probabilistic reasoning?
\newblock {\em The Review of Symbolic Logic}, 17(1):106--131.

\bibitem[Mostowski and Tarski, 1949]{MostowskiTarski}
Mostowski, A. and Tarski, A. (1949).
\newblock Arithmetical classes and types of well-ordered systems.
\newblock {\em Bulletin of the American Mathematical Society}, (55):65.

\bibitem[Mostowski, 1998]{Mostwoski1998}
Mostowski, M. (1998).
\newblock Computational semantics for monadic quantifiers.
\newblock {\em Journal of Applied Non-Classical Logics}, 8:107--121.

\bibitem[N\'{e}meti, 1996]{Nemeti}
N\'{e}meti, I. (1996).
\newblock Fine-structure analysis of first-order logic.
\newblock In Marx, M., Masuch, M., and P\'{o}los, L., editors, {\em Arrow Logic
  and Multidimensional Logic}, pages 221--247. CSLI Publications.

\bibitem[Oppen, 1978]{OPPEN1978323}
Oppen, D.~C. (1978).
\newblock A $2^{2^{2^{pn}}}$ upper bound on the complexity of {P}resburger
  {A}rithmetic.
\newblock {\em Journal of Computer and System Sciences}, 16(3):323--332.

\bibitem[Otto, 1997]{Otto}
Otto, M. (1997).
\newblock {\em Bounded Variable Logics and Counting}.
\newblock Springer.

\bibitem[Parikh, 1966]{Parikh}
Parikh, R. (1966).
\newblock On context-free languages.
\newblock {\em Journal of the ACM}, 13(4):570--581.

\bibitem[Peters and Westerst\r{a}hl, 2006]{PetersWesterstahl}
Peters, S. and Westerst\r{a}hl, D. (2006).
\newblock {\em Quantifiers in Language and Logic}.
\newblock Oxford University Press.

\bibitem[Piaget and Garcia, 1983]{PiagetGarcia}
Piaget, J. and Garcia, R. (1983).
\newblock {\em Psychogen\`{e}se et Histoire des Sciences}.
\newblock Paris: Flammarion.

\bibitem[Pietroski et~al., 2009]{Pietroski}
Pietroski, P., Lidz, J., Hunter, T., and Halberda, J. (2009).
\newblock The meaning of ‘most’: Semantics, numerosity and psychology.
\newblock {\em Mind \& Language}, 24(5):554--585.

\bibitem[Pratt-Hartmann, 2005]{Pratt-Hartmann}
Pratt-Hartmann, I. (2005).
\newblock Complexity of the two-variable fragment with counting quantifiers.
\newblock {\em Journal of Logic, Language and Information}, 14(3):369--395.

\bibitem[Pratt-Hartmann, 2008]{Hartmann2008}
Pratt-Hartmann, I. (2008).
\newblock On the computational complexity of the numerically definite
  syllogistic and related logics.
\newblock {\em Bulletin of Symbolic Logic}, 14(1):1--28.

\bibitem[Pratt-Hartmann, 2009]{Hartmann2009}
Pratt-Hartmann, I. (2009).
\newblock No syllogisms for the numerical syllogistic.
\newblock In {\em Languages: From Formal to Natural}, volume 5533 of {\em
  LNCS}, pages 129--203. Springer.

\bibitem[Quine, 1946]{Quine}
Quine, W.~V. (1946).
\newblock Concatenation as a basis for arithmetic.
\newblock {\em The Journal of Symbolic Logic}, 11(4):105--114.

\bibitem[Reichenbach, 1956]{reichenbach:56}
Reichenbach, H. (1956).
\newblock {\em {The Direction of Time}}.
\newblock University of California Press, Berkeley.

\bibitem[Rescher, 1962]{rescher}
Rescher, N. (1962).
\newblock Plurality quantification.
\newblock {\em Journal of Symbolic Logic}, 27:373--374.

\bibitem[Restall, 2000]{Restall}
Restall, G. (2000).
\newblock {\em An Introduction to Substructural Logics}.
\newblock Routledge.

\bibitem[Rett, 2018]{Rett}
Rett, J. (2018).
\newblock The semantics of \emph{many}, \emph{much}, \emph{few}, and
  \emph{little}.
\newblock {\em Language and Linguistics Compass}, 12(1).

\bibitem[Robinson, 1949]{Robinson}
Robinson, J. (1949).
\newblock Definability and decision problems in arithmetic.
\newblock {\em Journal of Symbolic Logic}, 14(2):98--114.

\bibitem[Rothstein, 2010]{Rothstein}
Rothstein, S. (2010).
\newblock Counting and the mass/count distinction.
\newblock {\em Journal of Semantics}, 27(3):343--397.

\bibitem[S\'{a}nchez-Valencia, 1991]{Sanchez91}
S\'{a}nchez-Valencia, V. (1991).
\newblock {\em Studies on Natural Logic and Categorial Grammar}.
\newblock PhD thesis, Universiteit van Amsterdam.

\bibitem[Schrijver, 1998]{Schrijver}
Schrijver, A. (1998).
\newblock {\em Theory of Linear and Integer Programming}.
\newblock John Wiley \& Sons.

\bibitem[Schweikardt, 2005]{Schweikardt}
Schweikardt, N. (2005).
\newblock Arithmetic, first-order logic, and counting quantifiers.
\newblock {\em ACM Transactions on Compututational Logic}, 6(3):634–671.

\bibitem[Scott, 1965]{Scott1965}
Scott, D. (1965).
\newblock Logic with denumerably long formulas and finite strings of
  quantifiers.
\newblock In Addition, J., Henkin, L., and Tarski, A., editors, {\em The Theory
  of Models}, pages 329--341. North-Holland.

\bibitem[Seidenberg, 1959]{Seidenberg1}
Seidenberg, A. (1959).
\newblock A simple proof of a theorem of {Erd\H{o}s and Szekeres}.
\newblock {\em Journal of the London Mathematical Society}, s1-34(3):352.

\bibitem[{Sk\o lem}, 1938]{Skolem}
{Sk\o lem}, T. (1938).
\newblock {\em Diophantische Gleichungen}.
\newblock Ergebnisse der Mathematik und ihrer Grenzgebiete. Springer, Berlin.

\bibitem[Slomson, 1968]{Slomson}
Slomson, A. (1968).
\newblock The monadic fragment of predicate calculus with the {C}hang
  quantifier and equality.
\newblock In L{\"o}b, M.~H., editor, {\em Proceedings of the Summer School in
  Logic Leeds, 1967}, pages 279--301, Berlin, Heidelberg. Springer Berlin
  Heidelberg.

\bibitem[Steinert-Threlkeld and Icard, 2013]{Steinert}
Steinert-Threlkeld, S. and Icard, T.~F. (2013).
\newblock Iterating semantic automata.
\newblock {\em Linguistics and Philosophy}, 36(2):151--173.

\bibitem[Steinhorn, 1985]{STEINHORN1985161}
Steinhorn, C. (1985).
\newblock Borel structures for first-order and extended logics.
\newblock In Harrington, L., Morley, M., S\^{v}\^{e}drov, A., and Simpson, S.,
  editors, {\em Harvey Friedman's Research on the Foundations of Mathematics},
  volume 117 of {\em Studies in Logic and the Foundations of Mathematics},
  pages 161--178. Elsevier.

\bibitem[Sun and Liu, 2020]{LiuMonotone}
Sun, Z. and Liu, F. (2020).
\newblock The inference pattern \emph{Mou} in {M}ohist logic---a montonicity
  reasoning view.
\newblock {\em Roczniki Filozoficzne}, 68:257--270.

\bibitem[Szymanik, 2016]{Szymanik}
Szymanik, J. (2016).
\newblock {\em Quantifiers and Cognition: Logical and Computational
  Perspectives}.
\newblock Springer.

\bibitem[Tarski et~al., 1953]{Tarski}
Tarski, A., Mostowski, A., and Robinson, R.~M. (1953).
\newblock {\em Undecidable Theories}.
\newblock North-Holland Publishing Co.

\bibitem[V\"{a}\"{a}n\"{a}nen, 1977]{Vaananen}
V\"{a}\"{a}n\"{a}nen, J. (1977).
\newblock Remarks on generalized quantifiers and second-order logics.
\newblock In {\em Set Theory and Hierarchy Theory}, volume~14, pages 117--123.
  Prace Naukowe Instytutu Matematyki Politechniki Wroclawskiej, Wroclaw.

\bibitem[Visser, 2009]{Visser}
Visser, A. (2009).
\newblock Growing commas. a study of sequentiality and concatenation.
\newblock {\em Notre Dame Journal of Formal Logic}, 50(1):61 -- 85.

\bibitem[Westerst\r{a}hl, 1985]{Westerstahl1985}
Westerst\r{a}hl, D. (1985).
\newblock Logical constants in quantifier languages.
\newblock {\em Linguistics and Philosophy}, 8:387--413.

\end{thebibliography}

\newpage

\appendix

In these appendices we present some additional material that broadens the context for the main results of this paper.  Appendix \ref{app:related} is a survey of relevant literature. Appendices \ref{app:infinityquantifier}, \ref{app:infinityaddition}, and \ref{app:automata} present the details on some results mentioned in the main text, concerning infinity quantifiers and monadic second-order logic, infinitary addition and multiplication, and semantic automata, respectively. Finally, Appendix \ref{appendix:syntax} highlights an  intriguing interface  of logic and counting that we have largely ignored in this paper, namely, the historical tradition of results on the entanglement of the very syntax of logical systems and systems of arithmetic.

\renewcommand\thesection{\Alph{section}}

\section{Related Work on Logic and Counting} \label{app:related} 
As we have mentioned, there is a vast amount of important research on mixtures of logic and counting. Here we discuss logical systems  in the literature that bear a close relationship to the hierarchy of systems studied here (summarized in Tables \ref{table:overview} and \ref{table:overview2}).

\subsubsection*{Logics with Generalized Quantifiers} An expansive literature has explored adding generalized quantifiers to first-order logic (as well as other languages, including monadic first-order logic). The system $\mathsf{FO}(\#)$ has been studied explicitly in that literature \citep{hartig,antonelli2010,PetersWesterstahl}, and of course it is closely related to both the H\"{a}rtig quantifier, $\#_x\varphi \approx \#_x\psi$, and the strict version $\#_x\varphi \succ \#_x \psi$ introduced explicitly by \cite{Lindstrom}. 
Earlier, \cite{rescher} had considered a unary version, namely, 
$\#_x\varphi \succ \#_x\neg \varphi$. 

Work on the \emph{monadic} fragment of $\mathsf{FO}$ with generalized quantifiers dates back at least to \cite{Slomson}, who studied the Chang quantifier, $\#_x\varphi \approx \#_x\top$, in this context. We refer to \cite{PetersWesterstahl} for many other results and references in the area related to these particular generalized quantifiers, both for $\mathsf{FO}$ and for $\mathsf{MFO}$.

\subsubsection*{Computational Logic} Perhaps the largest body of work related to our systems comes from computation logic. A significant strand focuses on \emph{extensions} of $\mathsf{FO}(\#)$ and even of $\mathsf{FO}(\sharp)$, but interpreted over finite models (e.g., \citealt{Immerman1992,Grumbach,Schweikardt,Kuske}, among many others).  As discussed in Remark \ref{remark:finvar}, much is known about \emph{finite variable} fragments with counting quantifiers as well, though here most of the results are negative \citep{Otto,Gradel1999,Kieronski}. Back-and-forth games, similar in spirit to our $\#$-bisimulations (Definition \ref{def:bisim}), have also been explored (see, e.g., \citealt{Immerman1992,Otto}).

\subsubsection*{Syllogistic and Propositional Counting Logic} A number of weak fragment of $\mathsf{MFO}(\#)$ and even of $\mathsf{PL}(\#)$ have been studied as \emph{extended syllogistic systems}. For example, a whole series of papers charts the territory of small systems including `more than', `most', `at least $k$', and related operators \citep{Hartmann2008,Hartmann2009,Moss2016,LaiEndrullisMoss,EndrullisMoss,moss_topal_2020,KisbyMoss}. \cite{Hartmann2008} in particular explores $\mathsf{FO}(\#)$ with \emph{one} free variable, which is seen to be decidable. He also notes a natural probabilistic interpretation of the system. Locating precisely where these systems fit inside of our logics would be worthwile. Notably, many of them enjoy quite low complexity. 

Recent work by \cite{Ding2020} essentially deals with what we call $\mathsf{PL}(\#)$, interpreted over (possibly) infinite models.  As highlighted in Table \ref{table:overview2}, the main difference between $\mathsf{PL}(\#)$ and sentences in $\mathsf{MFO}(\#)$ is the ability of the latter to express inequalities with numerical bounds. An important instance is $\mathsf{s} \geq  \mathsf{s} +1$, showing that $\mathsf{MFO}(\#)$, unlike $\mathsf{PL}(\#)$, can characterize the infinite predicates. However, the higher expressive power of numerical bounds also marks an important distinction in the valid principles. 

For instance, the main principle in one of the axiomatizations from \cite{Ding2020} employs a type of \emph{polarization rule} \citep{kps59,Burgess2010}. Adapted to our setting, provided the predicate $P$ occurs nowhere in $\varphi$ or $\psi$, the rule would say: \begin{quote} From $\#_x\big(\varphi \wedge P(x)\big)\approx \#_x\big(\varphi \wedge \neg P(x)\big) \rightarrow \psi$, infer $\psi$. \hfill (\emph{Polarization}) \end{quote} Polarization is not admissible even in our basic system $\mathsf{MFO}(\#)$. It implies, amongst other things, that consistent formulas can also be made true while duplicating the size of all regions. This is true for sets of inequalities without numerical bounds, but not for the ones expressible in $\mathsf{MFO}(\#)$.
As discussed in \S\ref{section:axioms}, it remains to be seen whether a more intricate polarization rule for $\mathsf{MFO}(\#)$ would support a ``purely logical'' axiomatization. 


\subsubsection*{Probability Logic} We mentioned a connection with probability logic in \S\ref{section:probprop}, namely, the systems $\mathsf{PL}^\phi(\#)$, $\mathsf{ML}^\phi(\#)$, $\mathsf{MFO}^\phi(\#)$, and $\mathsf{MSO}^\phi(\#)$ can all be interpreted probabilistically without any further ado, viz. proportionality. Under that interpretation, $\mathsf{PL}^\phi(\#)$ is indistinguishable from the propositional probability logic considered in \cite{Hoek1996b}, which is equivalent to the system studied earlier by \cite{Gardenfors1975}, provided the latter is restricted to \emph{regular} probability measures, i.e., those assigning all non-empty sets strictly positive probability. $\mathsf{MSO}^\phi(\#)$ is easily seen to be equally expressive as the probability logic with linear inequalities studied by \cite{Fagin}, again under the assumption of regularity. For discussion of regularity in probability logic, see \cite{Ding2021}. 

A very strong probability logic was studied in \cite{Bacchus} and \cite{Halpern}, allowing inequalities between sums and products of terms $\pi_\mathbf{x}\varphi$ (cf. \S\ref{section:explicitarithmetic}). While our polyadic terms $\sharp_\mathbf{x} \varphi$ in $\mathcal{L}_{\sharp}^1$ and $\mathcal{L}_{\sharp}^2$ are interpreted as cardinalities of Cartesian products, these terms $\pi_\mathbf{x}\varphi$ are interpreted directly as products of probabilities, which in general leads to a different set of principles (cf. Example \ref{example:probprop}). Quantifiers over term variables are also allowed. Unsurprisingly, these languages are highly undecidable, although decidable fragments can be found, e.g., by allowing only monadic predicates and eliminating variable equality \citep{Halpern}.

\subsubsection*{Graded Modal Logic} In the areas of modal and description logics, a number of authors (since \citealt{Fine}) have considered graded modal logics involving unary modalities like $\Diamond^{\geq k}$. We mentioned that $\mathsf{ML}(\#)$ cannot express these modalities (Corollary \ref{cor:grade}), but of course the reverse is also true: the binary modality $\succsim$ is beyond the expressive capacity of graded modal logic. A broad study, with connections to generalized quantifiers, appears in \cite{Hoek1993}. More recently, some researchers have probed the precise counting capacity of such systems, employing notions of count-bisimulations as well (see, e.g., \citealt{Baader}). There has also been study of related logical systems that are expressively equivalent to, but more complex than, graded modal logic \citep{Bednarczyk}, as well as natural expressive extensions that remain of relatively low complexity \citep{DEMRI2010233}. Emerging connections between graded modal logic and classes of \emph{graph neural networks} \citep{Barcelo} promise yet further dimensions to our subject.

\section{The Infinity Quantifer and Monadic Second-Order Logic} \label{app:infinityquantifier}
Let \textsf{MFO}$^\infty$ be monadic first order logic with an infinity quantifier (simply the language $\mathcal{L}_\#^1$ without $\#$-formulas but with $\exists^\infty$ added), and let \textsf{WMSO} be weak monadic second order logic (quantification only over finite sets). It turns out \textsf{MFO}$^\infty$ and \textsf{WMSO} are expressively equivalent. A version of this result without equality is due to \cite{Vaananen}, and here we describe the result with equality. To translate \textsf{MFO}$^\infty$ into \textsf{WMSO} the only interesting case is $(\exists^\infty y. \varphi)^* = \forall X.\exists y.\big(\neg X(y) \wedge (\varphi)^*\big)$. In the other direction, \textsf{MFO}$^\infty$ possesses a normal form result \citep[Thm 3.15]{Carreiro} whereby every sentence is equivalent to a disjunction of existentially quantified formulas of the form: $$\mathsf{diff}(\mathbf{x}) \wedge \bigwedge \tau(x_i) \wedge \forall z. (\mathsf{diff}(\mathbf{x},z) \rightarrow \bigvee \sigma(z)) \wedge \bigwedge \exists^\infty y.\rho(y) \wedge \forall^\infty y. \bigvee \upsilon(y).$$ Supposing that $X$ is one of our monadic predicates, assuming it can only take on finite sets as values, the above is equivalent to one of the form: \begin{eqnarray*} & & \alpha(\mathbf{x}) \wedge \forall z. \big(\mathsf{diff}(\mathbf{x},z) \rightarrow (X(z) \rightarrow \psi(z)) \wedge (\neg X(z) \rightarrow \chi(z))\big) \\
& & \wedge \bigwedge \exists^\infty y.(\neg X(z) \wedge \rho(y)) \wedge \forall^\infty y. (\neg X(y) \rightarrow \varphi(y)). \end{eqnarray*} Because $\exists X$ commutes with $\exists \mathbf{y}$ and disjunction, we need only consider what happens when appending $\exists X$ to this formula. This is evidently equivalent to another formula with no occurrences of $X$ at all: $$\alpha'(\mathbf{x}) \wedge \forall z. \big(\mathsf{diff}(\mathbf{x},z) \rightarrow (\psi(x) \vee \chi(x))\big) \wedge \bigwedge \exists^\infty y.(\rho(y) \wedge \chi(y)) \wedge \forall^\infty y.(\varphi(y) \wedge \chi(y)).$$ This concludes the argument for the other direction.

\section{Cardinal Arithmetic: Quantifier Elimination and Separation} \label{app:infinityaddition}

Consider the elementary theory of the structure $\mathcal{C} = \langle C_{\aleph_\omega};+\rangle$, that is, the first-order theory of addition on cardinal numbers less than $\aleph_\omega$. As in ordinary Presburger Arithmetic, $\{0\}$, $s$, $\equiv_n$ and $>$ are all definable in this structure, where $s$ is the function that takes a cardinal number to the next largest cardinal number, 
and $\equiv_n$ is congruence mod $n$, for $1<n<\omega$. Note that $\{\aleph_0\}$ is also definable. Assume we have all of these constants, functions, and relations in the signature, so we are considering  $\mathcal{C}^+ = \langle C_{\aleph_\omega};0,\aleph_0,s,\{\equiv_n\}_{1<n<\omega},>,+\rangle$. 


We first derive a normal form for the quantifier-free fragment. By propositional reasoning we assume a disjunction of conjunctions of atomic formulas: \begin{eqnarray*}
\mathsf{t} & = & \mathsf{u} \\
\mathsf{t} & \equiv_m & \mathsf{u} \\
\mathsf{t} & > & \mathsf{u} 
\end{eqnarray*} and also by propositional reasoning we can assume that every disjunct includes a conjunct $x < \aleph_0$ or $x \geq \aleph_0$, for every variable $x$ appearing in the disjunct. This allows us to separate the atomic formulas into those involving ``finite'' terms and those involving ``infinite'' terms: the successor function of course takes (in)finite to (in)finite cardinals, and infinite terms absorb finite terms in sums. Furthermore, if either $\mathsf{t}$ or $\mathsf{u}$ contains an infinite term, then we can assume without loss that both $\mathsf{t}$ and $\mathsf{u}$ contain only infinite terms, since otherwise all three types of atomic formulas trivialize. In other words, we have obtained a normal form characterized by disjunctions of conjunctions which include statements about which variables are finite/infinite, a set of statements describing the finite terms, and a set of statements describing the infinite terms.

The finite component can, as usual, be further regimented so that the three types of atomic statements involve sums of terms of the form $s^k(0)$ and $s^k(x)$ for $k\geq0$ and $x$ a variable. This is because of the law $s(x+y)=x+s(y)$. As usual, models of these  conjunctions are effectively solutions to linear programs. 

For the infinite component, successor in fact distributes over addition, that is, $s(x+y) = s(x)+s(y)$, which allows a similar regimentation. More regimentation is possible. First note that $\equiv_n$ can everywhere be replaced by $=$. But we can also eliminate all sums. For instance, $\mathsf{t}=\mathsf{u}+\mathsf{v}$ is equivalent to the disjunction $(\mathsf{t}=\mathsf{u} \wedge \mathsf{u}\geq \mathsf{v}) \vee (\mathsf{t}=\mathsf{v} \wedge \mathsf{v} > \mathsf{u})$. The same reduction works for strict inequalities. 

Thus, the component describing the infinite terms simply contains conjuncts of the form $x = s^k(y)$, $x>s^k(y)$, $x = \aleph_k$, and $x > \aleph_k$, for $k\geq 0$. 
There is a trivial isomorphism 
from $\langle \mathbb{N};0,s,> \rangle$ onto 
$\langle \{\aleph_k\}_{k \in \mathbb{N}};\aleph_0,s,>\rangle$ sending $k$ to $\aleph_k$. This shows that the definable subsets of infinite cardinals coincides with the definable sets of indices in $\mathbb{N}$, viz. the finite and co-finite sets. This of course also easily establishes the decidability of determining whether a quantifier-free formula in the original language is satisfiable. Summarizing:
\begin{proposition} \label{addition-normalform} Every first-order quantifier-free formula is equivalent over the structure $\mathcal{C}^+ = \langle C_{\aleph_\omega};0,\aleph_0,s,\{\equiv_n\}_{1<n<\omega},>,+\rangle$ to a disjunction of conjunctions, specifying:\end{proposition} \begin{enumerate}
    \item which variables in that disjunct are finite or infinite
    \item for the finite component a description of a linear set, and
    \item for the infinite component a description of a set of infinite cardinals using $0,s,>$ over the aleph-number indices. 
\end{enumerate}
\begin{corollary} The quantifier-free theory of $\mathcal{C}^+$ is decidable. 
\end{corollary}

What about the full first-order theory of $\mathcal{C}$? As in ordinary Presburger Arithmetic, this theory does not admit quantifier elimination. But the theory of $\mathcal{C}^+$, in the augmented language, does. Consider a formula $\exists x.\theta$, where $\theta$ is in normal form (Proposition \ref{addition-normalform}), i.e., $\theta$ is a conjunction $\delta\wedge\iota\wedge\phi$, where $\delta$ is a description of which variables are (in)finite, $\iota$ describes the infinite terms, and $\phi$ describes the finite terms. In our normal form $x$ does not appear in both $\iota$ and $\phi$, so $\exists x. \theta$ simplifies to either $\exists x.\iota$ or $\exists x.\phi$, where $\iota$ and $\phi$ are assumed to involve only infinite or finite terms, respectively. In the latter case we can perform the quantifier elimination as usual in additive arithmetic, reducing $\exists x.\phi$ to a quantifier free statement using $0,s,>,+$, and the congruence relations $\equiv_m$.

In the former case we want to show that we can reduce $\exists x. \iota$ to a quantifer-free form using only $\aleph_0$, $s$, and $>$. In fact, this proceeds exactly as the quantifier elimination procedure for $\langle \mathbb{N};0,s,> \rangle$: the isomorphism between the latter structure and 
$\langle \{\aleph_k\}_{k \in \mathbb{N}};\aleph_0,s,>\rangle$ shows they have the same quantificational theory as well.

Having shown quantifier elimination for $\mathcal{C}^+$, this establishes: \begin{theorem} \label{thm:decidablecardinal}
The first-order theory of $\mathcal{C}$ is decidable.
\end{theorem}

We now show essentially the same result for full first-order arithmetic over cardinals. That is, let $\langle C_{\aleph_{\omega}}; +,\times\rangle$ be the structure of cardinal numbers less than $\aleph_{\omega}$ under addition and multiplication. The first-order theory of this structure is of course undecidable, but it is easy to see that this is only due to the substructure $\langle \mathbb{N}; +,\times\rangle$. As before, this substructure is definable in the sense that a term $\mathsf{t}$ denotes a natural number if and only if $\mathsf{t}+1>\mathsf{t}$. Indeed, by the same argument as above, any formula will be equivalent to a disjunction of conjunctions $\delta \wedge \iota \wedge \phi$, where $\delta$ specifies which terms are (in)finite, $\iota$ involves the infinite terms, and $\phi$ the finite terms. 

The $\phi$ component will be an arbitrary arithmetical formula, where quantifier elimination of course fails. But the $\iota$ component does allow for quantifier elimination. That is, we can consider the elementary theory of $\langle \{\aleph_{k}\}_{k \in \mathbb{N}}; +,\times\rangle$. The crucial step is the same as in the purely additive case: every equality statement $\mathsf{t}=\mathsf{u}+\mathsf{v}$ or $\mathsf{t}=\mathsf{u}\times \mathsf{v}$ is equivalent over this structure to the disjunction $(\mathsf{t}=\mathsf{u} \wedge \mathsf{u} \geq \mathsf{v}) \vee (\mathsf{t}=\mathsf{v} \wedge \mathsf{v} > \mathsf{u})$ (and similarly for strict inequalities between complex terms), implying that we can systematically eliminate both addition and multiplication. Thus, quantifier elimination for the language augmented with constant $\aleph_0$ and successor $s$ again follows from the fact that $\langle \mathbb{N};0,s,> \rangle$ admits it. 
\begin{theorem} Every formula in the language of first-order arithmetic is equivalent over $\langle C_{\aleph_{\omega}}; +,\times\rangle$ to a disjunction of conjunctions involving a finite and an infinite component. Moreover, the set of ``infinitary formulas'' (all of whose terms are declared infinite) possesses quantifier elimination and they define precisely the same relations over cardinals as the pure language of equality and strict inequality. 
\end{theorem}

\section{Finite Automata and Quantifier Recognition Procedures} \label{app:automata}
 \emph{Finite automata} are particularly simple counting devices, and in what follows, we will  determine what binary logical quantifiers this device can recognize. We first recall the main definitions and statement of the result from \S\ref{section:automata}. Linear sets are the solutions to equations 
\begin{equation}
 \begin{pmatrix}
\;v_1\;  \\
v_2
\end{pmatrix} \hspace{.1in} = \hspace{.1in}
 \begin{pmatrix}
b_1   +  i_{1,1}\mathsf{x}_1+ \dots + i_{1,m}\mathsf{x}_m   \\
b_2   +  i_{2,1}\mathsf{x}_1+ \dots + i_{2,m}\mathsf{x}_m
\end{pmatrix}
\end{equation} while rectilinear sets are those in which $i_{1,k} = 0$ or $i_{2,k}=0$  (or both) for all all $k \leq m$. 
Finally, a set is \emph{semi-rectilinear} if it is a finite union of rectilinear sets. For the purpose of this appendix we will notate the linear sets by $(v_1,v_2) + \mathsf{x}_1.(i_{1,1},i_{2,1})+\dots+\mathsf{x}_m.(i_{1,m},i_{2,m})$. So in this notation rectilinear sets can be seen as defined by forms $(v_1,v_2) + \mathsf{x}_k.(i_k,0) + \dots + \mathsf{x}_j.(0,i_j)$, which first lists the periods of 
type $(i, 0)$ and then those of type $(0, i)$. Our result states: 






\begin{theorem} 	The following are equivalent  for permutation-closed languages $L$: \begin{enumerate}[label=(\alph*)]
  \item \label{pcl1} $L$ is regular, 
  \item \label{pcl2} The set of occurrence vectors for strings in $L$ is semi-rectilinear.
\end{enumerate}
	\end{theorem}

\begin{proof}  The idea of the proof is to associate semi-rectilinear forms with finite automata. In showing how this works, we  shall be using geometrical representations in a number of places which are like the tree of numbers for generalized quantifiers (\S\ref{section:quantifiers}), except for a rotation to the grid $\mathbb{N} \times \mathbb{N}$ which fits our purposes better. In fact, the terminology ``rectilinear'' was motivated by shapes in this grid. Also, we shall be using several well-known useful properties of finite automata, such as the closure under unions of the languages recognized, the fact that nondeterministic finite automata have the same recognizing power as deterministic ones, or the fact that the recognizing power of deterministic finite automata is not changed when we allow 0, 1 or more transitions for a symbol read in some states.

\subsubsection*{From \ref{pcl2} to \ref{pcl1}} It suffices to show the implication for rectilinear forms, since the permutation-closed regular languages are closed under taking unions. 

There are a few special cases here that are easily shown to be regular, namely, a single vector $(v_1, v_2)$, or such a vector plus one period $(i, 0)$ or $(0, i)$ with $i \neq 0$. Before starting the main proof, here is a warm-up example.

\begin{example}   The rectilinear form $(1, 2) + \mathsf{x}.(2, 0)$ matches the permutation-closed regular language of strings with an odd number of symbols $\mathtt{a}$ and two occurrences of  symbol $\mathtt{b}$. The following  finite automaton recognizes just these strings.
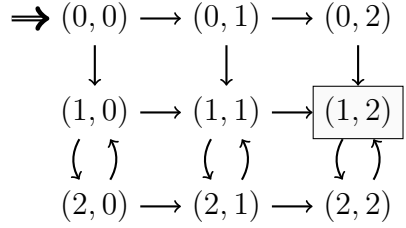
\begin{center}
\begin{tikzpicture}
\node (0) at (-1.25,0) {};
\node (00) at (0,0) {$(0,0)$};
\node (10) at (0,-1.25) {$(1,0)$};
\node (11) at (1.75,-1.25) {$(1,1)$};
\node (12) [draw=black,fill=gray!4] at (3.5,-1.25) {$(1,2)$};
\node (01) at (1.75,0) {$(0,1)$};
\node (02) at (3.5,0) {$(0,2)$};
\node (20) at (0,-2.5) {$(2,0)$};
\node (21) at (1.75,-2.5) {$(2,1)$};
\node (22) at (3.5,-2.5) {$(2,2)$};
\path (0) edge[->,thick,double] (00);
\path (00) edge[->,thick] (10);
\path (00) edge[->,thick] (01);
\path (01) edge[->,thick] (02);
\path (01) edge[->,thick] (11);
\path (10) edge[->,thick] (11);
\path (11) edge[->,thick] (12);
\path (02) edge[->,thick] (12);
\path (20) edge[->,thick] (21);
\path (21) edge[->,thick] (22);
\path (10) edge[->,thick,bend right] (20);
\path (20) edge[->,thick,bend right] (10);
\path (11) edge[->,thick,bend right] (21);
\path (21) edge[->,thick,bend right] (11);
\path (12) edge[->,thick,bend right] (22);
\path (22) edge[->,thick,bend right] (12);
\end{tikzpicture}
\end{center}
Horizontal arrows are for $\mathtt{b}$-moves, vertical arrows for $\mathtt{a}$-moves, rightmost states allow no $\mathtt{b}$-moves, the starting state is $(0, 0)$, and the only accepting state is $(1, 2)$. Here are two illustrations. (a) It is easy to see that a state $(i, j)$ can only be visited after having seen $j$ occurrences of  $\mathtt{b}$ plus a number of occurrences for  $\mathtt{a}$ that equals $i$ plus some multiple of 2 (reflecting the available cyclic detours). (b) A correct string such as  $\mathtt{a}^{5} \mathtt{b} \mathtt{a}^{3} \mathtt{b}\mathtt{a}^{5} $ can be recognized by first cycling through $(1, 0)$ and $(2, 0)$ ending in $(1, 0)$, then moving to $(1, 1)$, then cycling through  $(1, 1)$ and $(2, 1)$ ending in $(2,1)$, then moving to $(2, 2)$, and finally cycling through $(2, 2)$ and $(1, 2)$ ending in $(1, 2)$. The general principle should be clear. Taken together, (a) and (b) show that the automaton recognizes the given language. Incidentally, the automaton is not unique. The preceding reasoning would yield the same conclusion if we had allowed cycling between the top and middle layers of the state transition diagram.

\end{example}



		    
		
	   

Next, consider a general rectilinear form 
\begin{eqnarray*} F  & = &    (v_1, v_2) + \mathsf{x}_k.(i_k,0) + \dots + \mathsf{x}_j.(0,i_j)    
\end{eqnarray*}
Let $N_1$ be the sum of $v_1$ plus the maximum of all numbers $ i_k$ occurring to the left in periods of $F$, while $N_2$ is defined likewise using the right-hand side of the pairs occurring in $F$.  Now we define a non-deterministic partial finite automaton $\mathcal{S}$:\vspace{0.5ex}
\begin{itemize}
  \item States are all pairs $(u, v)$ with $u \leq N_1, v \leq N_2$, \vspace{1mm}
    \item The only recognizing state is $(v_1, v_2)$, \vspace{1mm}
     \item The transition function is defined as follows, with two types of moves:\vspace{1mm}
     \begin{enumerate}[label=\Roman*.]
         \item \label{rom1} from $(x, y)$ via reading $\mathtt{a}$ to $(x+1, y)$, if $(x+1, y)$ is a state, \\ and analogously for reading $\mathtt{b}$,\vspace{1mm}
         \item \label{rom2} an $\mathtt{a}$-move from state $(x + i - 1, y)$ to $(x , y)$, if the period $(i, 0)$ occurs 
	in $F$. Likewise for $\mathtt{b}$-moves and periods $(0, j)$.\vspace{0.5ex}
     \end{enumerate} 
\end{itemize}

We say that an automaton $\mathcal{S}$ is \emph{permutation invariant} if, whenever reading a string 
	X can drive $\mathcal{S}$ from state $S$ to state $T$,  any permuted version of X can also 
	drive $\mathcal{S}$ from state $S$ to state $T$.
The following can be shown by direct inspection of  the above-defined transitions.

\begin{fact} 	The automaton $\mathcal{S}$ is permutation invariant. \end{fact}

\begin{proof} It suffices to show that $\mathtt{a}$ and $\mathtt{b}$ transitions can be interchanged at an input state without changing the output state. This is easily established by considering the various combinations of Type \ref{rom1} transitions and Type \ref{rom2} transitions.
\end{proof}
\begin{lemma} The following assertions are equivalent: \begin{enumerate}[label=(\roman*)]
    \item \label{lemma:rec1} String $X$ is recognized by the above-defined automaton $\mathcal{S}$,
    \item \label{lemma:rec2} the occurrence numbers 
	for $\mathtt{a}, \mathtt{b}$ in $X$ are in the set defined by the rectilinear form $F$.
\end{enumerate}\end{lemma}

\begin{proof}   	\emph{From \ref{lemma:rec1} to \ref{lemma:rec2}}. Suppose that a string $X$ drives $\mathcal{S}$ from the starting state to the accepting state $(v_1, v_2)$.  We prove the following stronger invariance statement by induction on the length of finite strings:

\vspace{1.5ex}

\noindent{\bf Claim.}	\, If string $X$ drives $\mathcal{S}$ to state $(x, y)$, then the occurrence numbers in $X$  are generated by $(x, y)$ plus a (possibly empty) finite sum of periods occurring  in the  rectilinear form $F$.

\begin{proof}[Proof of Claim] The claim is clear for the empty string at the starting state $(0, 0)$. (Here we use the fact that our automaton $\mathcal{S}$ as defined above has no $\epsilon$-moves except the identity.)

The inductive step is by inspecting  possible transitions. We discuss $\mathtt{a}$-moves only, $\mathtt{b}$-moves are similar. (a) Suppose that $X\mathtt{a}$ drives the initial state of $\mathcal{S}$ to $(x, y)$, and then moves to $(x+1, y)$ by reading the final $\mathtt{a}$. By the inductive hypothesis about $X$, the occurrence numbers  match the stated description at the state $(x, y)$. But then the occurrence numbers for $X\mathtt{a}$ satisfy that same description with respect to $(x+1, y)$. (b) Now suppose that $X\mathtt{a}$ first reaches $(x + i - 1, y)$ in $\mathcal{S}$, and then moves to $(x, y)$ by reading the final $\mathtt{a}$. By the inductive hypothesis, the occurrence numbers in $X$ match the stated description at $(x + i - 1, y)$. But then, since by the definition of $\mathcal{S}$ there is a period $(i, 0)$ in $F$ allowing a cyclic move, the occurrence numbers for $X\mathtt{a}$ satisfy the stated description at the state $(x, y)$. 	\end{proof}

In particular, once the accepting state is reached, the string must have a pair of occurrence numbers in the given rectilinear set.	

\vspace{1ex}

	\emph{From \ref{lemma:rec2} to \ref{lemma:rec1}}. Let string $X$ have occurrence numbers in the given rectilinear set, with particular values for the period variables $x$. By the permutation-invariance of the automaton $\mathcal{S}$, the string $X$ will be recognized iff the following permuted version is recognized: ``first $v_1$ symbols $\mathtt{a}$, then $v_2$ symbols $\mathtt{b}$ (i), then the remaining symbols $\mathtt{a}$ followed by the remaining $\mathtt{b}$ (ii).'' Part (i) of this sequence takes us to the recognizing state $(v_1, v_2)$. The symbols in the final Part (ii) can be discounted by making the appropriate looping moves corresponding to admissible periods, always returning toward $(v_1, v_2)$.    \end{proof}

\subsubsection*{From \ref{pcl1} to \ref{pcl2}}  Consider any permutation-closed regular language $\mathcal{L}$. First, we produce a suitable automaton to work with in the rest of the proof.

\begin{fact}  	 $\mathcal{L}$ is recognized by a permutation-invariant deterministic finite automaton $\mathcal{S}$.
\end{fact}

\begin{proof}  Consider the standard Nerode construction for regular languages, where two strings are called equivalent if they send the same continuations to accepting states. A recognizing deterministic finite automaton for the language has the equivalence classes for its states, and a transition function plus  accepting states defined in an obvious manner. Now, it suffices to note the simple fact that, if the regular language we start with is itself permutation-closed, then the Nerode automaton is permutation-invariant in the earlier sense. \end{proof}


The permutation invariance allows us to define, for each pair of numbers $(i, j)$, a unique state $S_{ij}$ that $\mathcal{S}$ will reach from its starting state when presented with any string with these occurrence numbers. We call $(i, j)$ accepting iff $S_{ij}$ is. While not strictly necessary for what follows, it is helpful to think of our two structures abstractly as two  bimodal relational models: $\mathcal{S}$ and its ``grid unraveling'' $\mathbb{N} \times \mathbb{N}$ which carries two commuting functions ``moving one step up'' and ``moving one step right''. Then the following connection arises:

\begin{fact}	$S_{ij}$ is a modal $p$-morphism from the grid $\mathbb{N} \times \mathbb{N}$ to the automaton $\mathcal{S}$.\end{fact}

We can therefore consider the grid model $\mathbb{N} \times \mathbb{N}$ as an automaton that is equivalent to $\mathcal{S}$ in an obvious sense, and analyze its geometrical shape.  

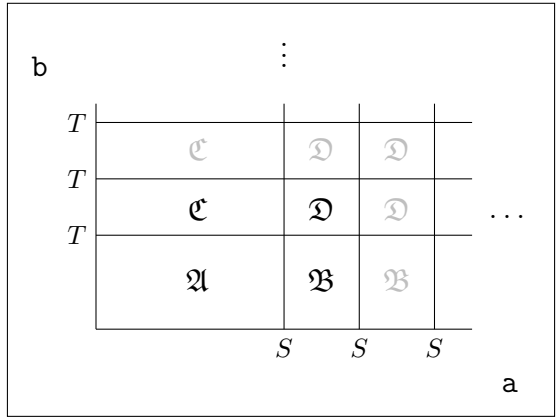
\begin{center}
\begin{tikzpicture}[framed]
\node at (-.75,3.5) {$\mathtt{b}$};
\node at (2.5,3.75) {$\vdots$};
\node at (5.5,-.75) {$\mathtt{a}$};
\node at (5.5,1.5) {$\dots$};
\node at (-.25,1.25) {\footnotesize{}$T$};
\node at (-.25,2) {\footnotesize{}$T$};
\node at (-.25,2.75) {\footnotesize{}$T$};
\node at (2.5,-.25) {\footnotesize{}$S$};
\node at (3.5,-.25) {\footnotesize{}$S$};
\node at (4.5,-.25) {\footnotesize{}$S$};
 \path (0,0) edge[-] (5,0);
 \path (0,0) edge[-] (0,3);
 \path (2.5,0) edge[-] (2.5,3);
 \path (3.5,0) edge[-] (3.5,3);
 \path (4.5,0) edge[-] (4.5,3);
  \path (0,1.25) edge[-] (5,1.25);
  \path (0,2) edge[-] (5,2);
\path (0,2.75) edge[-] (5,2.75);
\node at (1.35,.65) {$\mathfrak{A}$};
\node at (3,.65) {$\mathfrak{B}$};
\node at (4,.65) {\textcolor{gray!50}{$\mathfrak{B}$}};
\node at (1.35,1.6) {$\mathfrak{C}$};
\node at (1.35,2.4) {\textcolor{gray!50}{$\mathfrak{C}$}};
\node at (3,1.6) {$\mathfrak{D}$};
\node at (3,2.4) {\textcolor{gray!50}{$\mathfrak{D}$}};
\node at (4,2.4) {\textcolor{gray!50}{$\mathfrak{D}$}};
\node at (4,1.6) {\textcolor{gray!50}{$\mathfrak{D}$}};
\end{tikzpicture} \end{center}
\subsubsection*{Explanation of the grid automaton} The two  symbols $\mathtt{a},\mathtt{b}$ represent the functions in this grid model. The state $S$ is the first recurring state as we start reading symbols $\mathtt{a}$ only from the starting state. Each interval from $S$ to $S$ on the bottom row is then the same. And the same is true for their matching intervals on horizontal rows higher up, as these arise from applying the function $\mathtt{b}$ the same number of times to identical states. In particular, the rectangles toward the right in the area $\mathfrak{B}$ are all the same. The same analysis works for the first recurring state $T$ on the left w.r.t. the $\mathfrak{C}$ area. Next, the area $\mathfrak{A}$ can have arbitrary state content, but it is finite, since non-recurring state sequences are bounded in length by the size of the given automaton $\mathcal{S}$. Finally, the rectangle $\mathfrak{D}$ is very special. All its corner points must be the same (given their origins from the $S$ and the $T$ intervals), and $\mathfrak{D}$ will then repeat to fill the whole remaining quadrant of $\mathbb{N} \times \mathbb{N}$ with identical copies of itself.  \vspace{1ex}

Now consider any recognizing state $U$ in $\mathcal{S}$. Its occurrences in  above grid can be described as follows, area by area in the diagram. The typical features of rectilinear forms now emerge. In area $\mathfrak{A}$: a finite disjunction of descriptions of single  vectors. In area $\mathfrak{B}$: a finite disjunction of occurrences of $U$ in the first rectangle, plus periods $\mathsf{x}.(k, 0)$ where $k$ is the length of the first interval from $S$ to $S$. For area $\mathfrak{B}$ the enumeration is analogous with a period $\mathsf{x}.(0, l)$ for moving upward. Finally, for area $\mathfrak{D}$, all occurrences of $U$ in its quadrant are described by a finite disjunction of their occurrences in the first generating rectangle while allowing both periods $\mathsf{x}.(k, 0)$ and $\mathsf{x}.(0, l)$. In particular, no ``oblique'' periods $\mathsf{x}.(i, j)$ (like the period $\mathsf{x}.(1, 1)$ used in defining  the non-regular quantifier `most')  are needed for this enumeration.

\vspace{1ex}

The preceding descriptions, taken disjunctively over all occurrences of accepting states in the grid, shows that the permutation-closed language recognized by the given automaton $\mathcal{S}$ has a semi-rectilinear description.        
\end{proof}

The earlier-mentioned characterization of first-order quantifiers \citep{vanBenthem1985} is a special case, 
where the crucial area $\mathfrak{D}$ collapses to one state whose behavior then extends downward. As for generalizations, the result probably also holds for arbitrary \emph{finite alphabets}, given the affinities of our treatment with the graph-theoretic analysis of permutation-closed regular languages over arbitrary alphabets in \cite{Hoffman} (cf. \citealt{EHRENFEUCHT1983311}). See in addition \cite{Kanazawa}, who also gives an arithmetical description of the permutation-invariant languages recognized by pushdown automata.

Here are a few questions raised by our results and proof method. In terms of other formats, what is the structure of the special \emph{regular expressions} that describe permutation-invariant finite automata, and what algebraic laws govern their manipulation? Rectilinear forms amount to a flattening of nested iterations to just one level, which is reminiscent of the flattening of nested count terms in  the normal forms for $\mathsf{MFO(\#)}$. Also, could the modal perspective in the above proof yield further insights? In particular,  the use  of the grid $\mathbb{N} \times \mathbb{N}$ might be significant, in that its decoration with a finite set of states is a form of a \emph{tiling}, while modal logics of tiling problems have high complexity.
Next, connecting back to our  counting logics, another natural question is this. Are the above results reflected in \emph{arithmetical definability} results for finite-state quantifiers, whether in terms of the inequalities in normal forms for $\mathsf{MFO(\#)}$ or directly in the first-order language of Presburger Arithmetic? Finally, our counting logics  typically allow for \emph{infinite cardinalities}. Can the above automata analysis be extended to  infinite cardinalities, perhaps using B\"{u}chi automata for infinite strings?

\section{Logical Syntax and Counting} \label{appendix:syntax}
In addition to the mixtures of logic and counting discussed in this paper, here is one more perspective, with a long history. Working with a logical system presupposes an understanding of its \emph{syntax}. But syntax is a combinatorial entity, and syntactic manipulations are very close to computing. We saw hints of this whenever we encountered \emph{counting in the syntax} (e.g.,  Example \ref{example:php}, Remark \ref{remark:mlsr}). But the connection goes much deeper.  Counting and arithmetic start as soon as we introduce a logical system, even in defining the set of well-formed expressions of the language, not to mention in our specifications for what counts as a legal \emph{proof derivation}. This potentially ``vicious circle'' was already emphasized by \cite{Hilbert} toward the very beginning of modern logic: ``In the usual exposition of the laws of logic certain fundamental concepts of arithmetic are already employed, for example the concept of the aggregate, in part also the concept of number'' (p. 347).

Subsequently work revealed a deep and precise sense in which syntax and counting are indeed two sides of the same coin. For instance, echoing related ideas from Tarski, Hermes, L\"{o}b, and others, \cite{Quine} showed that the first-order theory of the natural numbers (i.e., ``true arithmetic'') is in fact bi-interpretable with the first-order theory of \emph{concatenation} of strings (i.e., the theory of semigroups). That is, the theory of $+$ and $\times$ over the natural numbers is essentially the same as the theory of a concatenation operator $\smile$ over strings. 

To see the intuition for this, and also to connect this theme with other themes in the present work, consider the laws of concatenation over an alphabet of size one, consisting just of $a$. Let $\varepsilon$ be the empty string. It is easy to check that the following principles are all valid. 
\begin{enumerate}
    \item $\neg x \smile a = \varepsilon$
    \item $x \smile a = y \smile a \rightarrow x=y$
    \item $x \smile \varepsilon = x$
    \item $x \smile (y \smile a) = (x \smile y) \smile a$
    \item Induction: $\varphi(\varepsilon) \rightarrow \forall x (\varphi(x) \rightarrow \varphi(x \smile a)) \rightarrow \varphi(x)$
\end{enumerate}\vspace{0.5ex}
As it happens, interpreting $a$ as $1$, $\varepsilon$ as $0$, and $\smile$ as $+$, these principles completely axiomatize Presburger Arithmetic (they are precisely what you need to run the argument for quantifier elimination), the system we have met so often in this paper under different guises. Intuitively, the laws of addition are just the laws of concatenation for unary notations. What Quine showed is that, perhaps more surprisingly, the correspondence extends to full arithmetic as long as we have at least one more symbol. Similar results have also been shown for second-order number theory and second-order theories of strings (e.g., \citealt{Corcoran}).

More recently, \cite{Grz} has demonstrated that a very weak theory of concatenation can even replace axiomatic theories of arithmetic in the celebrated proof that  ``sufficiently strong'' theories are both undecidable and incomplete. Remarkably, this allows G\"{o}del-style arguments but with no detour through arithmetization of syntax (and thus no use of the Chinese remainder theorem, and so on). 
Later on, \cite{Visser} proved that Grzegorczyk's theory of concatenation is in fact \emph{essentially undecidable} (in the sense of \citealt{Tarski}) by showing it is mutually interpretable with Robinson's Arithmetic. These papers and the ensuing literature contain a wealth of further results on this rich topic, adding yet another dimension to the  interplay between logic and counting.

\end{document}